\pdfoutput=1
\documentclass[11pt]{article}
\usepackage[svgnames]{xcolor}
\usepackage{mathtools}
\usepackage[abbrev,lite,nobysame]{amsrefs}
\usepackage{amsthm,amsfonts,amsmath}
\usepackage{thmtools}
\usepackage{abstract}
\usepackage{hyphenat}
\usepackage{csquotes}
\usepackage{enumerate}
\usepackage{moresize}
\usepackage[font=footnotesize,labelfont=bf]{caption}
\usepackage{subcaption}
\usepackage[colorlinks=true,linkcolor=NavyBlue,
citecolor=NavyBlue,urlcolor=NavyBlue]{hyperref}
\usepackage[titles]{tocloft}

\usepackage{indentfirst}
\usepackage{titlesec}
\titlespacing*{\paragraph}{\parindent}{3.25ex plus 1ex minus .2ex}{5pt plus 1pt minus 1pt}

\usepackage[looser]{newtxtext}
\usepackage[bigdelims,varvw]{newtxmath}
\usepackage[protrusion=true, tracking=true, expansion=true, kerning=true,spacing=true]{microtype}
\usepackage{bm}

\usepackage{cleveref}
\crefname{equation}{}{}

\makeatletter
\AddToHook{cmd/appendix/before}{\def\cref@section@alias{appendix}}
\makeatother

\renewcommand{\leq}{\leqslant}

\renewcommand{\geq}{\geqslant}

\usepackage{tikz}
\usepgflibrary{patterns}
\usetikzlibrary{arrows.meta}

\numberwithin{equation}{section}

\theoremstyle{plain}
\newtheorem{thrm}{Theorem}[section]

\newtheorem{lmm}[thrm]{Lemma}
\newtheorem{crllr}[thrm]{Corollary}

\newtheorem{rmrk}[thrm]{Remark}
\theoremstyle{definition}
\newtheorem{dfntn}[thrm]{Definition}
\newtheorem{xmpl}[thrm]{Example}

\theoremstyle{plain}

\makeatletter
\newcommand*{\toccontents}{\@starttoc{toc}}
\makeatother

\newcommand\tfootnote[1]{%
	\begingroup
	\renewcommand\thefootnote{}\footnote{#1}%
	\addtocounter{footnote}{-1}%
	\endgroup
}

\begin{document}\linespread{1.05}\selectfont
	\date{}
	
	\author{Manuel~Rissel\,\footnote{Institute of Mathematical Sciences, ShanghaiTech University, Shanghai, 201210, China, e-mail: \href{mailto:mrissel@shanghaitech.edu.cn}{mrissel@shanghaitech.edu.cn}} \footnote{NYU-ECNU Institute of Mathematical Sciences at NYU Shanghai, 3663 Zhongshan Road North, Shanghai, 200062, China}}

	\title{Controllability of Boussinesq flows driven by finite-dimensional and physically localized forces}

	\maketitle
	
	\tfootnote{	{\bf MSC2020}: 35Q35, 35Q30, 93C10 (primary); 76B75, 80A19, 93B05 (secondary)}
	\tfootnote{	{\bf Keywords}:  Boussinesq system, incompressible fluids, approximate controllability, finite-dimensional controls, physically localized controls}
	
	\begin{abstract}

		We show approximate controllability of Boussinesq flows in $\mathbb{T}^2 = \mathbb{R}^2 / 2\pi\mathbb{Z}^2$ driven by finite-dimensional controls that are supported in any fixed region~$\omegaup \subset \mathbb{T}^2$. This addresses a Boussinesq version of a question by Agrachev and provides the first known example of incompressible fluids with this property. In this context, we complement results obtained for the Navier--Stokes system by Agrachev--Sarychev (Comm. Math. Phys. 265, 2006), where the controls are finite-dimensional but not localized in physical space, and Nersesyan--Rissel (Comm. Pure Appl. Math. 78, 2025), where physically localized controls admit for special $\omegaup$ a degenerate but not finite-dimensional structure.
		
		For our proof, we study controllability properties of tailored convection equations governed by time-periodic degenerately forced Euler flows that provide a twofold geometric mechanism: transport of information through $\omegaup$ versus non-stationary mixing effects transferring energy from low-dimensional sources to higher frequencies. The temperature is then controlled by using Coron's return method, while the velocity is mainly driven by the buoyant force.

		When $\omegaup$ contains two cuts of $\mathbb{T}^2$, our approach allows to effectively construct low-dimensional control spaces of dimensions that are independent of the choice of $\omegaup$ within this class of control regions.

	\end{abstract}
	\newpage
	\section*{Contents}
	\setcounter{tocdepth}{2}
	\toccontents 
	
	\section{Introduction}

	Let $\mathbb{T}^2=\mathbb{R}^2 / 2\pi\mathbb{Z}^2$ and $\emptyset \neq \omegaup \subset \mathbb{T}^2$ open. The objective of this work is to show that the Boussinesq system propagates degenerate classes of forces, which are finite-dimensional and physically localized in~$\omegaup$, to rich sets in the state space. Such properties are of independent interest and can be expressed using notions from control theory, but they are of relevance also beyond that scope; for instance, in the study of turbulence and other topics in mathematical physics (see~\cite{KuksinShirikyan2012,Agrachev2014,KuksinNersesyanShirikyan2020} and~\Cref{subsection:literature}). 
	Given any $T > 0$, we consider the velocity~$u$, temperature~$\theta$, and pressure~$p$
	describing the motion of an incompressible fluid with thermal buoyancy convection, driven by localized forces that belong at each time to the same finite-dimensional spaces. That is, $u\colon\mathbb{T}^2\times(0,T)\longrightarrow \mathbb{R}^2$ and $\theta, p\colon\mathbb{T}^2\times(0,T) \longrightarrow \mathbb{R}$ are sought to satisfy in $\mathbb{T}^2\times(0,T)$ the initial value problem
	\begin{equation}\label{equation:Boussinesq}
		\begin{gathered}
			\partial_t u - \nu \Delta u + \left(u \cdot \nabla\right) u + \nabla p = \theta e_2 + f + \mathbb{I}_{\omegaup} \xi, \\ \operatorname{div}(u) = 0, \\
			\partial_t \theta - \tau \Delta \theta + (u \cdot \nabla) \theta = g + \mathbb{I}_{\omegaup} \eta,\\
			u(\cdot, 0) = u_0, \quad \theta(
			\cdot, 0) = \theta_0,
		\end{gathered}
	\end{equation}
	where $\nu > 0$ and $\tau > 0$ specify the viscosity and thermal diffusivity,~$\mathbb{I}_{S}$ denotes the indicator function of a set $S$, the unit vector~$e_2 = (0, 1)$ points in the direction of gravity,~$u_0$ and~$\theta_0$ are the initial states,~$f$ and~$g$ are known forces, and~$\xi$ and~$\eta$ are to-be-determined controls of the particular type detailed below in \eqref{equation:controlrepresentation}. For further background on the Boussinesq system in general, which is among other subjects relevant to the study of geophysical phenomena and turbulent flows, we refer to~\cite{Majda2003,ConstantinDoering1999,FoiasManleyTemam1987}.

	As to be made precise in~\Cref{theorem:main2}, the meaning of approximate controllability of \eqref{equation:Boussinesq} can be sketched as follows. For any approximation accuracy $\varepsilon > 0$, time $T > 0$, initial- and target states $(u_0, \theta_0)$ and $(u_1, \theta_1)$, parameters $\nu, \tau > 0$, and forces $(f,g)$, there exist controls $(\xi, \eta)$ such that
	\[
		\|u(\cdot, T) - u_1\| + \|\theta(\cdot, T) - \theta_1\| < \varepsilon,
	\]
	where $\|\cdot\|$ denotes suitable norms. In particular, this is a global (large data) notion of controllability.
	
	To achieve this with finite-dimensional controls means that there are universal numbers $d_1, d_2 \in \mathbb{N}$ and functions $\xi_1, \dots, \xi_{d_1}\colon \mathbb{T}^2\longrightarrow\mathbb{R}^2$ and $\eta_1, \dots, \eta_{d_2}\colon \mathbb{T}^2\longrightarrow\mathbb{R}$, which depend on the fixed choice of $\omegaup$ but are independent of $\nu,\tau,\varepsilon$, and all data in \eqref{equation:Boussinesq}, such that $\xi$ and $\eta$ are for each $t \in [0,T]$ of the form
	\begin{equation}\label{equation:controlrepresentation}
		\begin{gathered}
			\xi(\cdot, t) = \alpha_1(t)\xi_1 + \dots + \alpha_{d_1}(t) \xi_{d_1}, \\
			\eta(\cdot, t) = \beta_1(t)\eta_1 + \dots + \beta_{d_2}(t) \eta_{d_2},
		\end{gathered}
	\end{equation}
	involving the control coefficients
	\[
		\alpha_1, \dots \alpha_{d_1}, \beta_1, \dots, \beta_{d_2} \colon [0,T]\longrightarrow\mathbb{R}.
	\]
	The numbers $d_1, d_2$ and profiles $\xi_1, \dots, \xi_{d_1}, \eta_1, \dots, \eta_{d_2}$ must remain unchanged when varying the viscosity, thermal diffusivity, approximation accuracy, initial- and target states, and prescribed body forces. Only the control coefficients $\alpha_1, \dots \alpha_{d_1}$ and $\beta_1, \dots, \beta_{d_2}$ in the representations of~$\xi$ and~$\eta$ can be chosen in dependence on the data in order to influence the final state of the solution to \eqref{equation:Boussinesq}. This translates to the goal of constructing universal finite-dimensional vector spaces $\mathscr{F}_{\mathscr{v}} \subset C^{\infty}(\mathbb{T}^2; \mathbb{R}^2)$ and $\mathscr{F}_{\mathscr{t}} \subset C^{\infty}(\mathbb{T}^2; \mathbb{R})$ of functions supported in~$\omegaup$ such that~\eqref{equation:Boussinesq} is approximately controllable in time $T > 0$ with controls that satisfy $\xi(\cdot, t) \in \mathscr{F}_{\mathscr{v}}$ and $\eta(\cdot, t) \in \mathscr{F}_{\mathscr{t}}$ for $t \in [0, T]$. 
	
	For the Navier--Stokes system on $\mathbb{T}^2$ driven by finite-dimensional but not physically localized forces, approximate controllability has been shown first in~\cite{AgrachevSarychev2005,AgrachevSarychev2006}. On the other hand, controllability of the Navier--Stokes and Boussinesq systems on $\mathbb{T}^2$ and $\mathbb{T}^3$ driven by physically localized but not finite-dimensional controls is known due to~\cite{CoronFursikov1996,FursikovImanuvilov1999}, where Coron's return method from \cite{Coron96,Coron1996EulerEq} is developed. However, the natural question whether approximate controllability holds with controls that are both finite-dimensional and physically localized constitutes an open problem, posed by Agrachev for the Navier--Stokes system ({\it c.f.}~\cite[Section 7]{Agrachev2014}). Also for other fluid models, corresponding versions of this question have remained unanswered. Here, we give a positive answer for the planar Boussinesq system \eqref{equation:Boussinesq}, which couples the Navier--Stokes equations for the velocity and a convection diffusion equation for the temperature.

	\subsection*{Notation}
	
	\paragraph*{Function spaces.} Given $m \in \mathbb{N}_0 \coloneq \mathbb{N}\cup\{0\}$, and writing $H_{\operatorname{avg}}$ for the $L^2(\mathbb{T}^2;\mathbb{R})$-functions with zero average, we denote the~$L^2$-based Sobolev spaces of~divergence-free vector fields and of zero average scalar functions
	\begin{gather*}
		H \coloneq \left\{ f \in H_{\operatorname{avg}}^2  \, \left| \right. \, \operatorname{div}(f) = 0 \mbox{ in } \mathbb{T}^2 \right\}, \quad
		V^m \coloneq H^m(\mathbb{T}^2;\mathbb{R}^2) \cap H, \\
		H^m \coloneq H^m(\mathbb{T}^2;\mathbb{R}) \cap H_{\operatorname{avg}},
	\end{gather*} 
	where $H^m$ and $V^m$ are endowed with the usual norms $\|\cdot\|_{m}$ of $H^m(\mathbb{T}^2;\mathbb{R})$ and $H^m(\mathbb{T}^2;\mathbb{R}^2)$ respectively. Further, we say that $f \in L^2((0,T); C^{\infty}(\mathbb{T}^2; \mathbb{R}^N))$, with $T > 0$ and $N \in \{1,2\}$, when $f \in L^2((0,T); H^m(\mathbb{T}^2; \mathbb{R}^N))$ for all $m \in \mathbb{N}$. Throughout, the Lebesgue measure is normalized such that $\int_{\mathbb{T}^2} \, dx = 1$.
	
	\paragraph*{Flow maps.} Let $T > 0$ and $v$ be a continuous map $\mathbb{T}^2\times[0,T] \longrightarrow \mathbb{R}^2$ that is Lipschitz continuous in the space variables with time-independent Lipschitz constant. Then, the Cauchy--Lipschitz theorem provides for each $x \in \mathbb{T}^2$ and $s \in [0, T]$ a unique solution $\Phi^v(x,s,\cdot) \colon [0,T] \longrightarrow \mathbb{T}^2$ to the initial value problem
	\begin{equation}\label{equation:flow}
		\begin{gathered}
			\frac{\rm d}{{\rm d}t} \Phi^v (x,s,t) = v(\Phi^v(x,s,t), t), \quad
			\Phi^v (x,s,s) = x.
		\end{gathered}
	\end{equation}
	We call $\Phi^v$ the flow of~$v$ and note that $\Phi^v(\Phi^v(x,s,r),r,t) = \Phi^v(x,s,t)$
	for all $x \in \mathbb{T}^2$ and $r,s,t \in [0,T]$. When $A \subset \mathbb{T}^2$ and $I, J \subset [0,T]$, we write $\Phi^v(A, I, J) \coloneq \{ \Phi^v(x, s, t) \, | \, x \in A, \, s \in I, \, t \in J \}$.
	
	\paragraph*{Div-curl problems.}For sufficiently regular~$U = (U_1,U_2)\colon\mathbb{T}^2 \longrightarrow \mathbb{R}^2$, the \enquote{curl} of~$U$ is defined as $\nabla \wedge U \coloneq \partial_1 U_2 - \partial_2 U_1$. Moreover, for $z \in H^m$, $m \in \mathbb{N}_0$, and $A \in \mathbb{R}^2$, we denote by $\Upsilon(z, A) \in H^{m+1}(\mathbb{T}^2;\mathbb{R}^2)$ the unique solution to the div-curl problem
	\begin{equation}\label{equation:Upsilon}
		\begin{gathered}
			\nabla \cdot \Upsilon(z, A) = 0, \quad 
			\nabla \wedge \Upsilon(z, A) = z
		\end{gathered}
	\end{equation}
	that satisfies $\int_{\mathbb{T}^2} \Upsilon(z, A)(x) \, dx = A$.
	One can express $\Upsilon(z, A) = \nabla^{\perp} \phi + A$, where the stream function $\phi$ solves Poisson's equation $\Delta \phi = - z$ in $\mathbb{T}^2$ and $\nabla^{\perp} \phi \coloneq (\partial_2 \phi, -\partial_1 \phi)$. When $A = 0$, we abbreviate $\Upsilon(z) = \Upsilon(z, 0)$.

	\subsection{Main results}\label{subsection:mainresults}
	The following Theorem is our main contribution. It provides quick approximate controllability for the temperature, while keeping the velocity close to its initial state. At this point, the data is assumed more regular than necessary; several assumptions will be relaxed subsequently due to the parabolic smoothing effects exhibited by~\eqref{equation:Boussinesq}. The full approximate controllability of the Boussinesq system is stated below in \Cref{theorem:main2}.

	\begin{thrm}\label{theorem:main}
		There are finite-dimensional spaces~$\mathscr{F}_{\mathscr{v}} \subset C^{\infty}(\mathbb{T}^2; \mathbb{R}^2)$ and~$\mathscr{F}_{\mathscr{t}} \subset C^{\infty}(\mathbb{T}^2; \mathbb{R})\cap H_{\operatorname{avg}}$ such that the following statement holds. For any given data
		\begin{gather*}
			\nu, \tau,  \varepsilon,  T > 0, \quad m \in \mathbb{N}, \quad
			u_0 \in V^{m+2}, \quad \theta_0, \theta_1 \in H^{m+2}, \\
			f \in L^2((0,T);V^{m}), \quad g \in L^2((0,T);H^{m}),
		\end{gather*}
		there exists~$\delta_0 > 0$ so that for each $\delta \in (0, \delta_0)$ there are $\xi \in L^2((0,\delta); \mathscr{F}_{\mathscr{v}})$ and $\eta \in L^2((0,\delta); \mathscr{F}_{\mathscr{t}})$ 
		for which the associated solution
		\[
		(u, \theta) \in C^0([0,\delta];H^{m}(\mathbb{T}^2;\mathbb{R}^2) \times H^m)\cap L^2((0,\delta);H^{m+1}(\mathbb{T}^2;\mathbb{R}^2) \times H^{m+1})
		\]
		to the Boussinesq problem \eqref{equation:Boussinesq} satisfies
		\[
		\|u(\cdot, \delta) - u_0\|_{m+1} + \|\theta(\cdot, \delta) - \theta_1\|_{m+1} < \varepsilon.
		\]
	\end{thrm}
	The proof of \Cref{theorem:main} is organized as follows. In \Cref{section:knownfinre}, a controllability result for transport problems with generating drift is recalled. In \Cref{section:lct}, approximate controllability via finite-dimensional and physically localized forces is established for a specially constructed convection problem. The argument is completed in \Cref{section:conclusion}.
	
	As a corollary of \Cref{theorem:main}, we can conclude also the approximate controllability of both the temperature and the velocity in arbitrary time, and for less regular initial states. 
	
	\begin{crllr}\label{theorem:main2}
		Let the spaces~$\mathscr{F}_{\mathscr{v}} \subset C^{\infty}(\mathbb{T}^2; \mathbb{R}^2)$ and~$\mathscr{F}_{\mathscr{t}} \subset C^{\infty}(\mathbb{T}^2; \mathbb{R})\cap H_{\operatorname{avg}}$ be obtained via \Cref{theorem:main}. For any given $\varepsilon, \nu, \tau, T > 0$, $k \in \mathbb{N}_0$, $u_0 \in H$, $u_1 \in V^k$, $\theta_0 \in H_{\operatorname{avg}}$, $\theta_1 \in H^{k}$, $f \in L^2((0,T);V^{\max\{k-1,1\}})$, and $g \in L^2((0,T);H^{\max\{k-1,1\}})$,
		there exist controls $\xi \in L^2((0,T); \mathscr{F}_{\mathscr{v}})$ and $\eta \in L^2((0,T); \mathscr{F}_{\mathscr{t}})$ such that the solution
		\begin{gather*}
			u \in C^0((0,T];H^{\max\{k,2\}}(\mathbb{T}^2;\mathbb{R}^2))\cap L^2((0,T);H^{\max\{k+1,3\}}(\mathbb{T}^2;\mathbb{R}^2)),\\
			\theta \in C^0((0,T];H^{\max\{k,2\}})\cap L^2((0,T);H^{\max\{k+1,3\}})
		\end{gather*}
		to the Boussinesq problem \eqref{equation:Boussinesq}
		satisfies
		\[
		\|u(\cdot, T) - u_1\|_{k} + \|\theta(\cdot, T) - \theta_1\|_{k} < \varepsilon.
		\]
	\end{crllr}

	A sketch of the proof of \Cref{theorem:main2}, which will be presented in more detail in \Cref{subsection:prfmain2}, is as follows. 
	
	1) Let $m \in \mathbb{N}$, $f \in L^2((0,T);V^{m})$, and $g \in L^2((0,T);H^{m})$. By the well-posedness of the $2$D Boussinesq system, one can choose $\sigma > 0$ so small that, if an uncontrolled solution $(u, \theta)$ to \eqref{equation:Boussinesq} is issued at $t = t_0$ from the $\varepsilon/2$-neighborhood of $(u_1, \theta_1)$ in $V^{m+1} \times H^{m+1}$, then $(u, \theta)(\cdot, t)$ remains in the $\varepsilon$-neighborhood of $(u_1, \theta_1)$ in $H^{m+1}(\mathbb{T}^2;\mathbb{R}^2) \times H^{m+1}(\mathbb{T}^2;\mathbb{R})$ for all $t \in [t_0, t_0+\sigma]$. Thus, to control the system in any given time~$T > 0$ and with less regular initial states, one can first issue a trajectory of \eqref{equation:Boussinesq} in the Leray-Hopf weak sense from $(u_0, \theta_0) \in H\times H_{\operatorname{avg}}$ at time $t = 0$ with zero controls ($\xi = 0$, $\eta = 0$). By parabolic regularization effects, available due to the choice of forces $f \in L^2((0,T);V^{m})$ and $g \in L^2((0,T);H^{m})$, this trajectory will belong to $V^{m+2}\times H^{m+2}$ at almost all times $t > 0$. Then, starting from $t = T - \sigma$ one employs the actual control strategy; we refer also to the similar situations in \cite{Nersesyan2021,NersesyanRissel2024a,NersesyanRissel2024}.
	
	2) To steer also the velocity approximately to any given target, and not only the temperature, a mechanism from \cite{NersesyanRissel2024} can be applied, relying on several scaling limits and the fact that the set $\mathscr{E}$, defined as
	\begin{equation}\label{equation:E}
		\begin{gathered}
			\mathscr{E} \coloneq \left\{ q_0 + \left(\Upsilon(q_1) \cdot \nabla\right) q_1 + \left(\Upsilon(q_2) \cdot \nabla\right) q_2 \, \, | \, \, q_0,q_1,q_2\in \operatorname{span}_{\mathbb{R}}\mathscr{E}_0 \right\},\\
			\mathscr{E}_0 \coloneq \left\{\sin(x\cdot n), \, \cos(x \cdot n) \, \left| \right. \, n \in \mathbb{N}\times\mathbb{N}_0 \right\},
		\end{gathered}
	\end{equation}
	contains $\pm\sin(x\cdot n)$ and $\pm\cos(x\cdot n)$ for all $n\in \mathbb{Z}^2\setminus\{0\}$; {\it c.f.}~\cite{AgrachevSarychev2006} and also \cite[Lemma 3.5]{NersesyanRissel2024}.
	More precisely, given any~$q \in C^{\infty}(\mathbb{T}^2;\mathbb{R})$ with zero average, it is shown in \cite[Theorem 3.4]{NersesyanRissel2024} that
	\[
	\nabla \wedge u_{\delta}(\cdot, \delta) \longrightarrow \nabla \wedge u_0 - \partial_1 q \, \mbox{ in } \, H^{m} \, \mbox{ as } \, \delta \longrightarrow 0,
	\]
	where $(u_{\delta}, \theta_{\delta})$ solves \eqref{equation:Boussinesq} with zero controls~$(\xi, \eta) = (0,0)$, initial velocity~$u_0 \in H^{m+2}$, and initial temperature of the form $\theta_0 = - \delta^{-1} q$. In addition, the latter reference provides in $H^{m}\times H^{m+1}$ the convergence
	\[
	(\nabla \wedge  u_{\delta}, \theta_{\delta})(\cdot, \delta) - (\delta^{-1/2}q, 0) \longrightarrow (\widetilde{w}_0 - (\Upsilon(q) \cdot \nabla) q, \theta_0)  \, \mbox{ as } \, \delta \longrightarrow 0,
	\]
	where $(u_{\delta}, \theta_{\delta})$ is the solution to \eqref{equation:Boussinesq} 
	with zero controls~$(\xi, \eta) = (0,0)$, initial temperature $\theta_0 \in H^{m+2}$, and initial vorticity $\nabla \wedge u_0 = \widetilde{w}_0 + \delta^{-1/2} q$ for a given $\widetilde{w}_0 \in H^{m+1}$. Combining iterations of these two convergence results and \Cref{theorem:main}, one can steer $\nabla \wedge u$ arbitrarily fast, and as close as desired in $H^m$, to any finite sum of the form
	\[
		\widetilde{w}_0 - q_0 - \sum_{i=1}^{2N}(\Upsilon(q_i) \cdot \nabla) q_i,
	\]
	where $N \in \mathbb{N}$ and $q_0, q_1, \dots, q_{2N} \in \operatorname{span}_{\mathbb{R}}\mathscr{E}_0$. 
	Owing to the form of $\mathscr{E}$, this implies approximate controllability for the vorticity. As both convergences from \cite[Theorem 3.4]{NersesyanRissel2024} are uniform with respect to~$(f, g)$ from bounded subsets of $L^2((0, T); V^{m}\times H^m)$, one can as in \cite{NersesyanRissel2024} define in a piece-wise (in time) manner a suitably controlled trajectory. Due to \Cref{theorem:main}, the resulting controls are here finite-dimensional and physically localized.

	\begin{rmrk}\label{remark:nza}
		Our approach also works for initial- and target states of non-zero average, possibly requiring the addition of a two-dimensional space to $\mathscr{F}_{\mathscr{v}}$ and of a one-dimensional space to $\mathscr{F}_{\mathscr{t}}$. This is explained in \Cref{remark:nza2} in \Cref{subsection:prfmain2}. Velocity controls cannot be divergence-free in general, as this function class may leave the velocity average invariant. For instance, when~$\omegaup$ is simply-connected, Poincar\'e's lemma yields for divergence-free $\xi \in L^2(\mathbb{T}^2;\mathbb{R}^2)$ with $\operatorname{supp}(\xi) \subset \omegaup$ the existence of $\phi \in H^1(\mathbb{T}^2;\mathbb{R})$ such that $\xi = \nabla^{\perp} \phi$. Thus, if such $\xi$ is plugged into~\eqref{equation:Boussinesq}, one finds that
		\begin{equation*}
			\frac{d}{dt} \int_{\mathbb{T}^2} u \cdot e_1 \, dx =  \int_{\mathbb{T}^2} (\nu \Delta u \cdot e_1 - \left(u \cdot \nabla\right) u \cdot e_1 - \partial_1 p + \partial_2 \phi) \, dx = 0,
		\end{equation*}
		obstructing approximate controllability of~\eqref{equation:Boussinesq} when $\int_{\mathbb{T}^2} u_0 \cdot e_1 \, dx \neq \int_{\mathbb{T}^2} u_1 \cdot e_1 \, dx$ for $e_1 = (1,0)$.
		The controls obtained here are not divergence-free, which is in alignment with the existing literature on controllability properties of incompressible fluids driven by physically localized forces; see also the references in \Cref{subsection:literature}.
	\end{rmrk}

	\begin{rmrk}
		The controls in \Cref{theorem:main} and~\Cref{theorem:main2} can be chosen smooth in time, using a density argument and the stability of solutions to the Boussinesq system with respect to small perturbations of the forces.
	\end{rmrk}

	\subsection{Description of the approach}\label{subsection:appraoch}
	
	We start with a so-called \enquote{generating} vector field~$\overline{u}^{\star}$ that has small uniform norm (depending only on~$\omegaup$) and is constructed from an observable family as described in \Cref{definition:ovf}. This notion of observability, first introduced in \cite{KuksinNersesyanShirikyan2020} for the study of randomly forced PDEs, induces a certain type of non-stationary mixing effect, propagating energy added to the system by low-dimensional forces to higher frequencies. Choosing~$\overline{u}^{\star}$ of small norm ensures that its flow cannot transport information in a fixed time over large distances, which will be crucial for the definition of our localized controls. Then, having a linearized and inviscid version of the temperature equation from~\eqref{equation:Boussinesq} in mind, we consider on a small time interval $[0,T^{\star}]$ the transport problem
	\begin{equation}\label{equation:introlg}
		\partial_t v + (\overline{u}^{\star}\cdot \nabla)v = g^{\star}
	\end{equation}
	for which approximate controllability by means of low-dimensional controls~$g^{\star}$ without physical localization is known; since $g^{\star}$ can act everywhere in $\mathbb{T}^2$, the small uniform norm of $\overline{u}^{\star}$ is not an obstruction.
	Based on this, we construct a universal vector field $\overline{U}$, depending only on~$\omegaup$, such that approximate controllability also holds for
	\begin{equation}\label{equation:introtrp}
		\partial_t V + (\overline{U}\cdot \nabla)V = \mathbb{I}_{\omegaup}G
	\end{equation}
	with a finite-dimensional control $G$. Up to a few technical details omitted at this point, the force $G$ will on the reference time interval $[0,1]$ be given by
	\begin{equation}\label{intro:control}
		G(x,t) = \mu(x) \sum_{i=1}^M  \mathbb{I}_{[t^i_a, t^i_b]}(t) g^{\star}(x-S_i, t-t^i_a),
	\end{equation}
	where $g^{\star}$ is a low-dimensional control for \eqref{equation:introlg}, $M$ is a number depending only on the geometry of the control region, $\mu$ is a particular cutoff supported in~$\omegaup$, $[t^i_a, t^i_b]\subset(0,1)$ are disjoint time intervals of length $T^{\star}$ on which $\overline{U}$ will represent a physically localized version of $\overline{u}^{\star}(\cdot-S_i, \cdot - t^i_a)$, and $S_i \in \mathbb{R}^2$ are fixed shifts related to the convection mechanism that $\overline{U}$ provides for times outside~$[t^i_a, t^i_b]$. See also \Cref{Figure:Pattern}.
	
	The idea is, by a careful construction of~$\overline{U}$ involving some geometric considerations, to achieve the rearrangement (see \Cref{theorem:locfinthm})
	\begin{equation*}
		\int_0^{T^{\star}} g^{\star}(\Phi^{\overline{u}^{\star}}(x,T^{\star}, s), s) \, ds = \int_0^{1} G(\Phi^{\overline{U}}(x,1, s), s) \, ds,
	\end{equation*}
	where $\Phi^{\overline{u}^{\star}}$ and $\Phi^{\overline{U}}$ are the flows of the vector fields $\overline{u}^{\star}$ and $\overline{U}$, respectively. 
	\begin{figure}[ht!]
		\centering
		\resizebox{0.8\textwidth}{!}{
			\begin{tikzpicture}
				\clip(3.9,-36.1) rectangle (46.1,10.3);
				
				\coordinate[label=center:\HUGE{$\overline{u}^{\star}$ on $[0, T^{\star}]$}] (A) at (4.5+15.25,9.7);
				\coordinate[label=center:\HUGE{$g^{\star}$ on $[0, T^{\star}]$}] (B) at (14.5+15.25,9.7);
				
				\draw[line width=0.3mm, color=black] plot[smooth cycle] (0+15.25,0) rectangle+ (3,3);
				\draw[line width=0.3mm, color=black] plot[smooth cycle] (3+15.25,0) rectangle+ (3,3);
				\draw[line width=0.3mm, color=black] plot[smooth cycle] (6+15.25,0) rectangle+ (3,3);
				\draw[line width=0.3mm, color=black] plot[smooth cycle] (0+15.25,3) rectangle+ (3,3);
				
				\draw[line width=0.3mm, color=black] plot[smooth cycle] (6+15.25,3) rectangle+ (3,3);
				\draw[line width=0.3mm, color=black] plot[smooth cycle] (0+15.25,6) rectangle+ (3,3);
				\draw[line width=0.3mm, color=black] plot[smooth cycle] (3+15.25,6) rectangle+ (3,3);
				\draw[line width=0.3mm, color=black] plot[smooth cycle] (6+15.25,6) rectangle+ (3,3);
				
				\draw[line width=1.5mm, color=black] plot[smooth cycle] (3+15.25,3) rectangle+ (3,3);
				\node[font=\HUGE, anchor=center, text=DarkBlue] at (1.5+15.25,1.5) {\bf A};
				\node[font=\HUGE, anchor=center, text=DarkBlue] at (4.5+15.25,1.5) {\bf B};
				\node[font=\HUGE, anchor=center, text=DarkBlue] at (7.5+15.25,1.5) {\bf C};
				\node[font=\HUGE, anchor=center, text=DarkBlue] at (1.5+15.25,4.5) {\bf D};
				\node[font=\HUGE, anchor=center, text=DarkBlue] at (4.5+15.25,4.5) {\bf E};
				\node[font=\HUGE, anchor=center, text=DarkBlue] at (7.5+15.25,4.5) {\bf F};
				\node[font=\HUGE, anchor=center, text=DarkBlue] at (1.5+15.25,7.5) {\bf G};
				\node[font=\HUGE, anchor=center, text=DarkBlue] at (4.5+15.25,7.5) {\bf H};
				\node[font=\HUGE, anchor=center, text=DarkBlue] at (7.5+15.25,7.5) {\bf I};
				
				\draw[line width=0.3mm, color=black] plot[smooth cycle] (10+15.25,0) rectangle+ (3,3);
				\draw[line width=0.3mm, color=black] plot[smooth cycle] (13+15.25,0) rectangle+ (3,3);
				\draw[line width=0.3mm, color=black] plot[smooth cycle] (16+15.25,0) rectangle+ (3,3);
				\draw[line width=0.3mm, color=black] plot[smooth cycle] (10+15.25,3) rectangle+ (3,3);
				
				\draw[line width=0.3mm, color=black] plot[smooth cycle] (16+15.25,3) rectangle+ (3,3);
				\draw[line width=0.3mm, color=black] plot[smooth cycle] (10+15.25,6) rectangle+ (3,3);
				\draw[line width=0.3mm, color=black] plot[smooth cycle] (13+15.25,6) rectangle+ (3,3);
				\draw[line width=0.3mm, color=black] plot[smooth cycle] (16+15.25,6) rectangle+ (3,3);
				
				\draw[line width=1.5mm, color=black] plot[smooth cycle] (13+15.25,3) rectangle+ (3,3);
				
				\node[font=\HUGE, anchor=center, text=DarkBlue] at (11.5+15.25,1.5) {$\mathbf 1$};
				\node[font=\HUGE, anchor=center, text=DarkBlue] at (14.5+15.25,1.5) {$\mathbf 2$};
				\node[font=\HUGE, anchor=center, text=DarkBlue] at (17.5+15.25,1.5) {$\mathbf 3$};
				\node[font=\HUGE, anchor=center, text=DarkBlue] at (11.5+15.25,4.5) {$\mathbf 4$};
				\node[font=\HUGE, anchor=center, text=DarkBlue] at (14.5+15.25,4.5) {$\mathbf 5$};
				\node[font=\HUGE, anchor=center, text=DarkBlue] at (17.5+15.25,4.5) {$\mathbf 6$};
				\node[font=\HUGE, anchor=center, text=DarkBlue] at (11.5+15.25,7.5) {$\mathbf 7$};
				\node[font=\HUGE, anchor=center, text=DarkBlue] at (14.5+15.25,7.5) {$\mathbf 8$};
				\node[font=\HUGE, anchor=center, text=DarkBlue] at (17.5+15.25,7.5) {$\mathbf 9$};

				\coordinate[label=center:\HUGE{$\overline{U}$ on $[0, t^1_a]$}] (C) at (4.5+4,9.7-11);
				\coordinate[label=center:\HUGE{$G$ on $[0, t^1_a]$}] (D) at (14.5+4,9.7-11);
				
				\draw[line width=0.3mm, color=black] plot[smooth cycle] (0+4,-11) rectangle+ (3,3);
				\draw[line width=0.3mm, color=black] plot[smooth cycle] (3+4,-11) rectangle+ (3,3);
				\draw[line width=0.3mm, color=black] plot[smooth cycle] (6+4,-11) rectangle+ (3,3);
				\draw[line width=0.3mm, color=black] plot[smooth cycle] (0+4,3-11) rectangle+ (3,3);
				
				\draw[line width=0.3mm, color=black] plot[smooth cycle] (6+4,3-11) rectangle+ (3,3);
				\draw[line width=0.3mm, color=black] plot[smooth cycle] (0+4,6-11) rectangle+ (3,3);
				\draw[line width=0.3mm, color=black] plot[smooth cycle] (3+4,6-11) rectangle+ (3,3);
				\draw[line width=0.3mm, color=black] plot[smooth cycle] (6+4,6-11) rectangle+ (3,3);
				
				\draw[line width=1.5mm, color=black] plot[smooth cycle] (3+4,3-11) rectangle+ (3,3);
				
				\draw[line width=1.5mm,-{Stealth[length=5mm, width=7mm]}] (1.5+4,1.5-11) -- (4.5+4,4.5-11);

				\draw[line width=0.3mm, color=black] plot[smooth cycle] (10+4,0-11) rectangle+ (3,3);
				\draw[line width=0.3mm, color=black] plot[smooth cycle] (13+4,0-11) rectangle+ (3,3);
				\draw[line width=0.3mm, color=black] plot[smooth cycle] (16+4,0-11) rectangle+ (3,3);
				\draw[line width=0.3mm, color=black] plot[smooth cycle] (10+4,3-11) rectangle+ (3,3);
				
				\draw[line width=0.3mm, color=black] plot[smooth cycle] (16+4,3-11) rectangle+ (3,3);
				\draw[line width=0.3mm, color=black] plot[smooth cycle] (10+4,6-11) rectangle+ (3,3);
				\draw[line width=0.3mm, color=black] plot[smooth cycle] (13+4,6-11) rectangle+ (3,3);
				\draw[line width=0.3mm, color=black] plot[smooth cycle] (16+4,6-11) rectangle+ (3,3);
				
				\draw[line width=1.5mm, color=black] plot[smooth cycle] (13+4,3-11) rectangle+ (3,3);
				
				\node[font=\HUGE, anchor=center, text=DarkBlue] at (11.5+4,1.5-11) {$-$};
				\node[font=\HUGE, anchor=center, text=DarkBlue] at (14.5+4,1.5-11) {$-$};
				\node[font=\HUGE, anchor=center, text=DarkBlue] at (17.5+4,1.5-11) {$-$};
				\node[font=\HUGE, anchor=center, text=DarkBlue] at (11.5+4,4.5-11) {$-$};
				\node[font=\HUGE, anchor=center, text=DarkBlue] at (14.5+4,4.5-11) {$-$};
				\node[font=\HUGE, anchor=center, text=DarkBlue] at (17.5+4,4.5-11) {$-$};
				\node[font=\HUGE, anchor=center, text=DarkBlue] at (11.5+4,7.5-11) {$-$};
				\node[font=\HUGE, anchor=center, text=DarkBlue] at (14.5+4,7.5-11) {$-$};
				\node[font=\HUGE, anchor=center, text=DarkBlue] at (17.5+4,7.5-11) {$-$};

				\coordinate[label=center:\HUGE{$\overline{U}$ on $[t^{1}_a, t^{1}_b]$}] (E) at (4.5-4+30.5,9.7-11);
				\coordinate[label=center:\HUGE{$G$ on $[t^{1}_a, t^{1}_b]$}] (F) at (14.5-4+30.5,9.7-11);
				\draw[line width=0.3mm, color=black] plot[smooth cycle] (0-4+30.5,0-11) rectangle+ (3,3);
				\draw[line width=0.3mm, color=black] plot[smooth cycle] (3-4+30.5,0-11) rectangle+ (3,3);
				\draw[line width=0.3mm, color=black] plot[smooth cycle] (6-4+30.5,0-11) rectangle+ (3,3);
				\draw[line width=0.3mm, color=black] plot[smooth cycle] (0-4+30.5,3-11) rectangle+ (3,3);
				
				\draw[line width=0.3mm, color=black] plot[smooth cycle] (6-4+30.5,3-11) rectangle+ (3,3);
				\draw[line width=0.3mm, color=black] plot[smooth cycle] (0-4+30.5,6-11) rectangle+ (3,3);
				\draw[line width=0.3mm, color=black] plot[smooth cycle] (3-4+30.5,6-11) rectangle+ (3,3);
				\draw[line width=0.3mm, color=black] plot[smooth cycle] (6-4+30.5,6-11) rectangle+ (3,3);
				
				\draw[line width=1.5mm, color=black] plot[smooth cycle] (3-4+30.5,3-11) rectangle+ (3,3);
				
				\node[font=\HUGE, anchor=center, text=DarkBlue] at (1.5-4+30.5,1.5-11) {$-$};
				\node[font=\HUGE, anchor=center, text=DarkBlue] at (4.5-4+30.5,1.5-11) {$-$};
				\node[font=\HUGE, anchor=center, text=DarkBlue] at (7.5-4+30.5,1.5-11) {$-$};
				\node[font=\HUGE, anchor=center, text=DarkBlue] at (1.5-4+30.5,4.5-11) {$-$};
				\node[font=\HUGE, anchor=center, text=DarkBlue] at (4.5-4+30.5,4.5-11) {\bf A};
				\node[font=\HUGE, anchor=center, text=DarkBlue] at (7.5-4+30.5,4.5-11) {$-$};
				\node[font=\HUGE, anchor=center, text=DarkBlue] at (1.5-4+30.5,7.5-11) {$-$};
				\node[font=\HUGE, anchor=center, text=DarkBlue] at (4.5-4+30.5,7.5-11) {$-$};
				\node[font=\HUGE, anchor=center, text=DarkBlue] at (7.5-4+30.5,7.5-11) {$-$};
				
				\draw[line width=0.3mm, color=black] plot[smooth cycle] (10-4+30.5,0-11) rectangle+ (3,3);
				\draw[line width=0.3mm, color=black] plot[smooth cycle] (13-4+30.5,0-11) rectangle+ (3,3);
				\draw[line width=0.3mm, color=black] plot[smooth cycle] (16-4+30.5,0-11) rectangle+ (3,3);
				\draw[line width=0.3mm, color=black] plot[smooth cycle] (10-4+30.5,3-11) rectangle+ (3,3);
				
				\draw[line width=0.3mm, color=black] plot[smooth cycle] (16-4+30.5,3-11) rectangle+ (3,3);
				\draw[line width=0.3mm, color=black] plot[smooth cycle] (10-4+30.5,6-11) rectangle+ (3,3);
				\draw[line width=0.3mm, color=black] plot[smooth cycle] (13-4+30.5,6-11) rectangle+ (3,3);
				\draw[line width=0.3mm, color=black] plot[smooth cycle] (16-4+30.5,6-11) rectangle+ (3,3);
				
				\draw[line width=1.5mm, color=black] plot[smooth cycle] (13-4+30.5,3-11) rectangle+ (3,3);
				
				\node[font=\HUGE, anchor=center, text=DarkBlue] at (11.5-4+30.5,1.5-11) {$-$};
				\node[font=\HUGE, anchor=center, text=DarkBlue] at (14.5-4+30.5,1.5-11) {$-$};
				\node[font=\HUGE, anchor=center, text=DarkBlue] at (17.5-4+30.5,1.5-11) {$-$};
				\node[font=\HUGE, anchor=center, text=DarkBlue] at (11.5-4+30.5,4.5-11) {$-$};
				\node[font=\HUGE, anchor=center, text=DarkBlue] at (14.5-4+30.5,4.5-11) {$\mathbf 1$};
				\node[font=\HUGE, anchor=center, text=DarkBlue] at (17.5-4+30.5,4.5-11) {$-$};
				\node[font=\HUGE, anchor=center, text=DarkBlue] at (11.5-4+30.5,7.5-11) {$-$};
				\node[font=\HUGE, anchor=center, text=DarkBlue] at (14.5-4+30.5,7.5-11) {$-$};
				\node[font=\HUGE, anchor=center, text=DarkBlue] at (17.5-4+30.5,7.5-11) {$-$};

				\coordinate[label=center:\HUGE{$\overline{U}$ on $[t^{1}_b, t^{2}_a]$}] (G) at (4.5+4,9.7-22);
				\coordinate[label=center:\HUGE{$G$ on $[t^{1}_b, t^{2}_a]$}] (H) at (14.5+4,9.7-22);
				
				\draw[line width=0.3mm, color=black] plot[smooth cycle] (0+4,-22) rectangle+ (3,3);
				\draw[line width=0.3mm, color=black] plot[smooth cycle] (3+4,-22) rectangle+ (3,3);
				\draw[line width=0.3mm, color=black] plot[smooth cycle] (6+4,-22) rectangle+ (3,3);
				\draw[line width=0.3mm, color=black] plot[smooth cycle] (0+4,3-22) rectangle+ (3,3);
				\draw[line width=0.3mm, color=black] plot[smooth cycle] (6+4,3-22) rectangle+ (3,3);
				\draw[line width=0.3mm, color=black] plot[smooth cycle] (0+4,6-22) rectangle+ (3,3);
				\draw[line width=0.3mm, color=black] plot[smooth cycle] (3+4,6-22) rectangle+ (3,3);
				\draw[line width=0.3mm, color=black] plot[smooth cycle] (6+4,6-22) rectangle+ (3,3);
				
				\draw[line width=1.5mm, color=black] plot[smooth cycle] (3+4,3-22) rectangle+ (3,3);
				
				\draw[line width=1.5mm,-{Stealth[length=5mm, width=7mm]}]  (4.5+4,1.5-22) -- (4.5+4,4.5-22);

				\draw[line width=0.3mm, color=black] plot[smooth cycle] (10+4,0-22) rectangle+ (3,3);
				\draw[line width=0.3mm, color=black] plot[smooth cycle] (13+4,0-22) rectangle+ (3,3);
				\draw[line width=0.3mm, color=black] plot[smooth cycle] (16+4,0-22) rectangle+ (3,3);
				\draw[line width=0.3mm, color=black] plot[smooth cycle] (10+4,3-22) rectangle+ (3,3);
				
				\draw[line width=0.3mm, color=black] plot[smooth cycle] (16+4,3-22) rectangle+ (3,3);
				\draw[line width=0.3mm, color=black] plot[smooth cycle] (10+4,6-22) rectangle+ (3,3);
				\draw[line width=0.3mm, color=black] plot[smooth cycle] (13+4,6-22) rectangle+ (3,3);
				\draw[line width=0.3mm, color=black] plot[smooth cycle] (16+4,6-22) rectangle+ (3,3);
				
				\draw[line width=1.5mm, color=black] plot[smooth cycle] (13+4,3-22) rectangle+ (3,3);
				
				\node[font=\HUGE, anchor=center, text=DarkBlue] at (11.5+4,1.5-22) {$-$};
				\node[font=\HUGE, anchor=center, text=DarkBlue] at (14.5+4,1.5-22) {$-$};
				\node[font=\HUGE, anchor=center, text=DarkBlue] at (17.5+4,1.5-22) {$-$};
				\node[font=\HUGE, anchor=center, text=DarkBlue] at (11.5+4,4.5-22) {$-$};
				\node[font=\HUGE, anchor=center, text=DarkBlue] at (14.5+4,4.5-22) {$-$};
				\node[font=\HUGE, anchor=center, text=DarkBlue] at (17.5+4,4.5-22) {$-$};
				\node[font=\HUGE, anchor=center, text=DarkBlue] at (11.5+4,7.5-22) {$-$};
				\node[font=\HUGE, anchor=center, text=DarkBlue] at (14.5+4,7.5-22) {$-$};
				\node[font=\HUGE, anchor=center, text=DarkBlue] at (17.5+4,7.5-22) {$-$};

				\coordinate[label=center:\HUGE{$\overline{U}$ on $[t^{2}_a, t^{2}_b]$}] (I) at (4.5-4+30.5,9.7-22);
				\coordinate[label=center:\HUGE{$G$ on $[t^{2}_a, t^{2}_b]$}] (J) at (14.5-4+30.5,9.7-22);
				\draw[line width=0.3mm, color=black] plot[smooth cycle] (0-4+30.5,0-22) rectangle+ (3,3);
				\draw[line width=0.3mm, color=black] plot[smooth cycle] (3-4+30.5,0-22) rectangle+ (3,3);
				\draw[line width=0.3mm, color=black] plot[smooth cycle] (6-4+30.5,0-22) rectangle+ (3,3);
				\draw[line width=0.3mm, color=black] plot[smooth cycle] (0-4+30.5,3-22) rectangle+ (3,3);
				
				\draw[line width=0.3mm, color=black] plot[smooth cycle] (6-4+30.5,3-22) rectangle+ (3,3);
				\draw[line width=0.3mm, color=black] plot[smooth cycle] (0-4+30.5,6-22) rectangle+ (3,3);
				\draw[line width=0.3mm, color=black] plot[smooth cycle] (3-4+30.5,6-22) rectangle+ (3,3);
				\draw[line width=0.3mm, color=black] plot[smooth cycle] (6-4+30.5,6-22) rectangle+ (3,3);
				
				\draw[line width=1.5mm, color=black] plot[smooth cycle] (3-4+30.5,3-22) rectangle+ (3,3);
				\node[font=\HUGE, anchor=center, text=DarkBlue] at (1.5-4+30.5,1.5-22) {$-$};
				\node[font=\HUGE, anchor=center, text=DarkBlue] at (4.5-4+30.5,1.5-22) {$-$};
				\node[font=\HUGE, anchor=center, text=DarkBlue] at (7.5-4+30.5,1.5-22) {$-$};
				\node[font=\HUGE, anchor=center, text=DarkBlue] at (1.5-4+30.5,4.5-22) {$-$};
				\node[font=\HUGE, anchor=center, text=DarkBlue] at (4.5-4+30.5,4.5-22) {\bf B};
				\node[font=\HUGE, anchor=center, text=DarkBlue] at (7.5-4+30.5,4.5-22) {$-$};
				\node[font=\HUGE, anchor=center, text=DarkBlue] at (1.5-4+30.5,7.5-22) {$-$};
				\node[font=\HUGE, anchor=center, text=DarkBlue] at (4.5-4+30.5,7.5-22) {$-$};
				\node[font=\HUGE, anchor=center, text=DarkBlue] at (7.5-4+30.5,7.5-22) {$-$};
				
				\draw[line width=0.3mm, color=black] plot[smooth cycle] (10-4+30.5,0-22) rectangle+ (3,3);
				\draw[line width=0.3mm, color=black] plot[smooth cycle] (13-4+30.5,0-22) rectangle+ (3,3);
				\draw[line width=0.3mm, color=black] plot[smooth cycle] (16-4+30.5,0-22) rectangle+ (3,3);
				\draw[line width=0.3mm, color=black] plot[smooth cycle] (10-4+30.5,3-22) rectangle+ (3,3);
				
				\draw[line width=0.3mm, color=black] plot[smooth cycle] (16-4+30.5,3-22) rectangle+ (3,3);
				\draw[line width=0.3mm, color=black] plot[smooth cycle] (10-4+30.5,6-22) rectangle+ (3,3);
				\draw[line width=0.3mm, color=black] plot[smooth cycle] (13-4+30.5,6-22) rectangle+ (3,3);
				\draw[line width=0.3mm, color=black] plot[smooth cycle] (16-4+30.5,6-22) rectangle+ (3,3);
				
				\draw[line width=1.5mm, color=black] plot[smooth cycle] (13-4+30.5,3-22) rectangle+ (3,3);
				
				\node[font=\HUGE, anchor=center, text=DarkBlue] at (11.5-4+30.5,1.5-22) {$-$};
				\node[font=\HUGE, anchor=center, text=DarkBlue] at (14.5-4+30.5,1.5-22) {$-$};
				\node[font=\HUGE, anchor=center, text=DarkBlue] at (17.5-4+30.5,1.5-22) {$-$};
				\node[font=\HUGE, anchor=center, text=DarkBlue] at (11.5-4+30.5,4.5-22) {$-$};
				\node[font=\HUGE, anchor=center, text=DarkBlue] at (14.5-4+30.5,4.5-22) {$\mathbf 2$};
				\node[font=\HUGE, anchor=center, text=DarkBlue] at (17.5-4+30.5,4.5-22) {$-$};
				\node[font=\HUGE, anchor=center, text=DarkBlue] at (11.5-4+30.5,7.5-22) {$-$};
				\node[font=\HUGE, anchor=center, text=DarkBlue] at (14.5-4+30.5,7.5-22) {$-$};
				\node[font=\HUGE, anchor=center, text=DarkBlue] at (17.5-4+30.5,7.5-22) {$-$};

				\coordinate[label=center:\HUGE{$\overline{U}$ on $[t^{8}_b, t^{9}_a]$}] (G) at (4.5+4,9.7-36);
				\coordinate[label=center:\HUGE{$G$ on $[t^{8}_b, t^{9}_a]$}] (H) at (14.5+4,9.7-36);
				
				\draw[line width=0.3mm, color=black] plot[smooth cycle] (0+4,-36) rectangle+ (3,3);
				\draw[line width=0.3mm, color=black] plot[smooth cycle] (3+4,-36) rectangle+ (3,3);
				\draw[line width=0.3mm, color=black] plot[smooth cycle] (6+4,-36) rectangle+ (3,3);
				\draw[line width=0.3mm, color=black] plot[smooth cycle] (0+4,3-36) rectangle+ (3,3);
				\draw[line width=0.3mm, color=black] plot[smooth cycle] (6+4,3-36) rectangle+ (3,3);
				\draw[line width=0.3mm, color=black] plot[smooth cycle] (0+4,6-36) rectangle+ (3,3);
				\draw[line width=0.3mm, color=black] plot[smooth cycle] (3+4,6-36) rectangle+ (3,3);
				\draw[line width=0.3mm, color=black] plot[smooth cycle] (6+4,6-36) rectangle+ (3,3);
				
				\draw[line width=1.5mm, color=black] plot[smooth cycle] (3+4,3-36) rectangle+ (3,3);

				\draw[line width=1.5mm,-{Stealth[length=5mm, width=7mm]}]  (7.5+4,7.5-36) -- (4.5+4,4.5-36);

				\draw[line width=0.3mm, color=black] plot[smooth cycle] (10+4,0-36) rectangle+ (3,3);
				\draw[line width=0.3mm, color=black] plot[smooth cycle] (13+4,0-36) rectangle+ (3,3);
				\draw[line width=0.3mm, color=black] plot[smooth cycle] (16+4,0-36) rectangle+ (3,3);
				\draw[line width=0.3mm, color=black] plot[smooth cycle] (10+4,3-36) rectangle+ (3,3);
				
				\draw[line width=0.3mm, color=black] plot[smooth cycle] (16+4,3-36) rectangle+ (3,3);
				\draw[line width=0.3mm, color=black] plot[smooth cycle] (10+4,6-36) rectangle+ (3,3);
				\draw[line width=0.3mm, color=black] plot[smooth cycle] (13+4,6-36) rectangle+ (3,3);
				\draw[line width=0.3mm, color=black] plot[smooth cycle] (16+4,6-36) rectangle+ (3,3);
				
				\draw[line width=1.5mm, color=black] plot[smooth cycle] (13+4,3-36) rectangle+ (3,3);
				
				\node[font=\HUGE, anchor=center, text=DarkBlue] at (11.5+4,1.5-36) {$-$};
				\node[font=\HUGE, anchor=center, text=DarkBlue] at (14.5+4,1.5-36) {$-$};
				\node[font=\HUGE, anchor=center, text=DarkBlue] at (17.5+4,1.5-36) {$-$};
				\node[font=\HUGE, anchor=center, text=DarkBlue] at (11.5+4,4.5-36) {$-$};
				\node[font=\HUGE, anchor=center, text=DarkBlue] at (14.5+4,4.5-36) {$-$};
				\node[font=\HUGE, anchor=center, text=DarkBlue] at (17.5+4,4.5-36) {$-$};
				\node[font=\HUGE, anchor=center, text=DarkBlue] at (11.5+4,7.5-36) {$-$};
				\node[font=\HUGE, anchor=center, text=DarkBlue] at (14.5+4,7.5-36) {$-$};
				\node[font=\HUGE, anchor=center, text=DarkBlue] at (17.5+4,7.5-36) {$-$};

				\coordinate[label=center:\HUGE{$\overline{U}$ on $[t^{9}_a, t^{9}_b]$}] (I) at (4.5-4+30.5,9.7-36);
				\coordinate[label=center:\HUGE{$G$ on $[t^{9}_a, t^{9}_b]$}] (J) at (14.5-4+30.5,9.7-36);
				\draw[line width=0.3mm, color=black] plot[smooth cycle] (0-4+30.5,0-36) rectangle+ (3,3);
				\draw[line width=0.3mm, color=black] plot[smooth cycle] (3-4+30.5,0-36) rectangle+ (3,3);
				\draw[line width=0.3mm, color=black] plot[smooth cycle] (6-4+30.5,0-36) rectangle+ (3,3);
				\draw[line width=0.3mm, color=black] plot[smooth cycle] (0-4+30.5,3-36) rectangle+ (3,3);
				
				\draw[line width=0.3mm, color=black] plot[smooth cycle] (6-4+30.5,3-36) rectangle+ (3,3);
				\draw[line width=0.3mm, color=black] plot[smooth cycle] (0-4+30.5,6-36) rectangle+ (3,3);
				\draw[line width=0.3mm, color=black] plot[smooth cycle] (3-4+30.5,6-36) rectangle+ (3,3);
				\draw[line width=0.3mm, color=black] plot[smooth cycle] (6-4+30.5,6-36) rectangle+ (3,3);
				
				\draw[line width=1.5mm, color=black] plot[smooth cycle] (3-4+30.5,3-36) rectangle+ (3,3);
				\node[font=\HUGE, anchor=center, text=DarkBlue] at (1.5-4+30.5,1.5-36) {$-$};
				\node[font=\HUGE, anchor=center, text=DarkBlue] at (4.5-4+30.5,1.5-36) {$-$};
				\node[font=\HUGE, anchor=center, text=DarkBlue] at (7.5-4+30.5,1.5-36) {$-$};
				\node[font=\HUGE, anchor=center, text=DarkBlue] at (1.5-4+30.5,4.5-36) {$-$};
				\node[font=\HUGE, anchor=center, text=DarkBlue] at (4.5-4+30.5,4.5-36) {\bf I};
				\node[font=\HUGE, anchor=center, text=DarkBlue] at (7.5-4+30.5,4.5-36) {$-$};
				\node[font=\HUGE, anchor=center, text=DarkBlue] at (1.5-4+30.5,7.5-36) {$-$};
				\node[font=\HUGE, anchor=center, text=DarkBlue] at (4.5-4+30.5,7.5-36) {$-$};
				\node[font=\HUGE, anchor=center, text=DarkBlue] at (7.5-4+30.5,7.5-36) {$-$};
				
				\draw[line width=0.3mm, color=black] plot[smooth cycle] (10-4+30.5,0-36) rectangle+ (3,3);
				\draw[line width=0.3mm, color=black] plot[smooth cycle] (13-4+30.5,0-36) rectangle+ (3,3);
				\draw[line width=0.3mm, color=black] plot[smooth cycle] (16-4+30.5,0-36) rectangle+ (3,3);
				\draw[line width=0.3mm, color=black] plot[smooth cycle] (10-4+30.5,3-36) rectangle+ (3,3);
				
				\draw[line width=0.3mm, color=black] plot[smooth cycle] (16-4+30.5,3-36) rectangle+ (3,3);
				\draw[line width=0.3mm, color=black] plot[smooth cycle] (10-4+30.5,6-36) rectangle+ (3,3);
				\draw[line width=0.3mm, color=black] plot[smooth cycle] (13-4+30.5,6-36) rectangle+ (3,3);
				\draw[line width=0.3mm, color=black] plot[smooth cycle] (16-4+30.5,6-36) rectangle+ (3,3);
				
				\draw[line width=1.5mm, color=black] plot[smooth cycle] (13-4+30.5,3-36) rectangle+ (3,3);
				
				\node[font=\HUGE, anchor=center, text=DarkBlue] at (11.5-4+30.5,1.5-36) {$-$};
				\node[font=\HUGE, anchor=center, text=DarkBlue] at (14.5-4+30.5,1.5-36) {$-$};
				\node[font=\HUGE, anchor=center, text=DarkBlue] at (17.5-4+30.5,1.5-36) {$-$};
				\node[font=\HUGE, anchor=center, text=DarkBlue] at (11.5-4+30.5,4.5-36) {$-$};
				\node[font=\HUGE, anchor=center, text=DarkBlue] at (14.5-4+30.5,4.5-36) {$\mathbf 9$};
				\node[font=\HUGE, anchor=center, text=DarkBlue] at (17.5-4+30.5,4.5-36) {$-$};
				\node[font=\HUGE, anchor=center, text=DarkBlue] at (11.5-4+30.5,7.5-36) {$-$};
				\node[font=\HUGE, anchor=center, text=DarkBlue] at (14.5-4+30.5,7.5-36) {$-$};
				\node[font=\HUGE, anchor=center, text=DarkBlue] at (17.5-4+30.5,7.5-36) {$-$};
				
				\filldraw[line width=0.3mm] (9.5+15.25,9.7-33) circle (0.2);
				\filldraw[line width=0.3mm] (9.5+15.25,9.7-34) circle (0.2);
				\filldraw[line width=0.3mm] (9.5+15.25,9.7-35) circle (0.2);
			\end{tikzpicture}
		}
		\caption{The behavior of~$\overline{u}^{\star}$ and $g^{\star}$ from~\eqref{equation:introlg} on the time interval $[0, T^{\star}]$ is schematically denoted for a subdivision of~$\mathbb{T}^2$ by letters from A to I and numbers from $1$ to $9$, respectively. It is then indicated how $G$ and $\overline{U}$ from \eqref{equation:introtrp} are obtained with the help of $g^{\star}$ and $\overline{u}^{\star}$, and arrows indicate how information is propagated along $\overline{U}$ into the control zone (center square). In regions marked with \enquote{$-$}, the flow or controls are inactive at the given time. For example, the values of $\overline{U}$ on the time interval \smash{$[t^1_a, t^1_b]$} marked with \enquote{A} in the center region of the domain correspond to the values of \smash{$\overline{u}^{\star}$}  on the time interval $[0, T^{\star}]$ marked with \enquote{A} in the bottom-left region of the domain.}
		\label{Figure:Pattern}
	\end{figure}
	Next, through a hydrodynamic scaling limit, which is commonly part of the return method in the context of incompressible fluids ({\it c.f.}~\cite[Chapter 6]{Coron2007} and \cite{Coron1996EulerEq}), our result for \eqref{equation:introtrp} yields finite-dimensional and physically localized controls that steer the temperature approximately to any target, while the velocity is kept near its initial state. To this end, the profile~$\overline{U}$ will be used as a reference trajectory (in the return method sense), satisfying an incompressible Euler system driven by a finite-dimensional and physically localized force. In addition, $\overline{U}$ has to encode certain non-stationary mixing effects, like those provided by $\overline{u}^{\star}$, to guarantee approximate controllability of \eqref{equation:introtrp} by finite-dimensional forces. But, to localize the controls, $\overline{U}$ should also behave like a gradient flow in the complement of $\omegaup$, transporting information into the control region.

	Compared to the Navier--Stokes equations, we can exploit here the buoyant force to steer the velocity indirectly. To control the velocity directly by using the return method, that is, to obtain a version of \Cref{theorem:ctl} with arbitrary target states for the vorticity, we would need to obtain finite-dimensional and physically localized controls $G$ not for \eqref{equation:introtrp}, but instead for the following convection problem with stretching term:
	\begin{equation}\label{equation:intstr}
		\partial_t V + (\overline{U}\cdot \nabla)V + (\Upsilon(V) \cdot \nabla) \nabla \wedge \overline{U} = \mathbb{I}_{\omegaup}G,
	\end{equation}
	where $\Upsilon$ is the inverse div-curl operator defined in \eqref{equation:Upsilon}.

	\subsection{Explicit representations of \texorpdfstring{$\mathscr{F}_{\mathscr{v}}$ and $\mathscr{F}_{\mathscr{t}}$ for a class of regions $\omegaup$}{the control spaces for a special case}}\label{subsection:ce}
	Our proofs of \Cref{theorem:main} and \Cref{theorem:main2} provide explicit control spaces of universal dimension for a the particular class of control regions $\omegaup$ that render $\mathbb{T}^2\setminus\omegaup$ simply-connected. Following the explanations in Remarks~\ref{remark:exampley} and~\ref{remark:exampley2}, as provided in Sections~\ref{section:lct} and~\ref{section:conclusion}, one can see that the space $\mathscr{F}_{\mathscr{v}}$ of smooth functions $\mathbb{T}^2\longrightarrow \mathbb{R}^2$ may be spanned, for instance, by
	\begin{gather*}
		\begin{aligned}
			&\Lambda, && \Sigma, & (e_k \cdot \nabla)\nabla^{\perp}[\chi s_l],  && (e_k \cdot \nabla)\nabla^{\perp}[\chi c_l], && \nabla^{\perp}[\chi s_l], &&
			\nabla^{\perp}[\chi c_l],\\
		\end{aligned}\\
		\begin{aligned}
			& \Delta \nabla^{\perp}[\chi s_l], && (\nabla^{\perp}[\chi s_k] \cdot \nabla)\nabla^{\perp}[\chi s_l], & (\nabla^{\perp}[\chi s_k] \cdot \nabla)\nabla^{\perp}[\chi c_l],\\
			& \Delta \nabla^{\perp}[\chi c_l], && (\nabla^{\perp}[\chi c_k] \cdot \nabla)\nabla^{\perp}[\chi s_l], & (\nabla^{\perp}[\chi c_k] \cdot \nabla)\nabla^{\perp}[\chi c_l]
		\end{aligned}
	\end{gather*}
	with indices $k, l \in \{1,2\}$, while the space $\mathscr{F}_{\mathscr{t}}$ of smooth zero average functions $\mathbb{T}^2\longrightarrow \mathbb{R}$ may be spanned by
	\begin{gather*}
		\begin{aligned}
				& \mu s_l - \mu\frac{\int_{\mathbb{T}^2}\mu(x) s_l(x) \, dx}{\int_{\mathbb{T}^2} \mu(x) \, dx}, && 
			\mu c_l - \mu\frac{\int_{\mathbb{T}^2}\mu(x) c_l(x) \, dx}{\int_{\mathbb{T}^2} \mu(x) \, dx},
		\end{aligned}\\
		\begin{aligned}
			&(\nabla^{\perp}[\chi s_l]\cdot\nabla)\mu, 
			&& (\nabla^{\perp}[\chi c_l]\cdot\nabla)\mu, &
			& (e_1\cdot\nabla)\mu, && (e_2\cdot\nabla)\mu,
		\end{aligned}
	\end{gather*}
	with indices $l \in \{1,2\}$, and where the yet undefined objects appearing in the representations above are specified as follows.
	\begin{enumerate}[$\bullet$]
		\item $s_l(x) = \sin(x_l)$ and $c_l(x) = \cos(x_l)$ for $l \in \{1,2\}$ and $x = (x_1,x_2) \in \mathbb{T}^2$.
		\item $\mu, \chi \in C^{\infty}(\mathbb{T}^2; \mathbb{R})$ are as defined in \Cref{subsection:convection}, solely depending on $\omegaup$ and satisfying $\operatorname{supp}(\mu)\cup\operatorname{supp}(\chi)\subset \omegaup$. See \Cref{example:excp2} for a concrete choice.
		\item The profiles $\Lambda, \Sigma \in C^{\infty}(\mathbb{T}^2; \mathbb{R}^2)$ are curl-free, have linearly independent averages, and their support is contained in $\omegaup$ (see also \cite{NersesyanRissel2024a}).
	\end{enumerate}

	Adding another suitable dimension (for instance, $\operatorname{span} \mu$) to $\mathscr{F}_{\mathscr{t}}$ would allow to control also the temperature average. Utilizing $\Lambda$ and $\Sigma$, one can adapt the explanations in \Cref{remark:nza2} to control \eqref{equation:Boussinesq} between states of different average without adding more dimensions to~$\mathscr{F}_{\mathscr{v}}$.
	
	When $\omegaup$ is arbitrary, our constructions are explicit up to \Cref{lemma:rtc}, which is proved in \cite[Lemma 5.1]{FursikovImanuvilov1999} by a contradiction argument.

	\subsection{Literature}\label{subsection:literature}
	
	The approximate controllability of Navier--Stokes and Euler systems driven by finite-dimensional but not physically localized controls has been established for the $2$D periodic setting via geometric control techniques in~\cite{AgrachevSarychev2006,AgrachevSarychev2005}. This nonlinear approach is known as the Agrachev--Sarychev method. The aforementioned works are further concerned with questions on the controllability of Galerkin approximations~and on the controllability in finite-dimensional projections (see also~\cite{AgrachevSarychev2008}); for an earlier and different result on the controllability of Galerkin approximations, we refer to~\cite{LionsZuazua1998}. Refinements of the Agrachev--Sarychev method and extensions to $3$D settings have been developed subsequently, for instance, in the articles \cite{Shirikyan2006,Shirikyan2007,Nersesyan2021,Nersisyan2011}. These ideas principally extend to other domains, provided that certain saturation properties can be verified. However, this has been done only for special configurations \cite{Shirikyan2007,AgrachevSarychev2008,Shirikyan2006,Rodrigues2007}; for example, when orthonormal bases of trigonometric functions or spherical harmonics are available. See also \cite{Rodrigues2006,PhanRodrigues2019} for $2$D and $3$D rectangular domains under imposition of a slip boundary condition. We also mention that an illustration of the Agrachev--Sarychev method for the example of a $1$D Burgers equation is provided by \cite{Shirikyan2018}, and that Lagrangian and trajectorial controllability have been studied in \cite{Nersesyan2015} for finite-dimensional controls which may act only through few components of the Navier--Stokes system posed on the $3$D torus. A negative result regarding the exact controllability of an incompressible Euler problem driven by finite-dimensional controls is obtained in \cite{Shirikyan2008} by a comparison argument for the Kolmogorov $\varepsilon$-entropy, underscoring that approximate controllability is an appropriate notion in this context. Let us also point to several applications of the Agrachev--Sarychev approach to Schr\"odinger equations, {\it e.g.}, in \cite{Sarychev2012,CoronXiangZhang2023,DucaNersesyan2025}. Further, we emphasize that controllability via finite-dimensional controls is related to the study of ergodicity and other properties of systems driven by degenerate noise; {\it e.g.}, see \cite{MattinglyPardoux2006,HairerMattingly2006} and the recent works \cite{KuksinNersesyanShirikyan2020,BoulvardGaoNersesyan2023,Nersesyan2024,FoldersHerzog2024}, noting that \cite{Nersesyan2024} treats Fourier-localized noise multiplied by a cutoff supported in an arbitrary region of the physical space.

	Recently, physically localized controls of a specific degenerate structure have been obtained in \cite{NersesyanRissel2024a} for the $2$D incompressible Navier--Stokes system with periodic boundary conditions. The there-described approach, which already refers to the notion of observable families and the return method, subsequently inspired a first (global) approximate controllability result for the Boussinesq system driven only by a temperature control~\cite{NersesyanRissel2024}. In both references, the controls are not finite-dimensional but admit explicit representations involving only a finite number of control coefficients. Moreover, in \cite{NersesyanRissel2024a}, the control region must contain two cuts $\mathcal{C}_1$ and $\mathcal{C}_2$ rendering $\mathbb{T}^2\setminus(\mathcal{C}_1 \cup \mathcal{C}_1)$ simply-connected, while in \cite{NersesyanRissel2024} the control region must contain a strip that cuts the torus into two doubly-connected pieces. Contrasting these works, we obtain now truly finite-dimensional controls that are physically localized. Furthermore, we do not impose any restriction on the nonempty open subset of~$\mathbb{T}^2$ containing the support of the controls. To this end, our strategy relies on several new elements. For example, physically localized controls are not obtained like in \cite{NersesyanRissel2024a,NersesyanRissel2024} by composing low-dimensional forces with certain flow maps that spoil their finite-dimensional nature. Instead, we introduce the tailored transport problem~\eqref{equation:introtrp}, which can be controlled directly by a finite-dimensional and physically localized force that is patched together from shifted versions of a frequency-localized control. The existence of such convection problems that are in addition compatible with the return method is an essential new ingredient established here. To achieve this, we carefully construct~the convection profile~$\overline{U}$ in~\eqref{equation:introtrp}, serving as a return method trajectory and encoding a combination of several geometric properties preventing it from being chosen constant with respect to the space variables. This highlights another difference to the aforementioned studies, which rely on spatially constant return method trajectories, for instance, to handle a stretching term as in \eqref{equation:intstr} or to drive the Boussinesq system only through the temperature.

	Regarding less degenerate physically localized distributed- or boundary controls for the Navier--Stokes system and related models, there is a vast body of literature and many questions are actively studied. Attention is often paid not only to approximate- but also to exact controllability properties; for instance, exact controllability to zero (null controllability) or to trajectories. For problems without diffusion, {\it e.g.}, the Euler equations, exact controllability to any target in the state space might be possible (see \cite{Coron1996EulerEq,Glass2000}). This distinction of exact controllability notions is related to the problem of understanding the reachable space; for more background in this direction, we point to the recent work \cite{ErvedozaEtal2022} and the references therein. Aiming for local (small data) results, many authors have invoked linearization techniques and developed Carleman estimates for associated linear problems, leading via local inversion theorems to local exact controllability properties; for instance, \cite{FursikovImanuvilov1998, Guerrero2006,Fernandez-CaraGuerreroImanuvilovPuel2004}, and \cite{CoronLissy2014,Fernandez-CaraGuerreroImanuvilovPuel2006} for controls acting only in few components. By involving the return method, which exploits the nonlineary of the considered system, global controllability results, related to questions posed by J.-L. Lions in the 1980s-1990s ({\it c.f.}~\cite{LionsJL1991}), have been obtained for the incompressible Euler and Navier--Stokes equations \cite{Glass2000,Coron1996EulerEq,Coron96,CoronFursikov1996,CoronMarbachSueurZhang2019,CoronMarbachSueur2020,LiaoSueurZhang2022,LiaoSueurZhang2022b} and other models like the Boussinesq system and magnetohydrodynamics \cite{Chaves-SilvaEtal2023,FursikovImanuvilov1999,Rissel2024, RisselWang2024,Liao2024,KukavicaNovackVicol2022}, to name only a few. In these situations, the main issue is usually to prove approximate controllability, which is a global notion of controllability that may be combined with a local result to achieve global exact controllability to zero or to trajectories. In particular, the study \cite{CoronMarbachSueur2020} resolves in both $2$D and $3$D a Navier slip version of J.-L. Lions' famous open problem on the approximate controllability of the Navier--Stokes system. Namely, the authors impose Navier slip-with-friction boundary conditions instead of the no-slip boundary condition; similar findings for the Boussinesq system have been obtained in \cite{Chaves-SilvaEtal2023}, and \cite{Rissel2024} demonstrates approximate controllability only through the temperature for a Boussinesq system in a planar channel with thermally insulated physical boundaries on two sides along which the fluid can slip. Concerning the no-slip boundary condition, \cite{CoronMarbachSueurZhang2019} establishes global controllability of the Navier--Stokes system in a rectangular region under the addition of a small phantom force (see \cite{LiaoSueurZhang2022b} for curved boundaries). We mention also recent achievements of small-time local stabilization~for the planar Navier--Stokes system in \cite{Shengquan2023} and small-time global stabilization for~the viscous Burgers equation with three scalar controls in \cite{CoronShengquan2021}, referring to the bibliographies of these articles for further references on stabilization problems.

	\section{Transport equations with generating drift}\label{section:knownfinre}
	
	In this section, we recall a recent result from \cite{NersesyanRissel2024a} on the approximate controllability of transport problems with convection along generating drifts; the argument goes back to \cite{Nersesyan2021} for a $3$D linearized Euler problem, based on a notion of observable families from \cite{KuksinNersesyanShirikyan2020}. 
	
	To begin with, let $\mathcal{K} \subset \mathbb{Z}^2 \setminus \{0\}$ be a finite set with $\operatorfont{span}_{\mathbb{Z}}(\mathcal{K}) = \mathbb{Z}^2$ and consider the space
	\begin{equation}\label{equation:defH}
		\mathscr{H} \coloneq \mathscr{H}(\mathcal{K}) \coloneq \operatorname{span} \left\{ s_{\ell}, c_{\ell} \, \left| \right. \, \ell \in \mathcal{K} \right\},
	\end{equation}
	where
	\begin{gather*}
		s_{\ell}(x) \coloneq \sin( \ell \cdot x), \quad
		c_{\ell}(x) \coloneq \cos( \ell \cdot x)
	\end{gather*}
	for $x \in \mathbb{T}^2$ and $\ell \in \mathcal{K}$. 
	Then, observable families are defined as in \cite{Nersesyan2021}, providing a stronger version of the concept introduced in \cite[Definition 4.1]{KuksinNersesyanShirikyan2020}.
	
	\begin{dfntn}\label{definition:observable}
		Given $T > 0$ and $N \in \mathbb{N}$, a family $(\phi_j)_{j \in \{1,\dots,N\}} \subset L^2((0,T);\mathbb{R})$ is called observable if, for all $J \subset (0,T)$,~$b \in C^0(J;\mathbb{R})$, and~$(a_j)_{j \in \{1,\dots,N\}} \subset C^1(J; \mathbb{R})$ it holds
		\begin{gather*}
			b + \sum_{j=1}^N a_j \phi_j = 0 \, \mbox{ in } L^2(J;\mathbb{R})
			\iff 
			b = a_1 = \cdots = a_N = 0.
		\end{gather*}
	\end{dfntn}
	
	Observable families are known to exist; for instance, see \cite[Section 3.3]{Nersesyan2021}. We call a divergence-free vector field generating if it can be constructed in the below-described manner from an observable family.
	\begin{dfntn}\label{definition:ovf}
		Let $T > 0$. We say that $\overline{u}\in W^{1,2}((0,T);C^{\infty}(\mathbb{T}^2;\mathbb{R}^2))$ is generating, if it has the form
		\begin{equation}\label{equation:ObservableVF}
			\begin{gathered}
				\overline{u}(x, t) = \kappa\sum_{\ell \in \mathcal{K}} \left(\psi_{\ell}^{\operatorname{s}}(t) s_{\ell}(x) \ell^{\perp} + \psi_{\ell}^{\operatorname{c}}(t) c_{\ell}(x) \ell^{\perp} \right)\\
			\end{gathered}
		\end{equation}
		for $(x,t) \in \mathbb{T}^2\times[0,T]$, where
		\begin{gather*}
			\psi_{\ell}^{\operatorname{s}}(t) \coloneq \phi(t) \int_0^t \phi_{\ell}^{\operatorname{s}}(r) \, dr, \quad
			\psi_{\ell}^{\operatorname{c}}(t) \coloneq \phi(t) \int_0^t \phi_{\ell}^{\operatorname{c}}(r) \, dr
		\end{gather*}
		and
		\begin{enumerate}[$\bullet$]
			\item $(\phi_{\ell}^{\operatorname{s}},\phi_{\ell}^{\operatorname{c}})_{\ell \in \mathcal{K}} \subset L^2((0,T); \mathbb{R})$ is an observable family,
			\item $\phi \in C^1([0,T];\mathbb{R})$ obeys $\phi(t) = 0$ if an only if $t = T$,
			\item $\ell^{\perp} \coloneq (-\ell_2, \ell_1)$ for any $\ell = (\ell_1, \ell_2) \in \mathcal{K}$,
			\item $\kappa \in \mathbb{R}\setminus\{0\}$.
		\end{enumerate}
	\end{dfntn}
	
	\begin{rmrk}\label{remark:kappa}
		It is convenient for us to keep the parameter~$\kappa$ in \Cref{definition:ovf}, despite its redundancy. In particular, given $(\phi_{\ell}^{\operatorname{s}},\phi_{\ell}^{\operatorname{c}})_{\ell \in \mathcal{K}}$, $\phi$, and any $R > 0$, a compactness argument allows to ensure $\max_{(x,t)\in \mathbb{T}^2\times[0,T]}|\overline{u}(x,t)| < R$ by appropriately choosing~$\kappa$.
	\end{rmrk}

	The next result, which is known from \cite{NersesyanRissel2024a} and goes back to \cite{Nersesyan2021}, demonstrates a certain mixing effect that propagates energy from low frequency (Fourier-localized) sources to higher frequencies.
	\begin{lmm}{\cite[Theorem 2.6]{NersesyanRissel2024a}}\label{lemma:lineareverywheresupported_cpam}
		Let $T > 0$ and $\overline{u} \in W^{1,2}((0,T);C^{\infty}(\mathbb{T}^2;\mathbb{R}^2))$ be generating. Given $m \in \mathbb{N}$, $v_{1} \in H^m$, and $\varepsilon > 0$, there exists a control $\zeta \in L^2((0,T); \mathscr{H})$ such that the solution to the transport problem $\partial_t v + (\overline{u} \cdot \nabla) v = \zeta$ with initial condition $v(\cdot, 0) = 0$ satisfies $\|v(\cdot, T)-v_{1}\|_{m} < \varepsilon$.
	\end{lmm}

	\section{Finite-dimensional and physically localized transport controls}\label{section:lct}
	The goal of this section is to establish approximate controllability of a transport equation driven by physically localized and finite-dimensional controls. This will be possible for a special convection profile $\overline{U}$, introduced in \Cref{subsection:defbU} after several preliminary constructions of auxiliary profiles $\overline{y}$ and $\overline{u}^{\star}$. The main controllability result of this section is then stated in~\Cref{theorem:locfinthm}.
	

	\subsection{Finite-dimensional flushing profile \texorpdfstring{$\overline{y}$}{}}\label{subsection:convection}
	We construct a non-stationary vector field~$\overline{y}$, belonging at each time to a finite-dimensional subspace of $C^{\infty}(\mathbb{T}^2;\mathbb{R}^2)$ and constituting a time-periodic solution to an incompressible Euler system driven by finite-dimensional and physically localized controls (see~\Cref{lemma:retf}).
	
	\subsubsection{Partition of unity}
	For a given length $L = L_{\omegaup} > 0$, we denote by $(\mathcal{O}_i^L)_{i\in\{1,\dots,M^L\}}$ an open covering of $\mathbb{T}^2$ by overlapping squares of side-length~$L$ and with bottom left corners $(o_i^L)_{i \in \{1,\dots,M^L\}} \subset \mathbb{T}^2$. Moreover, $(\mu_i^L)_{i \in \{1,\dots,M^L\}} \subset C^{\infty}(\mathbb{T}^2; [0,1])$ is a subordinate partition of unity:
	\begin{gather*}
		\forall i \in \{1,\dots,M^L\}\colon \operatorname{supp}(\mu_i^L) \subset \mathcal{O}_i, \quad
		\sum_{i=1}^{M^L} \mu_i^L = 1. 
	\end{gather*}
	As demonstrated by \Cref{example:excp2} below, these choices can be made in agreement with the following additional properties.
	\begin{enumerate}[$\bullet$]
		\item There exists~$\mathcal{O}^L \subset \mathbb{T}^2$ with $\smash{\overline{\mathcal{O}^L}} \subset \omegaup$ so that $\mathcal{O}^L = \mathcal{O}_i^L + S_i^L \coloneq \{x + S_i^L \, | \, x \in \mathcal{O}_i^L\}$ for translation vectors $S_i^L \in \mathbb{R}^2$ and $i \in \{1,\dots,M^L\}$.
		\item There exists~$\mu^L \in C^{\infty}(\mathbb{T}^2; [0,1])$ with $\mu^L_i(\cdot) = \mu^L(\cdot + S_i^L)$ for each $i \in \{1,\dots,M^L\}$.
	\end{enumerate}
	
	We then introduce a cutoff $\chi^L \in C^{\infty}(\mathbb{T}^2; \mathbb{R})$ which satisfies $\operatorname{supp}(\chi^L) \subset \omegaup$ and $\chi^L = 1$ on a neighborhood of $\overline{\mathcal{O}^L}$. Furthermore, the reference time interval~$[0,1]$ is partitioned with equidistant spacing $T^{\star,L} \coloneq 1/(3M^L+2)$ by means of
	\begin{multline*}\label{equation:eqpartition}
		0 < t^{0,L}_c < t^{1,L}_a < t^{1,L}_b < t^{1,L}_c < \dots < t^{M^L,L}_a < 	t^{M^L,L}_b < t^{M^L,L}_c < 1.
	\end{multline*}
	
	\paragraph*{Simplified notations.} When the dependence on~$L = L_{\omegaup}$ is clear, we drop the superscript \enquote{$L$} from the notations and write $M$, $\mathcal{O}_i$, $\mathcal{O}$, $S_i$, $o_i$, $\mu_i$, $\mu$, $\chi$, $t^0_c$, $t^1_a$, $t^1_b$, $t^1_c$, \dots, $t^M_a$, $t^M_b$, $t^M_c$, $T^{\star}$.
	
		\begin{figure}[ht!]
		\centering
		\resizebox{0.55\textwidth}{!}{
			\begin{tikzpicture}
				\clip(-0.1-0.1,-0.1-0.1) rectangle (4.8+0.1,4.8+0.1);

				\draw[line width=0.2mm, color=black] plot[smooth cycle] (0,0) rectangle (1,1);
				\draw[line width=0.2mm, color=black] plot[smooth cycle] (0.75,0) rectangle (1.75,1);
				\draw[line width=0.2mm, color=black] plot[smooth cycle] (0,0.75) rectangle (1,1.75);
				
				\draw[line width=0.2mm, color=FireBrick] plot[smooth cycle] (3.35-0.5,3.35-0.5) rectangle (4.35-0.5,4.35-0.5);

				\coordinate[label=below:{\color{DarkBlue}\tiny$\chi = 1$}] (A) at (2.35+0.4,2.35+0.4);
				\coordinate[label=below:{\color{DarkGreen}\tiny$\chi = 0$}] (B) at (2.35,2.35-0.28);
				
				\coordinate[label=below:{\tiny$\mathcal{O}_1$}] (C) at (0.48,0.7);
				\coordinate[label=below:{\tiny$\mathcal{O}_2$}] (D) at (0.48+0.75,0.7);
				\coordinate[label=below:{\tiny$\mathcal{O}_{\sqrt{M}+1}$}] (C) at (0.51,0.8+0.75);

				\coordinate[label=below:\color{FireBrick}\tiny$\operatorname{supp}(\mu)$] (B) at (2.85+0.5,3.07+0.5);
				
				\draw[line width=0.2mm, dashed, color=DarkBlue] plot[smooth cycle] (2.3,2.3) rectangle (4.4,4.4);
				
				\draw[line width=0.2mm, dashed, color=DarkGreen] plot[smooth cycle] (2,2) rectangle (4.7,4.7);
				
			\end{tikzpicture}
		}
		\caption{Illustration of the covering $(\mathcal{O}_i)_{i\in\{1,\dots,M\}}$ (three example squares printed) and important values of the cutoff functions~$\mu$ and~$\chi$. Within the inner dashed square, which includes the support of $\mu$, one has $\chi = 1$. Exterior to the outer dashed square, $\chi$ vanishes.}
		\label{Figure:Covering}
	\end{figure}

	The next example demonstrates that partitions of unity with the required properties exist. See also \Cref{Figure:Covering} for a simplified illustration of the introduced setup.

	\begin{xmpl}\label{example:excp2}
		Let the closure of an open square $\mathcal{O}$ with side-length $L > 0$ be contained in $\omegaup$. As $\omegaup$ is open, we assume without loss of generality that~$2\pi/L$ is not an integer; thus, $M \coloneq \lceil 2\pi/L \rceil^2 > (2\pi/L)^2$. Now, choose $\mathcal{O}_1, \dots, \mathcal{O}_M$ as translations of~$\mathcal{O}$ with overlap-width $(\sqrt{M}L - 2\pi)/\sqrt{M}$ and bottom left corners $(o_i = (o_{i,1}, o_{i,2}))_{i\in\{1,\dots,M\}}$ given by
		\begin{align*}
			o_{i + \sqrt{M}(l - 1),1} = \frac{2\pi(i-1)}{\sqrt{M}}, \quad o_{i + \sqrt{M}(l - 1),2} = \frac{2\pi(l-1)}{\sqrt{M}}
		\end{align*}
		for $i,l = 1, \dots, \sqrt{M}$.
		Further, take $\widehat{\mu} \in C^{\infty}_0((0, L];[0,1])$ satisfying $\widehat{\mu}(s) = 1$ if and only if $s \in [(\sqrt{M}L-2\pi)/(2\sqrt{M}), L]$ and define $\widetilde{\mu} \in C^{\infty}(\mathbb{T};[0,1])$ via
		\[
		\widetilde{\mu}(s) = \mathbb{I}_{\left[0,\frac{\sqrt{M}L - 2\pi}{\sqrt{M}}\right]}(s) \widehat{\mu}(s) + \mathbb{I}_{\left(\frac{\sqrt{M}L - 2\pi}{\sqrt{M}},\frac{2\pi}{\sqrt{M}}\right)}(s) + \mathbb{I}_{\left[\frac{2\pi}{\sqrt{M}},L\right]}(s) \left(1 - \widehat{\mu}\left(s - \frac{2\pi}{\sqrt{M}} \right)\right)
		\]
		for each $s \in \mathbb{T}$. In particular, the function $\widetilde{\mu}$ satisfies $\operatorname{supp}(\widetilde{\mu}) \subset (0, L)$ and
		\begin{gather*}
			\sum_{l=1}^{\sqrt{M}} \widetilde{\mu}\left(x + \frac{2\pi(l-1)}{\sqrt{M}}\right) = 1
		\end{gather*}
		for all $x \in \mathbb{T}$. Finally, the cutoff $\mu\in C^{\infty}(\mathbb{T}^2;[0,1])$ is chosen via
		\[
		\mu(x) \coloneq \widetilde{\mu}(x_1-o_1) \widetilde{\mu}(x_2-o_2)
		\]
		for $x = (x_1,x_2)\in\mathbb{T}^2$ and $(o_1,o_2)$ being the bottom left vertex of $\mathcal{O}$. For a similar example see \cite[Example 3.1]{NersesyanRissel2024a}.
	\end{xmpl}

	\subsubsection{Definition of the vector field $\overline{y}$}
	The length $L > 0$ that determines the open covering $(\mathcal{O}_i^L)_{i \in \{1,\dots,M^L\}}$ will be fixed below in sole dependence on $\omegaup$; see \Cref{lemma:retf}. Hereto, we recall first the existence of a flushing trajectory in the return method sense (see \cite[Chapter 6]{Coron2007} for an introduction of the return method). 
	\begin{lmm}[{\cite[Lemma 5.1, Section 6]{FursikovImanuvilov1999}}]\label{lemma:rtc}
		Given any $T > 0$ and a nonempty open set $\omegaup_0 \subset \mathbb{T}^2$, there exists a vector field $\overline{Y} \in C^{\infty}_0((0,T); C^{\infty}(\mathbb{T}^2; \mathbb{R}^2))$ that satisfies
		\begin{equation*}\label{equation:propreto}
			\begin{gathered}
				\forall (x,t) \in \mathbb{T}^2\times[0,T]\colon \operatorname{div}(\overline{Y})(x,t) = 0,\\
				\exists \Psi \in C^{\infty}_0((0,T); C^{\infty}(\mathbb{T}^2; \mathbb{R})), \, \forall (x,t) \in (\mathbb{T}^2\setminus \omegaup_0)\times[0,T] \colon \overline{Y}(x,t) = \nabla \Psi (x,t),\\
				\forall x \in \mathbb{T}^2,\, \exists t_x \in (0,T) \colon \,\Phi^{\overline{Y}}(x, 0, t_x) \in \omegaup_0.
			\end{gathered}
		\end{equation*}
	\end{lmm}
	
	\begin{rmrk}\label{remark:retpp}
		Since $(\overline{Y} \cdot \nabla) \overline{Y}$ can be written as a gradient wherever $\overline{Y}$ is a gradient, the first two properties of $\overline{Y}$ in \Cref{lemma:rtc} ensure that $\overline{Y}$ solves in $\mathbb{T}^2\times[0,T]$ the controlled incompressible Euler system
		\begin{gather*}
			\partial_t \overline{Y} + (\overline{Y}\cdot\nabla)\overline{Y} + \nabla \overline{p} = \mathbb{I}_{\omegaup_0} \overline{\xi},\\
			\operatorname{div}(\overline{Y}) = 0,\\
			\overline{Y}(\cdot, 0) = \overline{Y}(\cdot, 1) = 0
		\end{gather*}
		for a smooth pressure $\overline{p}$ and a smooth control $\overline{\xi}$ with $\operatorname{supp}(\overline{\xi}) \subset \omegaup_0\times(0,T)$. The third property of $\overline{Y}$ in \Cref{lemma:rtc} states that information originating from any location in~$\mathbb{T}^2$ is flushed along~$\overline{Y}$ into the given set~$\omegaup_0$.
	\end{rmrk}
	
	A suitable choice of the length $L > 0$ in the definition of squares $(\mathcal{O}_i^L)_{i \in \{1,\dots,M^L\}}$ is now determined together with a finite-dimensional return method trajectory~$\overline{y}$ that has several refined properties.
	
	\begin{thrm}\label{lemma:retf}
		There exist $L = L_{\omegaup} > 0$, $D_{\omegaup} \in \mathbb{N}$, a $D_{\omegaup}$-dimensional vector space $\mathscr{H}_{\omegaup} \subset C^{\infty}(\mathbb{T}^2; \mathbb{R}^2)$, a neighborhood $\mathcal{N}(\mathbb{T}^2\setminus \omegaup)$ of $\mathbb{T}^2\setminus \omegaup$, and a vector field $\overline{y} = \overline{y}_{\omegaup} \in C^{\infty}_0((0,1); \mathscr{H}_{\omegaup})$ with the properties
		\begin{gather}
			\forall h \in \mathscr{H}_{\omegaup}, \,  \forall x \in \mathbb{T}^2\colon \operatorname{div}(h)(x) = 0\label{equation:propybar-1},\\
			\forall h \in \mathscr{H}_{\omegaup}, \, \exists \varphi_h \in C^{\infty}(\mathbb{T}^2; \mathbb{R}), \, \forall x \in \mathcal{N}(\mathbb{T}^2\setminus \omegaup) \colon h(x) = \nabla \varphi_h (x),\label{equation:propybar0}\\
			\forall t \in [0, T^{\star}] \cup [1-T^{\star}, 1] \cup [t^1_a, t^1_b] \cup \dots \cup [t^M_a, t^M_b] \colon \overline{y}(t) = 0\label{equation:propybar1}
		\end{gather}
		and
		\begin{gather}
			\begin{gathered}
				\overline{y}(\cdot, t^{i-1}_c+ t) = - \overline{y}(\cdot, t^{i}_c - t), \\
				\Phi^{\overline{y}}(x,t^{i-1}_c, t^{i-1}_c+t) = \Phi^{\overline{y}}(x,t^{i-1}_c, t^{i}_c-t), 
			\end{gathered}\label{equation:propybar2}\\
			\operatorname{dist}(x,\mathcal{O}_i) < L \implies \,\Phi^{\overline{y}}(x, 0, [t^i_a, t^i_b]) = \{x + S_i\}\label{equation:propybar3}
		\end{gather}
		for all $(x,t) \in \mathbb{T}^2\times[0, 3T^{\star}]$,  $i\in\{1,\dots,M\}$, and where the objects
		\[
		M = M^L, \, \mathcal{O} = O^L, \, \mathcal{O}_i = \mathcal{O}_i^L, \, S_i = S_i^L, \, t^i_c = t^{i,L}_c, \, t^i_a = t^{i,L}_a, \, t^i_b = t^{i,L}_b, \, T^{\star} = T^{\star,L}
		\]
		are chosen in dependence on $L$ as described at the beginning of \Cref{subsection:convection}. 
		
	\end{thrm}
	The proof of \Cref{lemma:retf}, which is based on \Cref{lemma:rtc}, is carried out after several remarks below.
	\begin{rmrk}
		The property \eqref{equation:propybar2} describes how transportation along $\overline{y}$ oscillates back and forth, thereby inducing a flow with specific periodic behavior. The property  \eqref{equation:propybar3} ensures that a neighborhood of each square~$\mathcal{O}_i$ is transported in time~$t^i_a$ to a rigid translation of itself contained in $\omegaup$. 
	\end{rmrk}
	\begin{rmrk}\label{remark:exampley}
		If~$\mathbb{T}^2\setminus\omegaup$ is simply-connected, one can skip the proof of \Cref{lemma:retf} and instead use the explicit example provided by \cite[Theorem 3.2]{NersesyanRissel2024a} with $\mathscr{H}_{\omegaup} = \mathbb{R}^2$ being a two-dimensional space of constant functions. More specifically, for any fixed basis $\{b_1, b_2\}$ of $\mathbb{R}^2$, one can take
		\[
		\overline{y}(t) = \overline{y}_1(t) b_1 +  \overline{y}_2(t) b_2,
		\]
		where $\overline{y}_1,\overline{y}_2\colon [0,1] \longrightarrow \mathbb{R}$ are smooth functions chosen such that \cref{equation:propybar-1,equation:propybar0,equation:propybar1,equation:propybar2,equation:propybar3} hold. In this case, a covering $(\mathcal{O}_i)_{i\in\{1,\dots,M\}}$ is obtained by fixing any $L > 0$ such that a closed square of side-length $L$ is contained in $\omegaup$. Because $\mathbb{T}^2\setminus \omegaup$ is simply-connected, for each $t \in [0,1]$ one can express~$\overline{y}(t)$ as the sum of a gradient and a curl-free function supported in~$\omegaup$. See \cite{NersesyanRissel2024a} for more details.
	\end{rmrk}
	\begin{proof}[{Proof of \Cref{lemma:retf}}]
		Let us begin with a description of the idea. We start with an appropriate reference flow from \Cref{lemma:rtc} and choose $L = L_{\omegaup}$ so small that it transports entire neighborhoods of $\mathcal{O}_1^L, \dots, \mathcal{O}_{M^L}^L$ through~$\omegaup$. In fact, we will concatenate several scaled copies of that flow in order to prescribe the instances of time at which the content of each square is mapped into the control region. A finite-dimensional approximation of each so-obtained flow is constructed via a time discretization argument. Everything until here, we call \enquote{Building block 1}. To achieve the property~\eqref{equation:propybar3}, we start by reversing in time the dynamics of each square~$\mathcal{O}_i$ being transported along a flow described via \enquote{Building block 1} to a respective set~$A_i$ contained inside the control region; this shows how the set $A_i$ can be mapped to a square using a return method flow. However, the so-achieved square will be~$\mathcal{O}_i$ and thus may intersect $\mathbb{T}^2\setminus\omegaup$. Therefore, to contain this reversed process fully in~$\omegaup$, we use stream function cutoffs and local shifts to modify the reversed return method flow inside of~$\omegaup$ while setting it zero away from~$\omegaup$. We will call this \enquote{Building block 2}.
		In the end, the building blocks are glued in an appropriate manner. Hereby, to ensure \eqref{equation:propybar2}, original and time-reversed versions of the building blocks are iterated. 
		
		\paragraph*{Fixing $L$ and a reference flow.}
		We choose $\overline{Y}$ via \Cref{lemma:rtc} for an open ball $\omegaup_0\subset \omegaup$ of diameter $d_0 > 0$ and a fixed time $T > 0$ ({\it e.g., $T = 1$}). By the compactness of $\mathbb{T}^2$ and smoothness of $\overline{Y}$, we take $L = L_{\omegaup} > 0$ so small that for each $i \in \{1, \dots, M = M^L\}$ there exists a time $t_i \in (0, T)$, a neighborhood $B_i$ of $\mathcal{O}_i = \mathcal{O}^L_i$, and a family of balls $(B_{i,t})_{t \in [0,T]}$ of radius $d_0/6$ satisfying
		\begin{gather*}
			\operatorname{dist}(\partial B_i, \mathcal{O}_i) > L, \\ \Phi^{\overline{Y}}(B_i, 0, t_i) \subset \omegaup_0, \\ \operatorname{dist}(\Phi^{\overline{Y}}(B_i, 0, t_i), \partial \omegaup_0) > d_0/3,\\
			\Phi^{\overline{Y}}(B_i, 0, t) \subset B_{i,t}
		\end{gather*}
		for all $t \in [0,T]$. The first property specifies the location of $O_i$ in the open set~$B_i$. The second and third properties express that~$B_i$ should be flushed in time~$t_i$ sufficiently \enquote{deep} into $\omegaup_0$. The last property states that $\Phi^{\overline{Y}}$ cannot tear $B_i$ too much apart: the image of $B_i$ under the flow must be confined at each time to a ball of radius $d_0/6$.

		To simplify the presentation, let us now always assume that the index $i$ ranges over the set $\{1, \dots, M\}$. 
		
		\paragraph*{Building block 1.}
		The scaled versions $\overline{Y}_i(x,t) \coloneq r_i\overline{Y}(x, r_i t)$ with $r_i \coloneq 2t_i/T^{\star}$ satisfy $\Phi^{\overline{Y}}(x, 0, r_i t) = \Phi^{\overline{Y}_i}(x,0,t)$ for all $(x,t) \in \mathbb{T}^2\times[0, T^{\star}/2]$. Consequentially,
		\begin{gather*}
			\Phi^{\overline{Y}_i}(B_i, 0, T^{\star}/2) \subset \omegaup_0, \\ \operatorname{dist}(\Phi^{\overline{Y}_i}(B_i, 0, T^{\star}/2), \partial \omegaup_0) > d_0/3,\\
			\Phi^{\overline{Y}_i}(B_i, 0, t) \subset B_{i,r_i t}
		\end{gather*}
		for each $t \in [0, T^{\star}/2]$.
		
		Next, we introduce finite-dimensional versions of the vector fields $\overline{Y}_1, \dots, \overline{Y}_M$. Hereto, we fix possibly large $N_i \in \mathbb{N}$, together with
		\begin{gather*}
			\rho_i \in C^{\infty}_0((0,T^{\star}/2)\setminus\{T^{\star}/2N_i, T^{\star}/N_i, \dots, (N_i-1)T^{\star}/2N_i\}; [0,1]),\\
			\mathcal{I}_k \coloneq [(k-1)T^{\star}/2N_i, kT^{\star}/2N_i]
		\end{gather*}
		for $k \in \{1, \dots, N_i\}$, such that 
		\begin{equation}\label{equation:TYi}
			\widetilde{Y}_i(x,t) \coloneq \rho_i(t) \sum_{k=1}^{N_i} \mathbb{I}_{\mathcal{I}_k}(t) \overline{Y}_i(x, k T^{\star}/2N_i)
		\end{equation}
		satisfies
		\begin{equation}\label{equation:npf}
			\begin{gathered}
				\Phi^{\widetilde{Y}_i}(B_i, 0, T^{\star}/2) \subset \omegaup_0, \\ \operatorname{dist}(\Phi^{\widetilde{Y}_i}(B_i, 0, T^{\star}/2), \partial \omegaup_0) > d_0/3, \\
				\max_{a\in \mathbb{T}^2,\, |s-r| \leq T^{\star}/2N_i} |	\Phi^{\widetilde{Y}_i}(a, s, r) - a| < d_0/6
			\end{gathered}\\
		\end{equation}
		and
		\begin{equation}\label{equation:ballit2}
			\Phi^{\widetilde{Y}_i}(B_i, 0, t) \subset B_{i,r_i t}
		\end{equation}
		for all $(x,t) \in \mathbb{T}^2\times[0, T^{\star}/2]$.
		
		The possibility to choose such numbers $N_1, \dots, N_M$ and profiles $\rho_1, \dots, \rho_{M}$ is due to the definition of flows of vector fields in \eqref{equation:flow} and Gr\"onwall's inequality. Indeed, it holds
		\begin{gather*}
			\max_{(x,t) \in \mathbb{T}^2\times[0,T^{\star}/2]}|\Phi^{\overline{Y}_i}(x,0,t) - \Phi^{\widetilde{Y}_i}(x,0,t)|
			\leq C \|\overline{Y}_i-\widetilde{Y}_i\|_{L^1((0,T^{\star}/2);C^0(\mathbb{T}^2;\mathbb{R}^2))},
		\end{gather*}
		where
		\[
		0 < C \leq \max_{i \in \{1,\dots, M\}}\operatorname{e}^{\int_0^{T^{\star}/2} \|\overline{Y}_i(\cdot, s)\|_{C^1(\mathbb{T}^2;\mathbb{R}^2)} \, ds}.
		\]
		Furthermore, the Lipschitz continuity of the smooth vector field $\overline{Y}_i$ implies
		\begin{multline*}
			\|\overline{Y}_i-\widetilde{Y}_i\|_{L^1((0,T^{\star}/2);C^0(\mathbb{T}^2;\mathbb{R}^2))} \\
			\begin{aligned}
				& \leq \sum_{k=1}^{N_i} \int_{\mathcal{I}_k} \left( |1-\rho(t)| \sup_{x\in\mathbb{T}^2} |\overline{Y}_i(x,t)| + \sup_{x\in\mathbb{T}^2} |\overline{Y}_i(x,t) - \overline{Y}_i(x, kT^{\star}/2N_i)| \right) \, dt \\
				& \leq C\|\rho_i - 1\|_{L^1((0,T^{\star}/2); \mathbb{R})} + CN_i^{-1},
			\end{aligned}
		\end{multline*}
		where $C > 0$ denotes a generic constant that can depend on $\omegaup$ but is independent of all data in \Cref{theorem:main} and \Cref{theorem:main2}. Thus, to approximate $\Phi^{\overline{Y}_i}$ by $\Phi^{\widetilde{Y}_i}$ it suffices to take large~$N_i$, while ensuring that $\|\rho_i - 1\|_{L^1((0,T^{\star}/2); \mathbb{R})}$ is small. Hereby, the latter smallness is always attained by some smooth $\rho_i$ vanishing on a neighborhood of $\{0, T^{\star}/2N_i, T^{\star}/N_i, \dots, T^{\star}/2\}$.
		
		Inspecting \eqref{equation:TYi}, one finds that each $\widetilde{Y}_i$ belongs to $C^{\infty}_0((0,T^{\star}/2); \widetilde{\mathscr{H}}_{\omegaup})$ for a universal space~$\widetilde{\mathscr{H}}_{\omegaup} \subset C^{\infty}(\mathbb{T}^2;\mathbb{R}^2)$ that consists of divergence-free functions and has at most dimension $\widetilde{N}_{\omegaup} \coloneq \sum_{i=1}^M N_i$.

		\paragraph*{Building block 2.} All elements of $\widetilde{\mathscr{H}}_{\omegaup}$ are divergence-free and $\omegaup$ is without loss of generality simply-connected. Thus, given any $S \in \mathbb{R}^2$, one has the stream function representations $\widetilde{Y}_i(x - S, t) = \nabla^{\perp} \widetilde{\Psi}_{i,S}(x, t)$ with \smash{$\widetilde{\Psi}_{i,S} \in C^{\infty}_0((0,T^{\star}/2); \widetilde{\mathscr{H}}_{\omegaup, S})$}
		for all $(x,t) \in \omegaup\times[0, T^{\star}/2]$ and an at most \smash{$\widetilde{N}_{\omegaup}$}-dimensional vector space \smash{$\widetilde{\mathscr{H}}_{\omegaup, S} \subset C^{\infty}(\mathbb{T}^2; \mathbb{R})$}. 
		
		Now, let $\chi_0 \in C^{\infty}(\mathbb{T}^2;[0,1])$ be a cutoff with $\operatorname{supp}(\chi_0) \subset \omegaup$ and $\chi_0 = 1$ on a neighborhood of $\overline{\omegaup}_0$. Then, we define for $(x,t) \in \mathbb{T}^2\times[0, T^{\star}/2]$ the following time-reversed, localized, and shifted version of~$\widetilde{Y}_i$:
		\begin{equation*}\label{equation:WHYi}
			\begin{aligned}
				\widehat{Y}_i(x,t) & \coloneq -\sum_{k=1}^{N_i} \mathbb{I}_k(T^{\star}/2-t) \nabla^{\perp} [\chi_0\widetilde{\Psi}_{i,s_i^k}(\cdot,T^{\star}/2-t)](x) \\
				& \quad \, + \nabla^{\perp} [\chi_0(x) (\widetilde{s}_i^1(T^{\star}/2-t) x_1 + \widetilde{s}_i^2(T^{\star}/2-t) x_2)],
			\end{aligned}
		\end{equation*}
		where the parameters $s_i^1, \dots, s_i^{N_i} \in \mathbb{R}^2$ and $\widetilde{s}^1_i, \widetilde{s}^1_i \in C^{\infty}((0,T^{\star}/2); \mathbb{R})$ are chosen such that
		\begin{gather*}
			\Phi^{\widehat{Y}_i}(\Phi^{\widetilde{Y}_i}(B_i, 0, T^{\star}/2), 0,
			t) \subset \omegaup_0, \label{equation:ifp1}\\
			\Phi^{\widehat{Y}_i}(\Phi^{\widetilde{Y}_i}(x, 0, T^{\star}/2), 0,  T^{\star}/2) = x+S_i\label{equation:ifp2}
		\end{gather*}
		for $t \in [0, T^{\star}/2]$ and $x \in B_i$. The cutoff $\chi_0$ ensures that the flow is stationary away from~$\omegaup$ and that~$\widehat{Y}_i$ is a gradient in~$\mathbb{T}^2\setminus\omegaup$. 
		
		These choices of~$s_i^1, \dots, s_i^{N_i} \in \mathbb{R}^2$ and $\widetilde{s}^1_i, \widetilde{s}^2_i \in C^{\infty}((0,T^{\star}/2); \mathbb{R})$ are possible by the following reasoning.
		
		1) Due to \eqref{equation:npf}, for each $ k \in \{1,\dots,N_i\}$,
		information cannot be transported distances larger than $d_0/6$ along~$\widetilde{Y}_i$ on the time interval $\mathcal{I}_k$.
		Moreover, the function~$\rho$ is known to vanish on intervals
		\[
		[0, r_0], \quad [T^{\star}/2N_i-r_1,T^{\star}/2N_i+r_1], \quad \dots, \quad [T^{\star}/2-r_{N_i}, T^{\star}/2]
		\]
		for sufficiently small $r_0, r_1, \dots, r_{N_i} > 0$.
		Therefore, after transporting the set $\Phi^{\widetilde{Y}_i}(B_i, 0, T^{\star}/2)$ along the vector field~$-\widetilde{Y}_i(\cdot, T^{\star}/2-\cdot)$ for a duration of $T^{\star}/2N_i - r_{N_i-1}$, the resulting set, temporarily called~$A_i^1$, is still contained in $\omegaup_0$ and of diameter less than $d_0/6$. Hence, one can take $s_i^{N_i} = 0$. Then, one quickly pushes $A_i^1$ back inwards~$\omegaup_0$ by prescribing suitable nonzero values for $\widetilde{s}^1_i$ and $\widetilde{s}^2_i$ on an interval
		\[
		(a^1_i,b^1_i)\subset(T^{\star}/2-T^{\star}/2N_i, T^{\star}/2-T^{\star}/2N_i+r_{N_i-1}).
		\]
		Notably, the choice of $\rho_i$ appearing in \eqref{equation:TYi} ensures that $\rho_i(t) = 0$ for all $t \in [a^1_i,b^1_i]$. The values $s_i^{N_i-1} \in \mathbb{R}^2$ are subsequently fixed such that
		\begin{equation*}\label{equation:sp1}
			\Phi^{\widehat{y}_i}(A_i^1,b^1_i,a^1_i) = A_i^1 + s_i^{N_i-1} = \{a + s_i^{N_i-1} \, | \, a \in A_i^1 \} \subset \omegaup_0,
		\end{equation*}
		where
		\[
		\widehat{y}_i(x,t) \coloneq \nabla^{\perp} [\chi_0(x) (\widetilde{s}_i^1(t) x_1 + \widetilde{s}_i^2(t) x_2)]
		\]
		for $(x, t) \in \mathbb{T}^2\times[0,T^{\star}/2]$. Hereby, the existence of such $s_i^{N_i-1}$ follows from the fact that $\widehat{y}_i$ is constant in $\overline{\omegaup}_0$ with respect to the space variables; its' flow rigidly translates~$A_i^1$ within $\omegaup_0$.
		Moreover, we can assume that the values of $(\widetilde{s}^1_i, \widetilde{s}^2_i)$ on $(a^1_i,b^1_i)$ are fixed so that
		\begin{gather*}
			\operatorname{dist}(\Phi^{\widehat{y}_i}(A_i^1,b^1_i,a^1_i), \partial \omegaup_0) > d_0/3, \\
			\forall s \in [T^{\star}/2-T^{\star}/2N_i, T^{\star}/2]\colon \Phi^{\widehat{y}_i}(A_i^1,T^{\star}/2, s) \subset \omegaup_0.
		\end{gather*}
		
		2) Starting at $t = T^{\star}/2N_i$, the set~$A_i^1 + s_i^{N_i-1}$ is transported along the vector field $t \mapsto-\widetilde{Y}_i(\cdot-s_i^{N_i-1}, T^{\star}/2-t)$ for a duration of $T^{\star}/2N_i - r_{N_i-2}$ to a set $A_i^2$. From here, the above idea is repeated iteratively until the shifts are defined on all intervals $\mathcal{I}_k$ for $k \in \{1,\dots,N_i\}$.
		
		\paragraph*{Dimensions.} The constructions ensure that $\widetilde{Y}_i, \widehat{Y}_i \in C^{\infty}_0((0,T^{\star}/2); \mathscr{H}_{\omegaup})$ for an at most $D_{\omegaup}$-dimensional vector space $\mathscr{H}_{\omegaup} \subset C^{\infty}(\mathbb{T}^2; \mathbb{R}^2)$, where
		\[
		D_{\omegaup} \coloneq  \widetilde{N}_{\omegaup} + \left( \sum_{i=1}^M N_i \right)^2 + 2.
		\]
		More precisely, the space~$\mathscr{H}_{\omegaup}$ is spanned by the following functions: 1) the elements of $\widetilde{\mathscr{H}}_{\omegaup}$; 2) the functions $\nabla^{\perp} (\chi_0 \psi)$ with $\psi \in \widetilde{\mathscr{H}}_{\omegaup, s^k_l}$, $k \in \{1,\dots, N_l\}$ and $l \in \{1,\dots,M\}$; 3) the two profiles $\nabla^{\perp}[\chi_0 x_1]$ and $\nabla^{\perp}[\chi_0 x_2]$.
		
		\paragraph*{Conclusion of the proof.} A function $\overline{y}$ with the desired properties is now defined by zero on $[0, t^0_c] \cup [t^M_c ,1]$, by $\widetilde{Y}_i(\cdot, t-t^{i-1}_c)$ on $[t^{i-1}_c,  t^{i-1}_c + T^{\star}/2]$, by $\widehat{Y}_i(\cdot, t - t^{i-1}_c - T^{\star}/2)$ on $[t^{i-1}_c + T^{\star}/2, t^i_a]$, by zero on $[t^i_a, t^i_b]$, by $-\widehat{Y}_i(\cdot, T^{\star}/2 + t^i_b - t)$ on $[t^i_b, t^i_b + T^{\star}/2]$, and by $-\widetilde{Y}_i(\cdot, T^{\star} + t^i_b - t)$ on $[t^i_b + T^{\star}/2, t^i_c]$.
	\end{proof}

	\subsection{Modified generating vector field \texorpdfstring{$\overline{u}^{\star}$}{}}\label{subsection:driftUbar}
	For the sake of explicitness, we fix now in \Cref{definition:ovf} the natural example $\mathcal{K} = \{(1,0), (0,1)\} \subset \mathbb{Z}^2 \setminus \{0\}$. Then, the space $\mathscr{H}$ from \eqref{equation:defH} is four-dimensional and given by
	\begin{equation}\label{equation:choiceK}
		\mathscr{H} = \operatorname{span}_{\mathbb{R}}\left\{ x \mapsto \sin(x_1), x \mapsto \sin(x_2), x \mapsto \cos(x_1), x \mapsto \cos(x_2)  \right\}.
	\end{equation}
	All choices of finite $\mathcal{K} \subset \mathbb{Z}^2 \setminus \{0\}$ with $\operatorfont{span}_{\mathbb{Z}}(\mathcal{K}) = \mathbb{Z}^2$ are allowed. But, for different~$\mathcal{K}$ the control spaces obtained in the end may be different, as well. 
	
	In view of \Cref{remark:kappa}, by taking $|\kappa| > 0$ in \eqref{equation:ObservableVF} small, we select in the sense of  \Cref{definition:ovf} with $T = T^{\star}/2$ a divergence-free generating vector field
	\[
		\widetilde{u}^{\star}\colon \mathbb{T}^2\times[0, 	T^{\star}/2]\longrightarrow\mathbb{R}^2
	\]
	such that
	\begin{gather}\label{equation:propchimu1}
		\bigcup\limits_{\substack{s,t\in[0,T^{\star}/2], \\ S \in \mathbb{R}^2}} \Phi^{\widetilde{u}^{\star}(\cdot-S,\cdot)}(\operatorname{supp}(\mu),s,t) \subset \mathcal{O},
	\end{gather}
	where $\mu$, $\chi$, and $\mathcal{O}$ are fixed via \Cref{lemma:retf} as described in~\Cref{subsection:convection}. This is possible as $\operatorname{supp}(\mu) \subset \mathcal{O}$ and $\mathcal{O}$ is open.
	\begin{rmrk}
		The choice of the parameter~$\kappa$, which ensures \eqref{equation:propchimu1}, is universal; it depends only on $\omegaup$, and thus is independent of all prescribed data in \Cref{theorem:main} and~\Cref{theorem:main2}. The importance of this choice will become apparent in the proofs of \Cref{lemma:f1} and \Cref{theorem:locfinthm}.
	\end{rmrk}
	Because the function~$\phi$ appearing in \Cref{definition:ovf} with $T = T^{\star}/2$ satisfies $\phi(T^{\star}/2) = 0$, we have $\widetilde{u}^{\star}(\cdot, 0) = \widetilde{u}^{\star}(\cdot, T^{\star}/2) = 0$.
	This allows us now to define the profile
	\begin{equation}\label{equation:defustar}
		\overline{u}^{\star}(x,t) \coloneq \begin{cases}
			\widetilde{u}^{\star}(x,t) & \mbox{ if } t \in [0,T^{\star}/2],\\
			-\widetilde{u}^{\star}(x,T^{\star}-t) & \mbox{ if } t \in [T^{\star}/2,  T^{\star}],
		\end{cases}
	\end{equation}
	noting that
	\begin{equation}\label{equation:rpus}
		\Phi^{\overline{u}^{\star}}(x,0,t) = \Phi^{\overline{u}^{\star}}(x,0,T^{\star}-t)
	\end{equation}
	for all $(x,t) \in \mathbb{T}^2\times[0,T^{\star}]$. Indeed, both sides in \eqref{equation:rpus} are equal at $t = T^{\star}/2$ and solve the same well-posed differential equation.
	
	The statement of \Cref{lemma:lineareverywheresupported_cpam} remains true when considering convection along~$\overline{u}^{\star}$ instead of~$\overline{u}$ defined via \Cref{definition:ovf}.
	\begin{lmm}\label{lemma:lineareverywheresupported}
		Given $m \in \mathbb{N}$, $v_0, v_1 \in H^m$, and $\varepsilon > 0$, there exists a control $g^{\star} \in L^2((0,T^{\star}); \mathscr{H})$ such that the solution $v \in C^0([0,T^{\star}];H^m)$ to
		\begin{equation}\label{equation:laes1}
			\begin{gathered}
				\partial_t v + (\overline{u}^{\star} \cdot \nabla) v = g^{\star}, \\
				v(\cdot, 0) = v_0
			\end{gathered}
		\end{equation}
		satisfies
		\begin{equation}\label{equation:laes2}
			\|v(\cdot, T^{\star})-v_1\|_{m} < \varepsilon.
		\end{equation}
	\end{lmm}
	\begin{proof}
		First, we assume $v_0 = 0$ and set $g^{\star}(\cdot, t) = 0$ for $t \in [0, T^{\star}/2]$. As a result, it holds $v(\cdot, T^{\star}/2) = 0$. To determine~$g^{\star}(\cdot, t)$ for $t \in (T^{\star}/2, T^{\star}]$, we apply \Cref{lemma:lineareverywheresupported_cpam} with $T = T^{\star}/2$ and target state $v_1$. The general case $v_0 \neq 0$ follows from a linear superposition principle: add the uncontrolled solution~$\widetilde{v}$ with initial state~$v_0$ to a controlled solution $\widehat{v}$ with zero initial state and target state $v_1 - v_0$, noting that $\widetilde{v}(\cdot, T^{\star}) = v_0$ by \eqref{equation:rpus}.
	\end{proof}
	
	\begin{rmrk}\label{remark:bop}
		It is known, {\it e.g.}, as explained in \cite{NersesyanRissel2024a} or \cite[Proof of Theorem 2.3]{Nersesyan2021}, that for given $\varepsilon > 0$ and a bounded subset $B$ of $H^{m+1}$ with $v_0, v_1 \in B$, one can choose the control in \Cref{lemma:lineareverywheresupported} of the form $g^{\star} = \mathcal{L}_{\varepsilon}(v_1 - v_0)$ with a bounded linear operator $\mathcal{L}_{\varepsilon}\colon H^{m}\longrightarrow L^2((0,T^{\star});\mathscr{H})$, as long as \eqref{equation:laes2} is replaced by
		\begin{equation*}
			\|v(\cdot, T^{\star})-v_1\|_{m} < \varepsilon \|v_0 -v_1\|_{m+1}.
		\end{equation*}
		Let us briefly recall the argument when $v_0 = 0$; the general case follows by superposition as in the proof of \Cref{lemma:lineareverywheresupported}. Hereto, consider the resolving operator $\mathcal{A}$ associating with $g^{\star} \in L^2((0,T^{\star});\mathscr{H})$ the solution to $\partial_t v + (\overline{u}^{\star} \cdot \nabla) v = g^{\star}$ with zero initial data $v(\cdot, 0) = 0$.
		Further, denote by $\mathcal{A}_{T^{\star}}$ the restriction $g \mapsto (\mathcal{A}g)(T^{\star}) \in H^m$. In particular, due to \Cref{lemma:lineareverywheresupported}, the range of $\mathcal{A}_{T^{\star}}$ is dense in $H^m$. Hence, by utilizing \cite[Proposition 2.6]{KuksinNersesyanShirikyan2020}, the desired operator $\mathcal{L}_{\varepsilon}$ can be chosen as a continuous approximate right inverse of $\mathcal{A}_{T^{\star}}$.
	\end{rmrk}

	By the Helmholtz-Hodge decomposition, the vector field $\overline{u}^{\star}$ admits a stream function $\overline{\phi}^{\star}$ in $\mathbb{T}^2\times[0,T^{\star}]$. Indeed, from \cref{equation:ObservableVF,equation:defustar} it follows that $\overline{u}^{\star}$ is divergence-free and has zero average. Namely,
	\begin{equation}\label{equation:strfus}
		\overline{u}^{\star}(x,t) = \nabla^{\perp}\overline{\phi}^{\star}(x,t) = \nabla^{\perp} \kappa\sum_{\ell \in \mathcal{K}} \left(\psi_{\ell}^{\operatorname{c},\star}(t) s_{\ell}(x) - \psi_{\ell}^{\operatorname{s},\star}(t) c_{\ell}(x)\right)
	\end{equation}
	for $(x,t) \in \mathbb{T}^2\times[0,T^{\star}]$ and $\psi_{\ell}^{\operatorname{s},\star}, \psi_{\ell}^{\operatorname{c},\star}$ obtained in the sense of \cref{equation:ObservableVF,equation:defustar}. Moreover, recall that we fixed in \Cref{subsection:convection} the translation vectors~$(S_i)_{i \in \{1,\dots,M\}}\subset\mathbb{R}^2$ such that
	\begin{equation}\label{equation:Si}
		\mathcal{O} = \mathcal{O}_i + S_i
	\end{equation} 
	for each $i \in \{1,\dots,M\}$.
	Then, by the definition of $\overline{u}^{\star}$ in \eqref{equation:defustar}, one has analogues of \eqref{equation:rpus} for flows arising from certain variations of the vector field
	\[
	(x, t) \mapsto \nabla^{\perp}[\chi(x) \overline{\phi}^{\star}(x,t)],
	\]
	where~$\chi$ is the cutoff introduced in \Cref{subsection:convection}. For instance, it holds
	\begin{equation}\label{equation:rpus3}
		\begin{gathered}
			\Phi^{\nabla^{\perp}[\chi(\cdot) \overline{\phi}^{\star}(\cdot-S,\cdot)]}(x,0,s) = \Phi^{\nabla^{\perp}[\chi(\cdot) \overline{\phi}^{\star}(\cdot-S,\cdot)]}(x,0,T^{\star}-s),\\
			\Phi^{\nabla^{\perp}[\chi (\cdot) \overline{\phi}^{\star}(\cdot-S_i, \cdot - t^i_a)]}(x,t^i_a,t) = \Phi^{\nabla^{\perp}[\chi (\cdot) \overline{\phi}^{\star}(\cdot-S_i, \cdot - t^i_a)]}(x,t^i_a,t^i_b-(t-t^i_a))
		\end{gathered}
	\end{equation}
	for all $i \in \{1,\dots,M\}$, $S \in \mathbb{R}^2$, $x \in \mathbb{T}^2$, $s \in [0,T^{\star}]$, and $t \in [t^i_a, t^i_b]$. Indeed, both sides in each line of \eqref{equation:rpus3} satisfy the same differential equations with identical states at $s = T^{\star}/2$ and $t = t^i_a+T^{\star}/2$, respectively.
	
	The next lemma is a consequence of \eqref{equation:propchimu1} for sufficiently small $|\kappa| > 0$. However, in view of \Cref{remark:kappa} and \eqref{equation:strfus}, it is shorter to argue that the lemma follows from a new (smaller) choice of $|\kappa| > 0$ which depends only on $\omegaup$. 
	\begin{lmm}\label{lemma:propchimu1B}
		There exists $r \in (0,L)$ with $\chi = 1$ on the $r$-neighborhood $\mathcal{N}_r$ of the reference square~$\mathcal{O}$ with side-length $L$ (defined in \Cref{subsection:convection}) and such that~$\mathcal{N}_r$ is also a neighborhood of $\bigcup_{s,t\in[0,T^{\star}], \,S \in \mathbb{R}^2} \Phi^{\nabla^{\perp}[\chi(\cdot)\overline{\phi}^{\star}(\cdot-S,\cdot)]}(\operatorname{supp}(\mu),s,t)$.
	\end{lmm}

	\subsection{Definition of \texorpdfstring{$\overline{U}$ based on $\overline{y}$ and $\overline{u}^{\star}$}{the main convection profile}}\label{subsection:defbU}
	The following function $\overline{U} \in W^{1,2}((0,1);C^{\infty}(\mathbb{T}^2;\mathbb{R}^2))$ will be used in \Cref{subsection:plfc} as a convection profile for linear transport equations steered~by physically localized and finite-dimensional controls:
	\begin{equation}\label{equation:Ubar}
		\overline{U}(x,t)\coloneq \overline{y}(x,t) + \sum_{i=1}^M \mathbb{I}_{[t^i_a, t^i_b]}(t) \nabla^{\perp}[\chi(x)\overline{\phi}^{\star}(x-S_i, t-t^i_a)]
	\end{equation}
	for $(x,t) \in \mathbb{T}^2\times[0,1]$. Aside of being divergence-free in $\mathbb{T}^2$, the vector field $\overline{U}(\cdot, t)$ is a gradient in a neighborhood of $\mathbb{T}^2\setminus \omegaup$ for each $t \in [0,1]$. Moreover, due to \cref{equation:propybar1,equation:propybar2,equation:rpus3}, the flow of $\overline{U}$ satisfies
	\begin{equation}\label{equation:clpropU}
		\begin{aligned}
			\Phi^{\overline{U}}(x,0, t) = \Phi^{\overline{U}}(x, t^M_c, t) =  \Phi^{\overline{U}}(x,t^l_a, t^l_b) = \Phi^{\overline{U}}(x,t^{l-1}_c, t^{l}_c) = \Phi^{\overline{U}}(x,0,1) = x
		\end{aligned}
	\end{equation}
	for all $x \in \mathbb{T}^2$, $t \in [0, T^{\star}]$, and $l \in \{1,\dots,M\}$. 
	
	The following two lemmas provide basic properties of the considered flows related to~$\overline{U}$.
	\begin{lmm}\label{lemma:f1}
		Given $x \in \mathbb{T}^2$ and $i \in \{1,\dots,M\}$ such that $\mu(\Phi^{\overline{U}}(x,0,t^i_a+s)) \neq 0$ or $\mu(\Phi^{\overline{u}^{\star}(\cdot-S_i, \cdot)}(x+S_i,0,s)) \neq 0$ are satisfied for at least one~$s \in [0,T^{\star}]$, it holds
		\[
		\Phi^{\overline{U}}(x,0,t^i_a+t) = \Phi^{\overline{u}^{\star}(\cdot-S_i, \cdot)}(x+S_i,0,t)
		\]
		for all $t \in [0, T^{\star}]$.
	\end{lmm}
	\begin{proof}
		By \cref{equation:Si,equation:clpropU,equation:Ubar} and \Cref{lemma:retf}, one has $\Phi^{\overline{U}}(z,0,t^i_a) = z + S_i$ for each $z$ in the $L$-neighborhood of $\mathcal{O}_i$. 
		Moreover, from the definition of $\overline{U}$ in \eqref{equation:Ubar}, one can infer
		\[
		\Phi^{\overline{U}}(z,t^i_a+r,t^i_a + t) = \Phi^{\nabla^{\perp}[\chi(\cdot)\overline{\phi}^{\star}(\cdot-S_i, \cdot)]}(z,r,t)
		\]
		for all $z \in \mathbb{T}^2$ and $r, t \in [0,T^{\star}]$. 
		
		\paragraph*{Case 1.} If there exists a number $s \in [0,T^{\star}]$ such that $\mu(\Phi^{\overline{U}}(x,0,t^i_a+s)) \neq 0$, by \Cref{lemma:propchimu1B} this means that $\Phi^{\overline{U}}(x,0,t^i_a)$ lies in a $L$-neighborhood of $\mathcal{O}$. As a consequence of \Cref{lemma:retf}, the point $x$ then belongs to an~$L$-neighborhood of~$\mathcal{O}_i$, which yields $\Phi^{\nabla^{\perp}[\chi(\cdot)\overline{\phi}^{\star}(\cdot-S_i, \cdot)]}(x+S_i,0,s) \in \operatorname{supp}(\mu)$ because of
		\begin{equation*}
			\begin{aligned}
				\operatorname{supp}(\mu) & \ni \Phi^{\overline{U}}(x,0,t^i_a+s) \\
				& = \Phi^{\overline{U}}(\Phi^{\overline{U}}(x,0,t^i_a),t^i_a,t^i_a+ s) \\
				& = \Phi^{\overline{U}}(x+S_i,t^i_a,t^i_a + s) \\
				& = \Phi^{\nabla^{\perp}[\chi(\cdot)\overline{\phi}^{\star}(\cdot-S_i, \cdot)]}(x+S_i,0,s).
			\end{aligned}
		\end{equation*}
		Therefore, it follows from \Cref{lemma:propchimu1B} that $\chi(\Phi^{\nabla^{\perp}[\chi(\cdot)\overline{\phi}^{\star}(\cdot-S_i, \cdot)]}(x+S_i,0,t)) = 1$
		for all $t \in [0, T^{\star}]$. As a result, the well-posed problem $\Phi'(t) = \overline{u}^{\star}(\Phi(t)-S_i, t)$ with~$\Phi(0) = x+S_i$ is satisfied on $[0, T^{\star}]$ by $ t \mapsto \Phi^{\overline{u}^{\star}(\cdot-S_i, \cdot)}(x+S_i,0,t)$ and also by $t \mapsto \Phi^{\nabla^{\perp}[\chi(\cdot)\overline{\phi}^{\star}(\cdot-S_i, \cdot)]}(x+S_i,0,t)$; thus, these two functions are equal.

		\paragraph*{Case 2.} If there exists $s \in [0,T^{\star}]$ such that $\mu(\Phi^{\overline{u}^{\star}(\cdot-S_i, \cdot)}(x+S_i,0,s)) \neq 0$, then \cref{equation:propchimu1,equation:defustar} ensure that $\chi(z) = 1$ for all $z$ from a neighborhood of $\cup_{t\in[0,T^{\star}]}\Phi^{\overline{u}^{\star}(\cdot-S_i, \cdot)}(x+S_i,0,t)$.
		This implies
		\[
		\Phi^{\overline{u}^{\star}(\cdot-S_i, \cdot)}(x+S_i,0,t) = \Phi^{\nabla^{\perp}[\chi(\cdot)\overline{\phi}^{\star}(\cdot-S_i, \cdot)]}(x+S_i,0,t)
		\]
		for $t \in [0, T^{\star}]$. Hence, the assertion can be concluded by analysis similar to the previous case.
	\end{proof}
	
	\begin{lmm}\label{lemma:f2}
		Let $T > 0$ and $v \colon \mathbb{T}^2\times[0,T]\longrightarrow\mathbb{R}^2$ sufficiently regular such that the flow $\Phi^{v}$ is well-defined. Then, one has
		\[
		\Phi^{v(\cdot-S, \cdot)}(x+S,0,t) = \Phi^{v}(x,0,t)+S
		\]
		for all $S \in \mathbb{R}^2$ and~$(x,t) \in \mathbb{T}^2\times[0,T]$.
	\end{lmm}

	\subsection{Controllability of \texorpdfstring{convection along $\overline{U}$}{the main convection problem}}\label{subsection:plfc}
	The goal of this section is to demonstrate approximate controllability of the linear transport equation
	\[
	\partial_t V + (\overline{U} \cdot \nabla) V = G,
	\]
	where $\overline{U}$ is defined via \eqref{equation:Ubar} and $G$ denotes a finite-dimensional physically localized control.

	\begin{thrm}\label{theorem:locfinthm}
		There exists a finite-dimensional space $\mathscr{F}_{\mathscr{t}}$ consisting of smooth zero average functions $\mathbb{T}^2\longrightarrow\mathbb{R}$ supported in $\omegaup$ such that the following statement holds. Given any $m \in \mathbb{N}$, $v_0, v_1 \in H^m$, and $\varepsilon > 0$, there is a control $G \in L^2((0,1); \mathscr{F}_{\mathscr{t}})$ such that the solution $V \in C^0([0,1];H^m)$ to
		\begin{equation}\label{equation:thmp1}
			\begin{gathered}
				\partial_t V + (\overline{U} \cdot \nabla) V = G, \\
				V(\cdot, 0) = v_0
			\end{gathered}
		\end{equation}
		satisfies
		\begin{equation}\label{equation:thmp2}
			\|V(\cdot, 1) - v_1\|_m < \varepsilon.
		\end{equation}
	\end{thrm}
	\begin{proof}
		The first step is to define a finite-dimensional and physically localized auxiliary control, which however fails to be of zero average. In the second step, the average is corrected. The last step will summarize how the universal space~$\mathscr{F}_{\mathscr{t}}$ arises from the foregoing constructions.
		\paragraph*{Step 1. Auxiliary control.}
		Let the functions $v \in C^0([0,T^{\star}];H^m(\mathbb{T}^2;\mathbb{R}))$ and $g^{\star}\in L^2((0,T^{\star}); \mathscr{H})$ be obtained by \Cref{lemma:lineareverywheresupported} such that~$v$ solves the transport equation \eqref{equation:laes1} with control~$g^{\star}$ and satisfies 
		\begin{equation}\label{equation:pfcgv}
			\|v(\cdot, T^{\star})-v_1\|_{m} < \varepsilon.
		\end{equation}
		Then, we define for all $(x,t) \in \mathbb{T}^2 \times [0,1]$ the auxiliary control
		\begin{equation}\label{equation:Gtilde}
			\widetilde{G}(x,t) \coloneq \mu(x) \sum_{i=1}^M  \mathbb{I}_{[t^i_a, t^i_b]}(t) g^{\star}(x-S_i, t-t^i_a), 
		\end{equation}
		where $\mu$ with $\operatorname{supp}(\mu) \subset \mathcal{O} \subset \omegaup$ is the cutoff from \Cref{subsection:convection} and $S_1, \dots, S_M$ satisfy \eqref{equation:Si}. 
		
		Associated with the control~$\widetilde{G}$, let $\widetilde{V}$ be the solution to $\partial_t \widetilde{V} + (\overline{U} \cdot \nabla) \widetilde{V} = \widetilde{G}$ with initial condition $\widetilde{V}(\cdot, 0) = v_0$.
		By the method of characteristics and Duhamel's principle, it holds
		\[
		\widetilde{V}(x, t) = v_0(\Phi^{\overline{U}}(x,t,0)) + \int_0^t \widetilde{G}(\Phi^{\overline{U}}(x,t,s),s) \, ds
		\]
		for all $(x,t) \in \mathbb{T}^2 \times [0,1]$. Now, let $x \in \mathbb{T}^2$ be arbitrary and note that $\Phi^{\overline{U}}(x,1,0) = x$ holds due to \eqref{equation:clpropU}. Therefore, 
		\begin{equation*}
			\begin{aligned}
				\widetilde{V}(x, 1) & = v_0(x) + \int_0^1 \widetilde{G}(\Phi^{\overline{U}}(x,1,s),s) \, ds \\
				& = v_0(x) + \sum_{i=1}^M \int_0^1 \mu(\Phi^{\overline{U}}(x,1,s)) \mathbb{I}_{[t^i_a, t^i_b]}(s) g^{\star}(\Phi^{\overline{U}}(x,1,s)-S_i,s-t^i_a) \, ds.
			\end{aligned}
		\end{equation*}
		Again by \eqref{equation:clpropU}, it follows that
		\begin{align*}
			\widetilde{V}(x, 1) & = v_0(x) + \sum_{i=1}^M \! \int_0^1 \!\! \! \mu(\Phi^{\overline{U}}(x,0,s)) \mathbb{I}_{[t^i_a, t^i_b]}(s) g^{\star}(\Phi^{\overline{U}}(x,0,s)-S_i,s-t^i_a) \, ds \\
			& = v_0(x) + \sum_{i=1}^M \int_{t^i_a}^{t^i_b} \mu(\Phi^{\overline{U}}(x,0,s)) g^{\star}(\Phi^{\overline{U}}(x,0,s)-S_i,s-t^i_a) \, ds,
		\end{align*}
		which is due to a change of variables under the integral sign equal to
		\begin{equation*}
			v_0(x) + \sum_{i=1}^M \int_{0}^{T^{\star}} \mu(\Phi^{\overline{U}}(x,0,s+t^i_a)) g^{\star}(\Phi^{\overline{U}}(x,0,s+t^i_a)-S_i,s) \, ds.
		\end{equation*}
		In view of \Cref{lemma:f1}, the previous expression equals
		\begin{equation*}
			v_0(x) + \sum_{i=1}^M \int_{0}^{T^{\star}} \!\! \mu(\Phi^{\overline{u}^{\star}(\cdot-S_i,\cdot)}(x+S_i,0,s)) g^{\star}(\Phi^{\overline{u}^{\star}(\cdot-S_i,\cdot)}(x+S_i,0,s)-S_i,s) \, ds,
		\end{equation*}
		and therefore \Cref{lemma:f2} allows to infer that
		\begin{equation*}
			\widetilde{V}(x, 1) = v_0(x) + \sum_{i=1}^M \int_{0}^{T^{\star}} \mu(\Phi^{\overline{u}^{\star}}(x,0,s)+S_i) g^{\star}(\Phi^{\overline{u}^{\star}}(x,0,s)+S_i-S_i,s) \, ds.
		\end{equation*}
		Recalling that $(\mu_i)_{i\in\{1,\dots,M\}}$ is the partition of unity from \Cref{subsection:convection} and that~$\overline{u}^{\star}$ satisfies \eqref{equation:rpus}, one obtains
		\begin{align*}
			\widetilde{V}(x, 1) & = v_0(x) + \sum_{i=1}^M \int_{0}^{T^{\star}} \mu_i(\Phi^{\overline{u}^{\star}}(x,0,s)) g^{\star}(\Phi^{\overline{u}^{\star}}(x,0,s),s) \, ds \\
			& = v_0(\Phi^{\overline{u}^{\star}}(x,T^{\star},0)) + \int_{0}^{T^{\star}} g^{\star}(\Phi^{\overline{u}^{\star}}(x,T^{\star},s),s) \, ds.
		\end{align*}
		This demonstrates that $\widetilde{V}(x, 1) = v(x, T^{\star})$, and because~$v$ satisfies~\eqref{equation:pfcgv}, one arrives at
		\[
		\|\widetilde{V}(\cdot, 1) - v_1\|_m = \|v(\cdot, T^{\star}) - v_1\|_m < \varepsilon.
		\]
		\paragraph*{Step 2. Control with zero average.}
		It remains to define suitable modifications of $\widetilde{V}$ and $\widetilde{G}$ with zero average. Since $\widetilde{V}(\cdot, 1) = v(\cdot, T^{\star})$ has zero average ($v$ is the function from \Cref{lemma:lineareverywheresupported}), we introduce
		\[
			V(x,t) \coloneq \widetilde{V}(x,t) - \frac{\mu(x)\int_{\mathbb{T}^2} \widetilde{V}(x,t) \, dx}{\int_{\mathbb{T}^2} \mu(x) \, dx},
		\]
		which by construction satisfies together with the modified control
		\begin{equation}\label{equation:G}
			G(x,t) \coloneq \widetilde{G}(x,t) - \frac{\frac{d}{dt}\int_{\mathbb{T}^2} \widetilde{V}(z,t) \, dz}{\int_{\mathbb{T}^2} \mu(z) \, dz} \mu(x) - \frac{\int_{\mathbb{T}^2} \widetilde{V}(z,t) \, dz}{\int_{\mathbb{T}^2} \mu(z) \, dz} (\overline{U}(x,t) \cdot \nabla)\mu(x)
		\end{equation}
		the transport equation \eqref{equation:thmp1} and the condition \eqref{equation:thmp2}.
		
		\paragraph*{{Step 3. The space $\mathscr{F}_{\mathscr{t}}$.}}
		By \Cref{definition:ovf}, \Cref{lemma:retf}, \cref{equation:choiceK,equation:defustar,equation:Ubar,equation:strfus}, there exist smooth universal functions $U_1, \dots, U_{D_{\omegaup} + 4} \in C^{\infty}(\mathbb{T}^2; \mathbb{R}^2)$ such that
		\begin{equation}\label{equation:repUbarfd}
			\overline{U}(x,t) = \sum_{i=1}^{D_{\omegaup} + 4} u_i(t) U_i(x)
		\end{equation}
		for coefficients $u_1, \dots, u_{D_{\omegaup} + 4} \in L^2((0,1); \mathbb{R})$. More specifically, one has $\overline{y} \in C^{\infty}_0((0,1); \mathscr{H}_{\omegaup})$ for an at most $D_{\omegaup}$-dimensional space $\mathscr{H}_{\omegaup} \subset C^{\infty}(\mathbb{T}^2; \mathbb{R}^2)$; hence, we can choose $U_1, \dots, U_{D_{\omegaup}} \in C^{\infty}(\mathbb{T}^2; \mathbb{R}^2)$ so that $\overline{y}(x,t) = \sum_{i=1}^{D_{\omegaup}} u_i(t) U_i(x)$.
		Further, by \eqref{equation:strfus} and trigonometric identities, the stream function~$\overline{\phi}^{\star}$ appearing in~\eqref{equation:Ubar} has the four-dimensional representation
		\[
			\overline{\phi}^{\star}(x,t) = \phi_1(t) \sin(x_1) + \phi_2(t) \sin(x_2) + \phi_3(t) \cos(x_1) + \phi_4(t) \cos(x_2),
		\]
		with $\phi_1, \dots, \phi_4 \in L^2((0,1); \mathbb{R})$. Therefore, one can choose
		\begin{gather*}
			U_{D_{\omegaup} + 1}(x) = \nabla^{\perp}[\chi(x)\sin(x_1)], \quad U_{D_{\omegaup} + 2}(x) = \nabla^{\perp}[\chi(x)\sin(x_2)],\\
			U_{D_{\omegaup} + 3}(x) = \nabla^{\perp}[\chi(x)\cos(x_1)], \quad U_{D_{\omegaup} + 4}(x) = \nabla^{\perp}[\chi(x)\cos(x_2)].
		\end{gather*}
		
		In a similar manner, from \eqref{equation:choiceK} it follows that there are universal functions $\widetilde{G}_1, \dots, \widetilde{G}_{4} \in C^{\infty}(\mathbb{T}^2; \mathbb{R}^2)$ with $\operatorname{supp}(\widetilde{G}_i) \subset \omegaup$ for $i \in \{1,\dots,4\}$ such that~$\widetilde{G}$ from \eqref{equation:Gtilde} is with coefficients $\widetilde{g}_1, \dots, \widetilde{g}_{4} \in L^2((0,1); \mathbb{R})$ of the form $\widetilde{G}(x,t) = \sum_{i=1}^{4} \widetilde{g}_i(t) \widetilde{G}_i(x)$.
		In order to replace $\widetilde{G}_1, \dots, \widetilde{G}_{4}$ with zero average versions, we first expand $\widetilde{G}$ in the following way
		\[
		\widetilde{G}(x,t) = \sum_{i=1}^{4} \widetilde{g}_i(t) \widehat{G}_i(x) + \sum_{i=1}^{4} \frac{\widetilde{g}_i(t) \int_{\mathbb{T}^2} \widetilde{G}_i(z) \, dz}{\int_{\mathbb{T}^2} \mu(z) \, dz} \mu(x),
		\]
		where
		\[
		\widehat{G}_i(x) \coloneq \widetilde{G}_i(x) - \frac{\int_{\mathbb{T}^2} \widetilde{G}_i(z) \, dz}{\int_{\mathbb{T}^2} \mu(z) \, dz} \mu(x)
		\]
		for $i \in \{1, \dots, 4\}$. Since the right-hand side in \eqref{equation:G} and its last term both have zero average, it follows that
		\[
		\sum_{i=1}^{4} \frac{\widetilde{g}_i(t) \int_{\mathbb{T}^2} \widetilde{G}_i(z) \, dz}{\int_{\mathbb{T}^2} \mu(z) \, dz} \mu(x) = \frac{\frac{d}{dt}\int_{\mathbb{T}^2} \widetilde{V}(z,t) \, dz}{\int_{\mathbb{T}^2} \mu(z) \, dz} \mu(x).
		\]
		Plugging this into \eqref{equation:G}, we obtain profiles $G_1, \dots, G_{D_{\omegaup} + 8} \in C^{\infty}(\mathbb{T}^2; \mathbb{R}^2)$ that have zero average and satisfy $\operatorname{supp}(G_i) \subset \omegaup$ for all $i \in \{1,\dots,D_{\omegaup} + 8\}$ such that $G(x,t) = \sum_{i=1}^{D_{\omegaup} + 8} g_i(t) G_i(x)$
		holds with coefficients $g_1, \dots, g_{D_{\omegaup} + 8} \in L^2((0,1); \mathbb{R})$. Thus, the space $\mathscr{F}_{\mathscr{t}}$ can be taken as the span of $G_1, \dots, G_{D_{\omegaup} + 8}$.
	\end{proof}
	
	\begin{rmrk}\label{remark:tempav}
		The result of \Cref{theorem:locfinthm} remains true for states with nonzero average, as long as one adds to $\mathscr{F}_{\mathscr{t}}$ the one-dimensional space spanned by the smooth cutoff $\mu$ from \Cref{subsection:convection} (or by any other smooth nonzero average function supported in $\omegaup$). Indeed, if $\tau_0$ denotes the average of~$v_0$ and~$\tau_1$ is the average of~$v_1$, one can take $\tau \in C^{\infty}_0((0,1);\mathbb{R})$ such that $\smallint_0^1 \tau(s) \, ds = 1$ and observe that the solution to
		\[
		\partial_t\widetilde{v} + (\overline{U}\cdot\nabla)\widetilde{v} = \zeta \coloneq \frac{\tau (\tau_1 - \tau_0)}{\int_{\mathbb{T}^2} \mu(x) \, dx}\mu
		\]
		with initial condition $\widetilde{v}(\cdot, 0) = \tau_0$ satisfies $\smallint_{\mathbb{T}^2} \widetilde{v}(x, 1) \, dx = \tau_1$. Subsequently, one obtains a zero average control $\widehat{G}$ by applying \Cref{theorem:locfinthm} with initial state $v_0 - \tau_0$ and target state $v_1 - \widetilde{v}(\cdot, 1)$. The desired control for the nonzero average trajectory is then $G = \widehat{G} + \zeta$. Notably, while this is an approximate controllability result, the average is controlled exactly.
	\end{rmrk}
	
	\begin{rmrk}\label{remark:bop2}
		Let $\varepsilon > 0$ and $B$ a bounded set in $H^{m+1}$ with $v_0, v_1 \in B$. By \Cref{remark:bop} and the constructions in \cref{equation:Gtilde,equation:G}, the controls from \Cref{theorem:locfinthm} and \Cref{remark:tempav} for the modified target condition
		\begin{equation*}
			\|V(\cdot, 1) - v_1\|_m < \varepsilon \|v_0-v_1\|_{m+1}
		\end{equation*}
		can be chosen of the form $G = \mathcal{L}_{\varepsilon}(v_1-v_0)$, where $\mathcal{L}_{\varepsilon}$ is a continuous linear operator $H^{m}\longrightarrow L^2((0,1); \mathscr{F}_{\mathscr{t}})$. 
	\end{rmrk}

	The following auxiliary lemma, which will be used in \Cref{section:conclusion}, emphasizes the $1$-periodicity of homogeneous transportation with stretching effect along the vector field~$\overline{U}$ from \eqref{equation:Ubar}.
	\begin{lmm}\label{lemmea:timeperiodicht}
		For each $V_0 \in H^1$, the solution $V$ to the homogeneous linear convection problem
		\begin{gather*}
			\partial_t V + (\overline{U} \cdot \nabla) V + (\Upsilon(V) \cdot \nabla) \nabla\wedge\overline{U} = 0
		\end{gather*}
		with $V(\cdot, 0) = V_0$ satisfies $V(\cdot, 1) = V_0$.
	\end{lmm}
	\begin{proof}
		Due to the definition of $\overline{U}$ in \eqref{equation:Ubar}, the problem reduces to showing that the solution $V^l$ on $\mathbb{T}^2\times[t^l_c,t^{l+1}_c]$ to 
		\begin{gather*}
			\partial_t V^l + (\overline{U} \cdot \nabla) V^l + (\Upsilon(V^l) \cdot \nabla) \nabla\wedge\overline{U} = 0
		\end{gather*}
		with initial state $V^l(\cdot, t^l_c) = V_0$ satisfies $V(t^{l+1}_c) = V_0$ for each $l \in \{0,\dots,M-1\}$ and $V_0 \in H^1$. This follows from a time reversibility argument and the properties \cref{equation:propybar2,equation:defustar,equation:rpus,equation:rpus3,equation:clpropU}. Indeed, 
		\[
		\overline{U}(\cdot, t^{l}_c + 3T^{\star}/2 +t) = -\overline{U}(\cdot, t^{l}_c + 3T^{\star}/2 - t)
		\]
		holds for all $t \in [0, 3T^{\star}/2]$. Therefore, each $\widehat{V}^l(\cdot, t) \coloneq V^l(\cdot, t^{l}_c + 3T^{\star}/2 - t)$ satisfies in $\mathbb{T}^2 \times [0, 3T^{\star}/2]$ the equation
		\begin{gather*}
			\partial_t \widehat{V}^l + (\overline{U}(\cdot, t^{l}_c + 3T^{\star}/2 + \cdot) \cdot \nabla) \widehat{V}^l + (\Upsilon(\widehat{V}^l) \cdot \nabla) \nabla\wedge\overline{U}(\cdot, t^{l}_c + 3T^{\star}/2 + \cdot) = 0
		\end{gather*}
		with initial condition $\widehat{V}^l(\cdot, 0) = V^l(\cdot, t^{l}_c + 3T^{\star}/2)$.
		Due to $\widehat{V}^l(\cdot, 3T^{\star}/2) = V^l(\cdot, t^l_c) = V_0$, noting that $\widetilde{V}^l(\cdot, t) \coloneq V^l(\cdot, t^l_c+3T^{\star}/2+t)$ and $\widehat{V}^l$ both solve on $[0, 3T^{\star}/2]$ the same well-posed problem, the claim follows.
	\end{proof}

	\section{The nonlinear problem and \texorpdfstring{conclusion of \Cref{theorem:main}}{conclusion of the main result}}\label{section:conclusion}
	Like the $2$D incompressible Navier--Stokes equations (see \cite{FoiasManleyTemam1987,Temam1997, Temam2001}), the planar Boussinesq system \eqref{equation:Boussinesq} with nonzero viscosity and thermal diffusivity is well-posed in common Sobolev space settings. We will work with the vorticity-temperature formulation obtained from \eqref{equation:Boussinesq} by formally applying the operator $\nabla \wedge$ in the velocity equation. 
	
	Given $T > 0$ and $m \in \mathbb{N}$, we define the spaces $\mathcal{X}_{T}^m \coloneq \mathcal{A}_{T}^{m-1} \times \mathcal{A}_{T}^m$ with the norm $\|(w,\theta)\|_{\mathcal{X}_{T}^m} \coloneq \|w\|_{\mathcal{A}_{T}^{m-1}}  + \|\theta\|_{\mathcal{A}_{T}^m}$, where
	\[
	\mathcal{A}_{T}^k \coloneq C^0([0,T];H^k(\mathbb{T}^2; \mathbb{R})) \cap L^2((0,T);H^{k+1}(\mathbb{T}^2; \mathbb{R}))
	\]
	is for $k \in \mathbb{N}_0$ equipped with
	\[
	\|\cdot\|_{\mathcal{A}_{T}^k} \coloneq \|\cdot\|_{C^0([0,T];H^{k}(\mathbb{T}^2; \mathbb{R}))} + \|\cdot\|_{L^2((0,T);H^{k+1}(\mathbb{T}^2; \mathbb{R}))}.
	\]
	Then, for any $(w_0, \theta_0) \in H^{m-1}\times H^{m}$,~$(h_1, h_2) \in L^2((0,T);H^{m-2}\times H^{m-1})$ and $A \in W^{1,2}((0,T); \mathbb{R}^2)$, there exists a unique solution~$(w, \theta) \in \mathcal{X}_{T}^m$
	to the Boussinesq system in vorticity-temperature form 
	\begin{equation}\label{equation:vfb}
		\begin{gathered}
			\partial_t w - \nu \Delta w + \left(u \cdot \nabla \right) w = \partial_1 \theta + h_1, \\
			u(\cdot, t) = \Upsilon(w, A), \\
			\partial_t \theta - \tau\Delta \theta + \left(u \cdot \nabla \right) \theta = h_2, \\
			w(\cdot, 0) = w_0, \quad \theta(\cdot, 0) = \theta_0.
		\end{gathered}
	\end{equation}
	Moreover, the resolving operator for \eqref{equation:vfb}, associating with $(w_0,\theta_0, h_1, h_2, A)$ the solution $(w, \theta)$ to \eqref{equation:vfb}, is continuous and given by
	\begin{gather*}
		S_{T} \colon H^{m-1}\times H^{m} \times L^2((0,T);H^{m-2}\times H^{m-1}) \times W^{1,2}((0,T); \mathbb{R}^2) \longrightarrow \mathcal{X}_{T}^m, \\
		(w_0,\theta_0, h_1, h_2, A) \longmapsto {S}_{T}(w_0,\theta_0, h_1, h_2, A) \coloneq (w, \theta).
	\end{gather*}

	To begin with the proof of \Cref{theorem:main}, let us emphasize that the profile $\overline{U}$ from \eqref{equation:Ubar} satisfies due to \Cref{lemma:retf} the following controllability problem for the incompressible Euler system in $\mathbb{T}^2 \times (0,1)$:
	\begin{equation}\label{equation:ovp}
		\begin{gathered}
			\partial_t \overline{U} + (\overline{U} \cdot \nabla)\overline{U} + \nabla \overline{P} = \mathbb{I}_{\omegaup}\overline{H}, \\ \operatorname{div}(\overline{U}) = 0,\\
			\overline{U}(\cdot, 0) = \overline{U}(\cdot, 1) = 0,
		\end{gathered}
	\end{equation}
	where
	\begin{gather*}
		\overline{P} \in L^2((0,1); C^{\infty}(\mathbb{T}^2; \mathbb{R}^2)), \quad
		\overline{H} \in L^2((0,1); C^{\infty}(\mathbb{T}^2; \mathbb{R}^2)),\\
		\operatorname{supp}(\overline{H}) \subset \omegaup \times (0,1).
	\end{gather*}
	Moreover, since $\overline{U}(\cdot, t)$ is a gradient in a neighborhood of $\mathbb{T}^2\setminus \omegaup$ at each time $t \in [0,1]$, the functions
	\[
		H^{1,1} \coloneq \nabla \wedge \overline{H}, \quad H^{1,2} \coloneq -\Delta (\nabla \wedge \overline{U})
	\]
	satisfy
	\[
		\operatorname{supp}(H^{1,1}) \cup \operatorname{supp}(H^{1,2}) \subset \omegaup\times[0, 1].
	\]
	Further, we define
	\begin{equation}\label{equation:defH12}
		\begin{gathered}
			\overline{U}_{\delta}(\cdot, t) \coloneq \delta^{-1} \overline{U}(\cdot, \delta^{-1}t), \\
			A_{\delta}(t) \coloneq \delta^{-1} \int_0^{\delta^{-1}t}\int_{\mathbb{T}^2} \overline{H}(x, s) \, dx ds,\\
			H_{1, \delta}(\cdot, t) \coloneq \delta^{-2}H^{1,1}(\cdot, \delta^{-1}t) + \delta^{-1}\nu H^{1,2}(\cdot, \delta^{-1}t)
		\end{gathered}
	\end{equation}
	for $\delta >0$ and all $t \in [0,\delta]$. As $\overline{U}$ is by construction supported in~$(0,1)$ with respect to the time variable, it follows from \eqref{equation:ovp} that
	\[
		A_{\delta}(\delta) = \delta^{-1} \int_0^{1}\int_{\mathbb{T}^2} \overline{H}(x, s) \, dx ds = 0.
	\]
	
	The next theorem relates the temperature in \eqref{equation:vfb} at a small time $t = \delta$ with the solution to \eqref{equation:thmp1} evaluated at~$t=1$.  The proof is similar to arguments given in \cite{NersesyanRissel2024,NersesyanRissel2024a}; the general approach is due to \cite{Coron96}, and we refer also to \cite{FursikovImanuvilov1999,Nersesyan2021}.  As the present situation still differs from that in the aforementioned references, at least with respect to the choice of the return method trajectory and due to the average correction terms allowed here, we recall and adapt the argument with details. 
	\begin{thrm}\label{theorem:ctl}
		Assume that $m \geq 2$, $(w_0, \theta_0) \in H^{m} \times H^{m+1}(\mathbb{T}^2;\mathbb{R})$, $A \in \mathbb{R}^2$, $b \in C^{\infty}(\mathbb{T}^2\times[0,1];\mathbb{R}^2)$, and $(h_1, h_2) \in L^2((0,T); H^{m-2} \times H^{m-1})$ for some $T > 0$. Moreover, denote by $(v_{\delta}, \vartheta_{\delta})_{\delta \in (0,\min\{1,T\})}$ the solutions to the linear problems with parameter $\delta \in (0,\min\{1,T\})$:
		\begin{equation}\label{equation:isll}
			\begin{gathered}
				\partial_t v_{\delta} + (\overline{U} \cdot \nabla) v_{\delta} + (\Upsilon(v_{\delta}, B_{\delta}) \cdot \nabla) \nabla \wedge \overline{U} = \partial_1 \vartheta_{\delta} + \nabla \wedge b, \\
				\partial_t \vartheta_{\delta} + (\overline{U} \cdot \nabla) \vartheta_{\delta} = \zeta_{\delta}, \\
				\quad v_{\delta}(\cdot, 0) = w_0, \quad \vartheta_{\delta}(\cdot, 0) = \delta\theta_0,
			\end{gathered}
		\end{equation}
		where 
		\[
		B_{\delta}(t) \coloneq A +  \int_0^{t} \int_{\mathbb{T}^2} b(x,s) \, dx ds + e_2 \int_0^{t} \int_{\mathbb{T}^2} \delta^{-1} \vartheta_{\delta}(x,s) \, dx ds
		\]
		for $t \in [0,1]$ and the family
		$(\zeta_{\delta})_{\delta \in (0,\min\{1,T\})} \subset L^2((0,1); C^{\infty}(\mathbb{T}^2; \mathbb{R}))$ of forces is chosen such that
		\begin{equation}\label{equation:od}
			\sup_{t \in [0, 1]} \|\vartheta_{\delta}(\cdot, t)\|_{m+1} = \mathscr{O}(\delta) \mbox{ as } \delta \longrightarrow 0.
		\end{equation}
		Further, for $t \in [0, \delta]$, denote
		\begin{gather*}
			H_{2,\delta}(\cdot, t) \coloneq \delta^{-2}\zeta_{\delta}(\cdot, \delta^{-1}t), \quad
			b_{\delta}(\cdot, t) \coloneq \delta^{-1} b(\cdot, \delta^{-1}t),\\
			\aleph_{\delta}(t) \coloneq A_{\delta}(t) + B_{\delta}(\delta^{-1}t).
		\end{gather*}
		Then, as $\delta \longrightarrow 0$, one has in $H^{m-1}\times H^m(\mathbb{T}^2;\mathbb{R})$ the convergence
		\begin{equation*}\label{equation:limitl}
			S_{\delta}(w_0, \theta_0, h_1 + H_{1, \delta} + \nabla \wedge b_{\delta},  h_2 + H_{2, \delta}, \aleph_{\delta})|_{t=\delta} - (v_{\delta}, \delta^{-1}\vartheta_{\delta})(\cdot, 1) \longrightarrow 0,
		\end{equation*}
		uniformly with respect to $(h_1, h_2)$ from bounded subsets of~$L^2((0,T); H^{m-2} \times H^{m-1})$ and $(w_0, \theta_0)$ from bounded subsets of $\in H^{m} \times H^{m+1}(\mathbb{T}^2;\mathbb{R})$.
	\end{thrm}
	\begin{proof}
		Given any fixed $\delta \in (0,\min\{1,T\})$, we denote by $(w_{\delta}, U_{\delta}, \Theta_{\delta})$ the solution to \eqref{equation:vfb} on $[0,\delta]$ issued from the initial state $(w_0, \theta_0)$, driven by $(h_1 + H_{1,\delta} + \nabla \wedge b_{\delta}, h_2 + H_{2,\delta})$, and having the velocity average~$\aleph_{\delta}$. That is,
		\begin{gather*}
			(W_{\delta}, \Theta_{\delta}) = S_{\delta}(w_0, \theta_0, h_1 + H_{1,\delta} + \nabla \wedge b_{\delta},  h_2 + H_{2,\delta}, \aleph_{\delta})
		\end{gather*}
		and $U_{\delta}(x,t) = \Upsilon(W_{\delta}, \aleph_{\delta})$. Then, in the limit $\delta \longrightarrow 0$, we study for the functions~$W_{\delta}$,~$U_{\delta}$, and $\Theta_{\delta}$ an ansatz of the form
		\begin{equation*}\label{equation:ansatz}
			\begin{gathered}
				W_{\delta} = \overline{w}_{\delta} + z_{\delta} + r_{\delta}, \quad U_{\delta} = \overline{U}_{\delta} + Z_{\delta} + \Upsilon(r_{\delta}), \quad \Theta_{\delta} = \theta_{\delta} + s_{\delta},
			\end{gathered}
		\end{equation*}
		where $r_{\delta}$ and $s_{\delta}$ denote remainders and
		\begin{gather*}
			z_{\delta} \coloneq v_{\delta}(\cdot, \delta^{-1}\cdot), \quad \overline{w}_{\delta} \coloneq \nabla \wedge \overline{U}_{\delta}, \quad
			\theta_{\delta} \coloneq \delta^{-1} \vartheta_{\delta}(\cdot, \delta^{-1} \cdot),\\
			Z_{\delta} \coloneq \Upsilon(z_{\delta}, B_{\delta}(\delta^{-1}\cdot)).
		\end{gather*}
		Since $\overline{w}_{\delta}(\cdot, \delta) = 0$ in $\mathbb{T}^2$, the proof will be complete after the following convergence is obtained:
		\begin{equation}\label{equation:gts}
			\|r_{\delta}(\cdot, \delta)\|_{m-1} + \|s_{\delta}(\cdot, \delta)\|_{m} \longrightarrow 0 \, \mbox{ as } \, \delta \longrightarrow 0.
		\end{equation}
		
		Plugging this ansatz to the equations in \eqref{equation:vfb}, while recalling \cref{equation:ovp,equation:isll,equation:defH12}, the remainder profiles are seen to satisfy in $\mathbb{T}^2\times[0,\delta]$ the problem
		\begin{equation}\label{equation:remainders}
			\begin{gathered}
				\partial_t r_{\delta} - \nu \Delta r_{\delta} + (\overline{U}_{\delta} + Z_{\delta} + \Upsilon(r_{\delta})) \cdot \nabla r_{\delta} + \Upsilon(r_{\delta})\cdot \nabla (\overline{w}_{\delta} + z_{\delta}) = \partial_1 s_{\delta} + \Xi_{\delta},\\
				\partial_t s_{\delta} - \tau \Delta s_{\delta} + (\overline{U}_{\delta} + Z_{\delta} + \Upsilon(r_{\delta})) \cdot \nabla s_{\delta} + \Upsilon(r_{\delta})\cdot \nabla \theta_{\delta}= \Lambda_{\delta},
			\end{gathered}
		\end{equation}
		with zero initial states $r_{\delta}(\cdot, 0) = s_{\delta}(\cdot, 0)$, and where
		\begin{gather*}
			\Xi_{\delta} \coloneq \nu \Delta z_{\delta} - (Z_{\delta} \cdot \nabla) z_{\delta} + h_1, \quad
			\Lambda_{\delta} \coloneq \tau \Delta \theta_{\delta} - (Z_{\delta}\cdot \nabla) \theta_{\delta} + h_2.
		\end{gather*}
		
		To underscore that all estimates are uniform with respect to initial states and prescribed forces from bounded sets, we note that $(v_{\delta}, \delta^{-1}\vartheta_{\delta})$ remains for each $\delta \in (0, \min\{1,T\})$ in a fixed bounded subset of $C^0([0,1];H^m\times H^{m+1})$ when $(w_0, \theta_0)$ vary in a fixed bounded subset of $H^{m} \times H^{m+1}$.
		Now, we multiply in \eqref{equation:remainders} with~$(-\Delta)^{m-1} r_{\delta}$ and $(-\Delta)^{m} s_{\delta}$, respectively, followed by applications of Poincar\'e's inequality and the elliptic estimate $\|\Upsilon(\phi)\|_{k} \lesssim \|\phi\|_{k-1}$ for $k \in \mathbb{N}$, where \enquote{$\lesssim$} means \enquote{$\leq C$} for a generic constant $C > 0$ independent of~$\delta$. For example, due to~\eqref{equation:od} and because $A$ and $b$ are fixed, we have
		\begin{align*}
			\operatorname{sup}\limits_{t \in [0,\delta]}\|Z_{\delta}(\cdot, t)\|_{m+1} & \lesssim \operatorname{sup}_{t \in [0,\delta]}\|z_{\delta}(\cdot, t)\|_m + |A| + \|b\|_{L^1(0,\delta); L^1(\mathbb{T}^2; \mathbb{R}^2)} + 1\\
			& \lesssim \operatorname{sup}_{t \in [0,\delta]}\|z_{\delta}(\cdot, t)\|_m + 1.
		\end{align*}
		As a result of these steps, and by temporarily using abbreviations of the type $f = f(\cdot,t)$, we obtain
		\begin{multline}\label{equation:ee}
			\frac{1}{2}\frac{d}{dt}\big(\|r_{\delta}\|_{m-1}^2 + \|s_{\delta}\|_{m}^2\big) + \nu \| r_{\delta}\|_{m}^2 + \tau \| s_{\delta} \|_{m+1}^2 \\
			\begin{aligned}
				& \lesssim \|r_{\delta}\|_{m-1}^2\big( \|z_{\delta}\|_{m}+ \|r_{\delta}\|_{m}\big) + \|\overline{U}_{\delta}\|_{m+1} \big(\|r_{\delta}\|_{m-1}^2+\|s_{\delta}\|_{m}^2\big)\\
				& \quad \, + \|r_{\delta}\|_{m-1}\big(\|\theta_{\delta}\|_{m+1}\|s_{\delta}\|_{m} +  \|s_{\delta}\|_{m}\|s_{\delta}\|_{m+1}\big) \\
				& \quad \, + (\|z_{\delta}\|_{m}+1) \big(\|r_{\delta}\|_{m-1}^2 + \|s_{\delta}\|_{m}^2\big) + \|r_{\delta}\|_{m-1}\|s_{\delta}\|_{m}\\
				& \quad \, + \|\Xi_{\delta}\|_{m-2} \|r_{\delta}\|_{m} + \|\Lambda_{\delta}\|_{m-1} \|s_{\delta}\|_{m+1}
			\end{aligned}
		\end{multline}
		for $t \in [0, \delta]$. Thanks to the continuous Sobolev embeddings $H^2(\mathbb{T}^2; \mathbb{R}^N) \hookrightarrow L^{\infty}(\mathbb{T}^2; \mathbb{R})$ and $H^1(\mathbb{T}^2; \mathbb{R}^N) \hookrightarrow L^{4}(\mathbb{T}^2; \mathbb{R}^N)$, $N \in \{1,2\}$, it follows that
		\begin{gather*}
			\|r_{\delta}\|_{m-1}^2\|r_{\delta}\|_{m}\leq \alpha\|r_{\delta}\|_{m}^2 + \alpha^{-1}C\|r_{\delta}\|_{m-1}^4,\\
			\|r_{\delta}\|_{m-1}\|\theta_{\delta}\|_{m+1}\|s_{\delta}\|_{m}
			\leq \|\theta_{\delta}\|_{m+1} \big(\|r_{\delta}\|_{m-1}^2 + \|s_{\delta}\|_{m}^2\big),\\
			\|r_{\delta}\|_{m-1} \|s_{\delta}\|_{m}\|s_{\delta}\|_{m+1} \leq \alpha \|s_{\delta}\|_{m+1}^2 + \alpha^{-1} C \big(\|r_{\delta}\|_{m-1}^4 + \|s_{\delta}\|_{m}^4\big)\\
			\|r_{\delta}\|_{m-1}\|s_{\delta}\|_{m} \leq 2^{-1}\|r_{\delta}\|_{m-1}^2 + 2^{-1}\|s_{\delta}\|_{m}^2,\\
			\|\Xi_{\delta}\|_{m-2} \leq \|z_{\delta}\|_m + \|z_{\delta}\|_m^2 + \|h_1\|_{m-2} + 1,\\
			\|\Lambda_{\delta}\|_{m-1} \leq \|\theta_{\delta}\|_{m+1} + \|z_{\delta}\|_{m}^2 + \|\theta_{\delta}\|_{m+1}^2 + \|h_2\|_{m-1},
		\end{gather*}
		where $\alpha > 0$ is small.
		The idea is then to observe that integrals involving several of the terms in \eqref{equation:ee} vanish when taking the limit $\delta\longrightarrow 0$. Indeed,
		\begin{equation*}
			\begin{gathered}
				\int_0^t \|f(\cdot, \sigma)\|_{l} \, d\sigma \leq \min \left\{ \delta \int_0^1 \|f(\cdot, \delta \sigma)\|_{l} \, d\sigma, \,  \int_0^\delta \|f(\cdot, \sigma)\|_l \, d\sigma \right\}
			\end{gathered}
		\end{equation*}
		for~any $f \in L^1((0, \delta); H^{l}(\mathbb{T}^2; \mathbb{R}))$ with~$l \geq 0$ and $t \in [0, \delta]$. Thus, recalling that $\vartheta_{\delta}(\cdot, 0) = \delta\theta_0$ together with the choice of $(\zeta_{\delta})_{\delta \in (0,\min\{1,T\})}$ in \eqref{equation:isll} yield $\sup_{t\in[0,\delta]}\|\theta_{\delta}(\cdot, t)\|_{m+1} = \mathscr{O}(1) \mbox{ as } \delta \longrightarrow 0$,
		we have the limits
		\begin{gather*}
			\lim\limits_{\delta \to 0} \int^{\delta} \left(\|\Xi_{\delta}(\cdot, \sigma)\|_{m-2} + \|\Lambda_{\delta}(\cdot, \sigma)\|_{m-1}\right) \, d\sigma = 0, \\
			\lim\limits_{\delta \to 0} \int_0^{\delta} \|\overline{U}_{\delta}(\cdot, \sigma) + Z_{\delta}(\cdot, \sigma)\|_{m+1} \, d\sigma  \leq \sup_{s \in [0, 1]} \|\overline{U}(\cdot, s)\|_{m+1},\\
			\int_0^{\delta} \|\theta_{\delta}(\cdot, \sigma)\|_{m+1}^2 \, d\sigma \leq \delta^{-1}\sup_{s \in [0,1]} \|\vartheta_{\delta}(\cdot, s)\|_{m+1}^2 = \mathscr{O}(\delta) \mbox{ as } \delta \longrightarrow 0.
		\end{gather*}
		Therefore, an application of Gr\"onwall's lemma implies the existence of a constants $(C_{\delta})_{\delta\in(0,\min\{1,T\})}$ satisfying $\lim_{\delta \to 0} C_{\delta} = 0$ and
		\begin{equation*}
			\|r_{\delta}(\cdot, t)\|_{m-1}^2 + \|s_{\delta}(\cdot, t)\|_{m}^2 \leq C_{\delta} + C\int_0^t \left(\|r_{\delta}(\cdot, \sigma)\|_{m-1}^4 + \|s_{\delta}(\cdot, \sigma)\|_{m}^4\right) \, d\sigma.
		\end{equation*}
		Thus, the function
		\[
		\Psi_{\delta}(t) \coloneq C_{\delta} + C\int_0^t \left(\|r_{\delta}(\cdot, \sigma)\|_{m-1}^4 + \|s_{\delta}(\cdot, \sigma)\|_{m}^4\right) \, d\sigma
		\]
		satisfies $d\Psi_{\delta}/dt \leq C \Psi_{\delta}^2$ and \eqref{equation:gts} follows by integrating this inequality. 
	\end{proof}
	\begin{rmrk}
		To conclude \Cref{theorem:main} and \Cref{theorem:main2}, we will apply \Cref{theorem:ctl}, and \Cref{theorem:msltc} below, only with $b = 0$, $A = 0$, $\tau = 0$, and $\smallint_{\mathbb{T}^2} \theta_0(x) \, dx = \smallint_{\mathbb{T}^2} \theta_1(x) \, dx = 0$. Regarding the relevance of other choices for these parameters, see \Cref{remark:nza2}.
	\end{rmrk}
	
	Let $\widetilde{\mathscr{F}}_{\mathscr{t}}$ be obtained by adding to $\mathscr{F}_{\mathscr{t}}$ the space spanned by the smooth cutoff~$\mu$ from \Cref{subsection:convection}. 
	
	\begin{crllr}\label{theorem:msltc}
		Let $T> 0$, $m \geq 2$, $(w_0, \theta_0, \theta_1) \in H^{m} \times H^{m+1}(\mathbb{T}^2;\mathbb{R}) \times H^{m+1}(\mathbb{T}^2;\mathbb{R})$, $b \in C^{\infty}(\mathbb{T}^2\times[0,1];\mathbb{R}^2)$, $A \in \mathbb{R}^2$, $(h_1, h_2) \in L^2((0, T); H^{m-2} \times H^{m-1})$, and $\tau \in C^{\infty}_0((0,1);\mathbb{R})$ with 
		\[
		\int_0^1\tau(s) \,ds = \int_{\mathbb{T}^2} (\theta_1-\theta_0)(x) \, dx.
		\]
		There exist $(\zeta_{\delta})_{\delta \in (0,\min\{1,T\})} \subset L^2((0, 1); \widetilde{\mathscr{F}}_{\mathscr{t}})$, with $\int_{\mathbb{T}^2} \zeta_{\delta}(x,t) \, dx = \delta\tau(t)$
		for $\delta \in (0,\min\{1,T\})$ and almost all $t \in [0,1]$, so that in~$H^{m-1} \times H^m(\mathbb{T}^2;\mathbb{R}^2)$ one has the convergence
		\begin{multline}\label{equation:limit}
			\lim\limits_{\delta \to 0} S_{\delta}\Big(w_0, \theta_0, h_1 + H_{1,\delta} + \nabla \wedge b_{\delta}, \\ h_2 + \delta^{-2}\zeta_{\delta}(\cdot, \delta^{-1}\cdot), \aleph_{\delta}\Big)\!|_{t=\delta} = (v^b(\cdot, 1), \theta_1),
		\end{multline}
		where
		\begin{gather*}
			b_{\delta}(\cdot, t) = \delta^{-1} b(\cdot, \delta^{-1}t), \quad
			\aleph_{\delta}(t) \coloneq A_{\delta}(t) + \widetilde{B}(\delta^{-1}t)
		\end{gather*}
		for $t \in [0, \delta]$ and
		\[
		\widetilde{B}(t) \coloneq A + \int_0^{t} \int_{\mathbb{T}^2} b(x,s) \, dx ds + 
		e_2 \left(\int_0^{t}\int_0^s \tau(r) \, dr ds + t \int_{\mathbb{T}^2} \theta_0(x) \, dx \right)
		\]
		for $t \in [0, 1]$, while the function $v^b$ solves
		\begin{equation}\label{equation:vb}
			\begin{gathered}
				\partial_t v^b + (\overline{U} \cdot \nabla) v^b + (\Upsilon(v^b,\widetilde{B}) \cdot \nabla) \nabla \wedge \overline{U} = \nabla \wedge b, \\
				v^{b}(\cdot, 0) = w_0.
			\end{gathered}
		\end{equation}
		The limit in \eqref{equation:limit} is uniform with respect to $(h_1, h_2)$ from bounded subsets of $L^2((0,{T}); H^{m-2} \times H^{m-1})$ and $(w_0, \theta_0)$ from bounded subsets of $\in H^{m} \times H^{m+1}(\mathbb{T}^2;\mathbb{R})$. Furthermore, if $(\theta_0, \theta_1) \in H^{m+1} \times H^{m+1}$ and $\tau = 0$, then one can use controls $(\zeta_{\delta})_{\delta \in (0,\min\{1,T\})} \subset L^2((0, 1); \mathscr{F}_{\mathscr{t}})$.
	\end{crllr}
	\begin{proof}
		Let $\varepsilon > 0$. If $\tau = 0$, we apply \Cref{theorem:locfinthm} for each $\rho \in (0,\min\{1,T\})$ with initial state $\rho\theta_0$ and target state $\rho\theta_1$. In view of \Cref{remark:bop2}, this provides controls $(\zeta_{\rho})_{\rho \in (0,\min\{1,T\})} \in L^2((0,1); \mathscr{F}_{\mathscr{t}})$ such that the solution $\vartheta_{\rho}$ to the transport equation
		\begin{gather*}
			\partial_t \vartheta_{\rho} + (\overline{U} \cdot \nabla) \vartheta_{\rho} = \zeta_{\rho}
		\end{gather*}
		with initial condition $\vartheta_{\rho}(\cdot, 0) = \rho \theta_0$ satisfies
		\begin{equation}\label{equation:acetd}
			\|\vartheta_{\rho}(\cdot, 1) - \rho \theta_1\|_m < \varepsilon \rho.
		\end{equation}
		If some values of $\tau$ are nonzero, we use the version of \Cref{theorem:locfinthm} described in \Cref{remark:tempav}, which leads instead to controls $(\zeta_{\rho})_{\rho \in (0,\min\{1,T\})} \in L^2((0,1); \widetilde{\mathscr{F}}_{\mathscr{t}})$.
		According to \Cref{remark:bop2}, the family $(\zeta_{\rho})_{\rho\in(0,\min\{1,T\})}$ can be fixed such that
		\[
		\|\zeta_{\rho}\|_{L^2((0,1); H^m)} = \mathscr{O}(\rho) \mbox{ as } \rho \longrightarrow 0.
		\]
		Now, we define for $\rho \in (0,\min\{1,T\})$ the function $v_{\rho}$ as the solution to
		\begin{gather*}
			\partial_t v_{\rho} + \overline{U} \cdot \nabla v_{\rho} + \left(\Upsilon(v_{\rho}, \widetilde{B})\cdot \nabla \right) (\nabla \wedge \overline{U}) = \partial_1 \vartheta_{\rho} + \nabla \wedge b
		\end{gather*}
		with initial condition $v_{\rho}(\cdot, 0) = w_0$. Basic estimates and \Cref{remark:bop2} imply
		\[
		\sup\limits_{t \in [0,1]}\|\vartheta_{\rho}(\cdot, t)\|_{m+1} + \|v_{\rho}(\cdot, 1) - v^b(\cdot, 1)\|_{m} = \mathscr{O}(\rho) \mbox{ as } \rho \longrightarrow 0.
		\]
		Thus, by combining \Cref{theorem:ctl} and \eqref{equation:acetd}, there exists $\delta_0 = \delta_0(\varepsilon) > 0$ such that one has for all $\delta \in (0, \delta_0)$ the estimate
		\[
		\|(w_{\delta},\theta_{\delta})(\cdot,\delta) - (v^b(\cdot, 1), \theta_1)\|_{H^{m-1}\times H^m(\mathbb{T}^2;\mathbb{R}^2)} < \varepsilon,
		\]
		where
		\[
		(w_{\delta},\theta_{\delta}) \coloneq S_{\delta}\Big(w_0, \theta_0, h_1 + H_{1,\delta} + \nabla \wedge b_{\delta}, h_2 + \delta^{-2}\zeta_{\delta}(\cdot, \delta^{-1}\cdot), \aleph_{\delta}\Big).
		\]
	\end{proof}

	\Cref{theorem:main} follows now from the choice of controls $(\zeta_{\delta})_{\delta \in (0,\min\{1,T\})} \subset L^2((0, 1); \mathscr{F}_{\mathscr{t}})$ made via \Cref{theorem:msltc} for $w_0 = \nabla \wedge u_0$, $h_1 = \nabla \wedge f$, $h_2 = g$, $b=0$, $A = 0$, $\tau = 0$, and $\smallint_{\mathbb{T}^2}\theta_0(x)\,dx = \smallint_{\mathbb{T}^2}\theta_1(x)\,dx = 0$. Indeed, \Cref{lemmea:timeperiodicht} implies in this case that $v^b(\cdot, 1) = w_0$ for the solution $v^b$ to \eqref{equation:vb}. Moreover,
	the solution to~\eqref{equation:Boussinesq} driven by the controls
	\begin{equation}\label{equation:ctrlfrm}
		\begin{gathered}
			\xi(\cdot, t) = \delta^{-2} \overline{H}(\cdot, \delta^{-1}t) - \delta^{-1}\Delta \overline{U}(\cdot, \delta^{-1}t), \\
			\eta(\cdot, t) = \delta^{-2} \zeta_{\delta}(\cdot, \delta^{-1}t)
		\end{gathered}
	\end{equation}
	is for $t \in [0,\delta]$ given by $(u_{\delta}, \theta_{\delta})$, where $u_{\delta} = \Upsilon(w_{\delta}, A_{\delta})$ is the velocity associated with $(w_{\delta}, \theta_{\delta}) = S_{\delta}(w_0, \theta_0, H_{1,\delta}, h_2 + \delta^{-2}\zeta_{\delta}(\cdot, \delta^{-1}\cdot), A_{\delta})$.
	Since $A_{\delta}(\delta) = 0$, the approximate controllability of $u_{\delta}$ follows from \Cref{theorem:msltc} by using the elliptic estimate
	\[
	\|u_{\delta}(\cdot, \delta) - u_0\|_m \lesssim \|w_{\delta}(\cdot, \delta) - w_0\|_{m-1}.
	\]
	
	In view of \eqref{equation:ctrlfrm}, we can now determine possible choices for $\mathscr{F}_{\mathscr{t}}$  and $\mathscr{F}_{\mathscr{v}}$.  The space $\mathscr{F}_{\mathscr{t}}$ in~\Cref{theorem:main} can be taken as the $(D_{\omegaup} + 8)$-dimensional one from \Cref{theorem:locfinthm}, recalling that $D_{\omegaup}$ is fixed in terms of the universal choice of~$\overline{y}$ via \Cref{lemma:retf}. 
	Regarding the space $\mathscr{F}_{\mathscr{v}}$, from \Cref{lemma:retf} and the representation $\overline{U}(x,t) = \sum_{i=1}^{D_{\omegaup}+4} u_i(t) U_i(x)$ in \eqref{equation:repUbarfd}, we know that $t \mapsto \overline{U}(\cdot, t)$ and $t\mapsto\Delta\overline{U}(\cdot, t)$ can be chosen as curves in possibly different but at most $(D_{\omegaup}+4)$-dimensional spaces. Moreover, for $q$ satisfying $\partial_t \overline{U} + (\overline{U}\cdot\nabla)\overline{U} = \nabla q$ in a neighborhood of $\mathbb{T}^2\setminus \omegaup$, we can take $\overline{P} = -(1-\chi) q$, with $\chi$ from \Cref{subsection:convection}, and fix $\overline{H} = \partial_t \overline{U} + (\overline{U}\cdot\nabla)\overline{U} + \nabla \overline{P}$. Therefore, the dimension of~$\mathscr{F}_{\mathscr{v}}$ can be kept below or equal to
	\[
		D_{\mathscr{v}} \coloneq 12 + 3D_{\omegaup} + \frac{(D_{\omegaup} + 4)^2-4-D_{\omegaup}}{2},
	\]
	where the contribution $4 + D_{\omegaup} + 2^{-1}((D_{\omegaup} + 4)^2-4-D_{\omegaup})$	to the above sum arises from grouping the terms arising when expanding the expression
	\[
		\left(\sum_{i=1}^{D_{\omegaup} + 4} u_i(t) U_i(x)\cdot\nabla\right)\sum_{i=1}^{D_{\omegaup} + 4} u_i(t) U_i(x)
	\]
	with respect to the common factors $u_i(t)u_i(t)$ and $u_j(t) u_l(t) = u_l(t)u_j(t)$ with $i,j,l \in \{1, \dots, D_{\omegaup} + 4\}$ and $l\neq j$. The pressure $\overline{P}$ in the above choice of $\overline{H}$ does not affect the dimension of $\mathscr{F}_{\mathscr{v}}$, as one may factor common control coefficients. Thus, by plugging \eqref{equation:repUbarfd} into \eqref{equation:ovp}, one can represent~$\overline{H}$ as
	\begin{equation*}\label{equation:etarep}
		\begin{gathered}
			\overline{H}(x,t) + \Delta \overline{U}(x,t) = \sum_{i=1}^{D_{\mathscr{v}}} w_i(t) W_i(x),
		\end{gathered}
	\end{equation*}
	for coefficients $w_1, \dots, w_{D_{\mathscr{v}}} \in L^2((0,1); \mathbb{R})$
	and universal profiles $W_1, \dots, W_{D_{\mathscr{v}}} \in C^{\infty}(\mathbb{T}^2; \mathbb{R}^2)$ supported in $\omegaup$. 
	
	\begin{rmrk}\label{remark:exampley2}
		If $\mathbb{T}^2\setminus \omegaup$ is simply-connected, instead of using the proof of \Cref{lemma:retf} one can define the profile $\overline{y}$ appearing in \eqref{equation:Ubar} as a function only of time ({\it c.f.}~\Cref{remark:exampley}). In that case, $D_{\omegaup} = 2$. As a result, the spaces $\mathscr{F}_{\mathscr{v}}$ and $\mathscr{F}_{\mathscr{t}}$ are independent of the choice of $\omegaup$ within the class of open sets $\omegaup$ for which $\mathbb{T}^2\setminus \omegaup$ is simply-connected; see \Cref{subsection:ce} for an explicit list of functions that span these spaces.
	\end{rmrk}

	\appendix
	
	\section{Proof of \texorpdfstring{\Cref{theorem:main2}}{the second main result}}\label{subsection:prfmain2}
	This appendix recalls a strategy from \cite{NersesyanRissel2024}, rendering \Cref{theorem:main2} as a consequence of \Cref{theorem:main}. Hereto, let $m \in \mathbb{N}$, which corresponds to $k \geq 2$ in \Cref{theorem:main2}. If $k \in \{0,1\}$, the statement follows from a density argument, as one can approximate the target states by smooth functions.
	
	\subsubsection*{Step 1. Stability} Due to the well-posedness of \eqref{equation:vfb} in the considered spaces, there exists a small time $\sigma > 0$ such that for each $\delta \in [0,\sigma]$ and $(a,b) \in H^{m}\times H^{m+1}$ with $\|(a,b)-(w_0, \theta_1)\|_{H^{m}\times H^{m+1}} < \varepsilon/2$,
	one has
	\[
	\|S_{\delta}(a,b,\nabla \wedge f(T-\delta+\cdot), g(T-\delta+\cdot), 0) |_{t = \delta} - (w_0, \theta_1)\|_{H^{m}\times H^{m+1}} < \varepsilon.
	\]
	
	\subsubsection*{Step 2. Regularization}  On a time interval $[0, T-\delta_0]$, where $\delta_0 \in (0,\sigma]$ will be selected below and~$\sigma$ is the number from the previous step, all controls are set to zero. Then, the smoothing effects of the considered viscous and thermally diffusive Boussinesq system imply that the solution $(u, \theta)$ to \eqref{equation:Boussinesq} in the Leray-Hopf class, with initial data $(u_0, \theta_0)$ and forces $(f, g)$, belongs to the space
	\[
	C^0((0,T-\delta_0];H^{m+1}(\mathbb{T}^2;\mathbb{R}^2)\times H^{m+1}) \cap L^2((0, T-\delta_0); H^{m+2}(\mathbb{T}^2;\mathbb{R}^2)\times H^{m+2}).
	\]
	As a result, one has
	\[
	(u,\theta)(\cdot, t) \in V^{m+2}\times H^{m+2}
	\]
	for almost all $t\in(0, T-\delta_0)$, which implies
	\[
	(\widetilde{w}_0, \widetilde{\theta}_0) \coloneq (\nabla\wedge u, \theta)(\cdot, T - \delta_0) \in H^{m+1} \times H^{m+2}
	\]
	for a well-chosen $\delta_0 \in (0,\sigma]$, which is fixed from now on. We refer also to \cite{Chaves-SilvaEtal2023,Temam1997,Temam2001,Temam1982} regarding Leray-Hopf type solutions and smoothing effects, presented for the more challenging case of domains with boundaries.  Moreover, since the goal is to prove approximate controllability, we can without loss of generality assume that the targets $u_1$ and $\theta_1$ are smooth. 
	
	\subsubsection*{Step 3. Control strategy} As the following mechanism is known from \cite{NersesyanRissel2024}, we simplify here the presentation by assuming $f = g = 0$ in \eqref{equation:Boussinesq}. Moreover, for the resolving operator $S_T$ of \eqref{equation:vfb} with time $T > 0$, we abbreviate $S_T(\cdot, \cdot) = S_T(\cdot, \cdot, 0, 0, 0)$. 
	The next theorem states two scaling limits from \cite{NersesyanRissel2024}. Similar results have been obtained in \cite{BoulvardGaoNersesyan2023} for the primitive equations.
	\begin{thrm}{\cite[Theorem 3.4]{NersesyanRissel2024}}\label{theorem:lcst}
		Let~$k \in \mathbb{N}$, $T$, zero average $q \in C^{\infty}(\mathbb{T}^2; \mathbb{R})$, and $(w_0, \theta_0) \in H^{k+1} \times H^{k+2}$. Further, denote by $\Pi_1 (w, \theta) \coloneq w$ the projection to the vorticity component of a solution $(w,\theta)$ to \eqref{equation:vfb}. Then, as $\delta \longrightarrow 0$, the limits
		\begin{gather}\label{equation:lsc1}
			\Pi_1 S_{\delta}(w_0 , \theta_0 - \delta^{-1} q) |_{t=\delta}  \longrightarrow w_0 - \partial_1q,\\
			S_{\delta}(w_0 + \delta^{-1/2}q, \theta_0) |_{t=\delta} - (\delta^{-1/2}q, 0) \longrightarrow (w_0 - (\Upsilon(q)  \cdot  \nabla) q, \theta_0)\label{equation:lsc2}
		\end{gather}
		hold in $H^{k}$ and $H^{k}\times H^{k+1}$, respectively.
	\end{thrm}
	\begin{rmrk}
		Lets us briefly describe the idea employed in \cite{NersesyanRissel2024} for proving \Cref{theorem:lcst}. Regarding \eqref{equation:lsc1}, the following ansatz is made as $\delta\longrightarrow0$:
		\begin{multline*}
			S_{\delta}(w_0, \theta_0 - \delta^{-1}q)(x,t) + (0, \delta^{-1}q(x)) = \\ \begin{bmatrix}
				w_0(x) \\ \theta_0(x) 
			\end{bmatrix} - \begin{bmatrix}
				\delta^{-1} t \partial_1 q(x)\\
				\delta^{-1}t \tau \Delta q(x) + \delta^{-1}t  (Q_{\delta}(x,t) \cdot \nabla) q(x)
			\end{bmatrix} + R_{\delta}(x,t),
		\end{multline*}
		where $R_{\delta}(x,t)$ denotes a supposedly small remainder term and
		\[
		Q_{\delta}(x,t) \coloneq \Upsilon\left(w_0 - \frac{\delta^{-1} t \partial_1 q}{2}\right)(x)
		\]
		for $(x,t) \in \mathbb{T}^2\times[0,\delta]$. The equation satisfied by $R_{\delta}$ is obtained by plugging the above ansatz into the Boussinesq system, and via energy estimates it is seen that $R_{\delta}(\cdot, \delta)$ vanishes in $H^k\times H^{k+1}$ as $\delta \longrightarrow 0$. Concerning \eqref{equation:lsc2}, the ansatz used in \cite{NersesyanRissel2024} is of the form
		\begin{multline*}
			S_{\delta}(w_0 + \delta^{-1/2}q, \theta_0)(x,t) - (\delta^{-1/2}q(x), 0) = \\
			\begin{bmatrix}
				w_0(x) \\ \theta_0(x)
			\end{bmatrix} - \begin{bmatrix}
				\delta^{-1} t \left(\Upsilon(q)(x)\cdot\nabla\right)q(x) - \delta^{-1/2} t \nu \Delta q(x) \\ 0
			\end{bmatrix} + R_{\delta}(x,t)
		\end{multline*}
		for $(x,t) \in \mathbb{T}^2\times[0,\delta]$. Also in this case, the remainder $R_{\delta}(\cdot, \delta)$ is seen to vanish in $H^k\times H^{k+1}$ as $\delta \longrightarrow 0$. Finding the right ansatz is the crucial step; the remainder estimates for both limits are then rather straight forward and similar to those in the proof of \Cref{theorem:ctl}.
	\end{rmrk}
	
	The convergence results of \Cref{theorem:lcst} will be combined with the fact that $\mathscr{E}$ from \eqref{equation:E} contains $\pm\sin(x\cdot n)$ and $\pm\cos(x\cdot n)$ for all $n\in \mathbb{Z}^2\setminus\{0\}$. Let us recall that
	\begin{equation*}
		\begin{gathered}
			\mathscr{E} = \left\{ q_0 + \left(\Upsilon(q_1) \cdot \nabla\right) q_1 + \left(\Upsilon(q_2) \cdot \nabla\right) q_2 \, \, | \, \, q_0,q_1,q_2\in \operatorname{span}_{\mathbb{R}}\mathscr{E}_0 \right\},
		\end{gathered}
	\end{equation*}
	where $\mathscr{E}_0$ collects all functions $\sin(x\cdot n)$ and $\cos(x \cdot n)$ for $n \in \mathbb{N}\times\mathbb{N}_0$. In particular, there exists an integer $N \geq 0$ and
	\[
	q_0, q_1, \dots, q_{2N} \in \operatorname{span}_{\mathbb{R}}\mathscr{E}_0, \quad  Q_0, Q_1, \dots, Q_{2N}\in C^{\infty}(\mathbb{T}^2; \mathbb{R})
	\]
	such that 
	\begin{equation}\label{equation:W1g}
		\|W_1 - w_1 \|_{m} < \frac{\varepsilon}{3},
	\end{equation}
	where
	\[
	W_1 \coloneq \widetilde{w}_0 - q_0 - \sum_{i=1}^{2N}(\Upsilon(q_i) \cdot \nabla) q_i
	\]
	and
	\[
	q_0 = \partial_1 Q_0, \quad q_1 = \partial_1 Q_1, \quad \dots, \quad q_{2N} = \partial_1 Q_{2N}.
	\]
	
	Starting a trajectory at time $T_0 = T - \delta_0$ from the state $(\widetilde{w}_0, \widetilde{\theta}_0)$, the following steps i)-iv) demonstrate how the vorticity in the Boussinesq system \eqref{equation:vfb} can be steered faster than $(T-T_0)/(10N+1)$, and up to any small error~$\overline{\varepsilon} > 0$ with respect to the $H^{m}$-norm, to the value
	\[
	\widetilde{w}_0 - q_0 - (\Upsilon(q_1) \cdot \nabla) q_1.
	\]
	Thanks to \eqref{equation:W1g}, by repeating the below argument $(2N-1)$-times and choosing~$\overline{\varepsilon}$, one can build a piece-wise (in time) defined trajectory so that the associated vorticity reaches $W_1$ in $H^m$ up to any prescribed error $\widetilde{\varepsilon}$. After a final application of \Cref{theorem:msltc} to steer the temperature in $H^{m+1}$ as close to $\theta_1$ as required, the proof is complete.

	i) For any $\varepsilon_1 > 0$, we take $0 < \delta_3 < (T-T_0)/(10N+1)$ so small that \eqref{equation:lsc2} implies
	\begin{equation*}
		\|S_{\delta_3}\left(\widetilde{w}_0 - \delta_3^{-1/2} q_1, \widetilde{\Theta}\right)\!|_{t=\delta_3} - (\widetilde{w}_0-\delta_3^{-1/2}q_1-(\Upsilon(q_1)\cdot\nabla)q_1, \widetilde{\Theta})\|_{H^{m}\times H^{m+1}} < \varepsilon_1,
	\end{equation*}
	where $\widetilde{\Theta}$ denotes a fixed element of $H^{m+1}$ ({\it e.g.}, choose $\widetilde{\Theta} = 0$). 
	
	ii) We apply \eqref{equation:lsc1} and \Cref{theorem:msltc}, the latter with target temperature $\widetilde{\Theta}$, in order to fix a small $0 < \delta_2 < (T-T_0)/(10N+1)$ of the form $\delta_2 = \delta_{2,1} + \delta_{2,2}$ and a control $\zeta^0 \in L^2((0, 1); \mathscr{F}_{\mathscr{t}})$ such that
	\begin{multline*}
		\|S_{\delta_3}\left(\widetilde{S}_{\delta_2}\left(\widetilde{w}_0, \widetilde{\theta}_0 - \delta_{2,1}^{-1}\delta_3^{-1/2}Q_1\right)\!|_{t=\delta_2}\right)\!|_{t=\delta_3}  \\ - (\widetilde{w}_0-\delta_3^{-1/2}q_1-(\Upsilon(q_1)\cdot\nabla)q_1, \widetilde{\Theta})\|_{H^{m}\times H^{m+1}} < \varepsilon_1,
	\end{multline*}
	denoting
	\begin{multline*}
		\widetilde{S}_{\delta_2}(A,B) \coloneq S_{\delta_2}\left(A,B,\mathbb{I}_{[\delta_{2,1},\delta_{2,2}]}H_{1,\delta_{2,2}}(\cdot,\cdot-\delta_{2,1}), \right. \\ \left. \delta_{2,2}^{-2}\mathbb{I}_{[\delta_{2,1},\delta_{2,2}]}\zeta^0(\cdot, \delta_{2,2}^{-1}(\cdot - \delta_{2,1})), \mathbb{I}_{[\delta_{2,1},\delta_{2,2}]}A_{\delta_{2,2}}(\cdot-\delta_{2,1})\right).
	\end{multline*}
	iii) By an application of \Cref{theorem:msltc} with target temperature $\widetilde{\theta}_0 - \delta_{2,1}^{-1}\delta_3^{-1/2}Q_1$, we fix $\zeta^1 \in L^2((0, 1); \mathscr{F}_{\mathscr{t}})$ and $\delta_1 < (T-T_0)/(10N+1)$ so small that
	\begin{multline*}
		\|S_{\delta_3}\left(\widetilde{S}_{\delta_2}\left(S_{\delta_1}\left(\widetilde{w}_0, \widetilde{\theta}_0, H_{1,\delta_1}, \delta_{1}^{-2}\zeta^1(\cdot, \delta_{1}^{-1}\cdot), A_{\delta_1}\right)\!|_{t=\delta_1}\right)\!|_{t=\delta_2}\right)\!|_{t=\delta_3} \\ - (\widetilde{w}_0-\delta_3^{-1/2}q_1-(\Upsilon(q_1)\cdot\nabla)q_1, \widetilde{\Theta})\|_{H^{m}\times H^{m+1}} < \varepsilon_1.
	\end{multline*}
	
	Now, for any given $\overline{\varepsilon} > 0$, we select the number $\varepsilon_1 > 0$ used in the steps above (hence, we fix a choice of $\delta_1 = \delta_1(\delta_2, \delta_3)$, $\delta_2 = \delta_2(\delta_3)$, and $\delta_3$) and determine $0 < \delta_5 < (T-T_0)/(10N+1)$ via \eqref{equation:lsc1} such that
	\begin{multline*}
		\|\Pi_1 S_{\delta_5}\left(\widetilde{w}_0-\delta_3^{-1/2}q_1-(\Upsilon(q_1)\cdot\nabla)q_1, \delta_5^{-1}(\delta_3^{-1/2} Q_1-Q_0)\right)\!|_{t=\delta_5} \\- (\widetilde{w}_0 - q_0 - (\Upsilon(q_1)\cdot\nabla)q_1)\|_{H^{m}\times H^{m+1}} < \overline{\varepsilon}.
	\end{multline*}

	iv) By applying again \Cref{theorem:msltc}, we select $0 < \delta_4 < (T-T_0)/(10N+1)$ and $\zeta^2 \in L^2((0, 1); \mathscr{F}_{\mathscr{t}})$ such that
	\begin{multline*}
		\|\Pi_1 S_{\delta_5}\left(S_{\delta_4}\left(S_3, H_{1,\delta_4}, + \delta_{4}^{-2}\zeta^2(\cdot, \delta_{4}^{-1}\cdot), A_{\delta_4}\right)\!|_{t=\delta_4}\right)\!|_{t=\delta_5} \\- (\widetilde{w}_0 - q_0 - (\Upsilon(q_1)\cdot\nabla)q_1, \widetilde{\Theta})\|_{H^{m}\times H^{m+1}} < \overline{\varepsilon},
	\end{multline*}
	where
	\[
	S_3 \coloneq S_{\delta_3}\left(\widetilde{S}_{\delta_2}\left(S_{\delta_1}\left(\widetilde{w}_0, \widetilde{\theta}_0, H_{1,\delta_1}, \delta_{1}^{-2}\zeta^1(\cdot, \delta_{1}^{-1}\cdot), A_{\delta_1}\right)\!|_{t=\delta_1} \right)\!|_{t=\delta_2}\right)\!|_{t=\delta_3}.
	\]
	The so-obtained controls for the velocity and temperature are zero on the union of time intervals
	\begin{equation*}\label{equation:ti1}
		\begin{gathered}
			\left[T_0+\delta_1, T_0+\delta_1+\delta_{2,1}\right],\\ \left[T_0+\delta_1+\delta_{2}, T_0 + \sum_{l=1}^3\delta_l\right], \quad \left[T_0 + \sum_{l=1}^4\delta_l, T_0 + \sum_{l=1}^5\delta_l\right],
		\end{gathered}
	\end{equation*}
	while assuming possibly nonzero values in the respective spaces $\mathscr{F}_{\mathscr{v}}$ and $\mathscr{F}_{\mathscr{t}}$ on the union of the time intervals
	\begin{equation*}\label{equation:ti2}
		\begin{gathered}
			\left[T_0, T_0+\delta_1\right], \quad [T_0+\delta_1+\delta_{2,1}, T_0+\delta_1+\delta_2], \quad \left[T_0 + \sum_{l=1}^3\delta_l, T_0 + \sum_{l=1}^4\delta_l\right].
		\end{gathered}
	\end{equation*}

	\begin{rmrk}\label{remark:nza2}
		The proof of \Cref{theorem:main2} extends to initial- and target states with nonzero average. Hereto, one has to add two short average control stages at the beginning and at the end of \enquote{Step 3. Control strategy}. To illustrate this, let us observe that the velocity and temperature averages of solutions to~\eqref{equation:Boussinesq} (with body forces of zero average) behave formally like
		\begin{gather*}
			\int_{\mathbb{T}^2} u(x,t) \, dx = \int_{\mathbb{T}^2} u_0(x) \, dx + \int_0^t \int_{\mathbb{T}^2} \theta(x,s) e_2 \, dx ds + \int_0^t \int_{\omegaup} \xi(x,s)\, dx ds,\\
			\int_{\mathbb{T}^2} \theta(x,t) \, dx = \int_{\mathbb{T}^2} \theta_0(x) \, dx + \int_0^t \int_{\omegaup} \eta(x,s) \, dx ds
		\end{gather*}
		for $t \in [0, T]$. 
		
		To begin with, suppose that $u_0$ and $u_1$ are of average $A_0 = (A_{0,1}, A_{0,2}) \in \mathbb{R}^2$ and $A_1 = (A_{1,1}, A_{1,2}) \in \mathbb{R}^2$, respectively. Further, assume that the averages of~$\theta_0$ and~$\theta_1$ are~$\tau_0 \in \mathbb{R}$ and~$\tau_1 \in \mathbb{R}$, respectively.
		Then, fix two vector fields $\mathscr{a}, \mathscr{b} \in C^{\infty}(\mathbb{T}^2; \mathbb{R}^2)$ with $\operatorname{supp}(\mathscr{a}) \cup \operatorname{supp}(\mathscr{b}) \subset \omegaup$ and
		\[
			\int_{\mathbb{T}^2} \mathscr{a}(x) \, dx = (1, 0), \quad
			\int_{\mathbb{T}^2} \mathscr{b}(x) \, dx = (0, 1).
		\]
		Moreover, choose profiles $\lambda_{0,1}, \lambda_{0,2}, \lambda_{1,1}, \lambda_{1,2}, r_0, r_1 \in C^{\infty}_0((0,1);\mathbb{R})$ such that
		\begin{gather*}
			\int_0^1 \lambda_{0,1}(s) \, ds = -A_{0,1}, \quad \int_0^1 \lambda_{1,1}(s) \, ds = A_{1,1}, \\
			\int_0^1 r_0(s) \, ds = -\tau_0, \quad \int_0^1 \lambda_{0,2}(s) \, ds + \int_0^1 \int_0^t r_0(s) \, ds dt = -A_{0,2} - \tau_0, \\ 
			\int_0^1 r_1(s) \, ds = \tau_1, \quad \int_0^1 \lambda_{1,2}(s) \, ds + \int_0^1 \int_0^t r_1(s) \, ds dt = A_{1,2}
		\end{gather*}
		
		Now, the above proof of \Cref{theorem:main2} is adapted as described in the following points. 
		
		1) Since the arguments in \enquote{Step 1. Stability} and \enquote{Step 2. Regularization} likewise work for initial data with nonzero average, no significant changes are made there: the velocity and temperature equations with zero average body forces and zero controls preserve the averages of the initial data until the time $T-\delta_0$. 
		
		2) Before starting with \enquote{Step 3. Control strategy}, we apply \Cref{theorem:msltc} with $b = \lambda_{0,1} \mathscr{a} + \lambda_{0,2} \mathscr{b}$, $A = A_0$, $\tau = r_0$, and zero target temperature. 
		Notably, the velocity and temperature averages are steered exactly to zero by this preliminary application of \Cref{theorem:msltc}; see also \Cref{remark:tempav}. 
		
		3) The original target vorticity $w_1$ in \enquote{Step 3. Control strategy} is replaced by $w_{A_1} = \widetilde{v}_b(\cdot, 0)$, where $\widetilde{v}_b$ solves backwards in time the problem \eqref{equation:vb} with prescribed endpoint $v^b(\cdot, 1) = w_1$ and $b = \lambda_{1,1} \mathscr{a} + \lambda_{1,2} \mathscr{b}$. 
		
		4) At the end of \enquote{Step 3}, we insert another application \Cref{theorem:msltc}, now with $b = \lambda_{1,1} \mathscr{a} + \lambda_{1,2} \mathscr{b}$, $A = 0$, $\tau = r_1$, zero initial temperature average, and desired target temperature.
	\end{rmrk}

	\subsubsection*{Acknowledgements}
	The author would like to thank Vahagn Nersesyan for insightful discussions and for suggesting an improved description of the control spaces.
	
	\subsubsection*{Data availability statement}
	Data sharing is not applicable to this article as no datasets were generated or analyzed during the current study.
	
	\subsubsection*{Declarations}
	The author declares that he has no conflicts of interest.

	\bibliographystyle{alpha}
	\bibliography{Bib}

\begin{bibdiv}
\begin{biblist}

\bib{Agrachev2014}{incollection}{
      author={Agrachev, Andrei~A.},
       title={Some open problems},
        date={2014},
   booktitle={Geometric control theory and sub-{R}iemannian geometry},
      series={Springer INdAM Ser.},
      volume={5},
   publisher={Springer, Cham},
       pages={1\ndash 13},
         url={https://doi.org/10.1007/978-3-319-02132-4_1},
}

\bib{AgrachevSarychev2005}{article}{
      author={Agrachev, Andrey~A.},
      author={Sarychev, Andrey~V.},
       title={Navier-{S}tokes equations: controllability by means of low modes
  forcing},
        date={2005},
        ISSN={1422-6928},
     journal={J. Math. Fluid Mech.},
      volume={7},
      number={1},
       pages={108\ndash 152},
         url={https://doi.org/10.1007/s00021-004-0110-1},
}

\bib{AgrachevSarychev2006}{article}{
      author={Agrachev, Andrey~A.},
      author={Sarychev, Andrey~V.},
       title={Controllability of 2{D} {E}uler and {N}avier-{S}tokes equations
  by degenerate forcing},
        date={2006},
        ISSN={0010-3616},
     journal={Comm. Math. Phys.},
      volume={265},
      number={3},
       pages={673\ndash 697},
         url={https://doi.org/10.1007/s00220-006-0002-8},
}

\bib{AgrachevSarychev2008}{incollection}{
      author={Agrachev, Andrey~A.},
      author={Sarychev, Andrey~V.},
       title={Solid controllability in fluid dynamics},
        date={2008},
   booktitle={Instability in models connected with fluid flows. {I}},
      series={Int. Math. Ser. (N. Y.)},
      volume={6},
   publisher={Springer, New York},
       pages={1\ndash 35},
         url={https://doi.org/10.1007/978-0-387-75217-4_1},
}

\bib{BoulvardGaoNersesyan2023}{article}{
      author={Boulvard, Pierre-Marie},
      author={Gao, Peng},
      author={Nersesyan, Vahagn},
       title={Controllability and ergodicity of three dimensional primitive
  equations driven by a finite-dimensional force},
        date={2023},
        ISSN={0003-9527,1432-0673},
     journal={Arch. Ration. Mech. Anal.},
      volume={247},
      number={1},
       pages={Paper No. 2, 49},
         url={https://doi.org/10.1007/s00205-022-01835-8},
}

\bib{Chaves-SilvaEtal2023}{article}{
      author={Chaves-Silva, F.~W.},
      author={Fern\'{a}ndez-Cara, E.},
      author={Le~Balc'h, K.},
      author={Machado, J. L.~F.},
      author={Souza, D.~A.},
       title={Global controllability of the {B}oussinesq system with
  {N}avier-slip-with-friction and {R}obin boundary conditions},
        date={2023},
        ISSN={0363-0129,1095-7138},
     journal={SIAM J. Control Optim.},
      volume={61},
      number={2},
       pages={484\ndash 510},
         url={https://doi.org/10.1137/21M1425566},
}

\bib{ConstantinDoering1999}{article}{
      author={Constantin, Peter},
      author={Doering, Charles~R.},
       title={Infinite {P}randtl number convection},
        date={1999},
        ISSN={0022-4715,1572-9613},
     journal={J. Statist. Phys.},
      volume={94},
      number={1-2},
       pages={159\ndash 172},
         url={https://doi.org/10.1023/A:1004511312885},
}

\bib{Coron96}{article}{
      author={Coron, Jean-Michel},
       title={On the controllability of the {$2$}-{D} incompressible
  {N}avier-{S}tokes equations with the {N}avier slip boundary conditions},
        date={1995/96},
        ISSN={1292-8119},
     journal={ESAIM Contr\^{o}le Optim. Calc. Var.},
      volume={1},
       pages={35\ndash 75},
         url={https://doi.org/10.1051/cocv:1996102},
}

\bib{Coron1996EulerEq}{article}{
      author={Coron, Jean-Michel},
       title={On the controllability of {$2$}-{D} incompressible perfect
  fluids},
        date={1996},
        ISSN={0021-7824},
     journal={J. Math. Pures Appl.},
      volume={75},
      number={2},
       pages={155\ndash 188},
}

\bib{Coron2007}{book}{
      author={Coron, Jean-Michel},
       title={Control and nonlinearity},
      series={Mathematical Surveys and Monographs},
   publisher={American Mathematical Society, Providence, RI},
        date={2007},
      volume={136},
        ISBN={978-0-8218-3668-2; 0-8218-3668-4},
}

\bib{CoronFursikov1996}{article}{
      author={Coron, Jean-Michel},
      author={Fursikov, Andrei~V.},
       title={Global exact controllability of the {$2$D} {N}avier-{S}tokes
  equations on a manifold without boundary},
        date={1996},
        ISSN={1061-9208},
     journal={Russian J. Math. Phys.},
      volume={4},
      number={4},
       pages={429\ndash 448},
}

\bib{CoronLissy2014}{article}{
      author={Coron, Jean-Michel},
      author={Lissy, Pierre},
       title={Local null controllability of the three-dimensional
  {N}avier-{S}tokes system with a distributed control having two vanishing
  components},
        date={2014},
        ISSN={0020-9910,1432-1297},
     journal={Invent. Math.},
      volume={198},
      number={3},
       pages={833\ndash 880},
         url={https://doi.org/10.1007/s00222-014-0512-5},
}

\bib{CoronMarbachSueur2020}{article}{
      author={Coron, Jean-Michel},
      author={Marbach, Fr\'{e}d\'{e}ric},
      author={Sueur, Franck},
       title={Small-time global exact controllability of the {N}avier-{S}tokes
  equation with {N}avier slip-with-friction boundary conditions},
        date={2020},
        ISSN={1435-9855},
     journal={J. Eur. Math. Soc. (JEMS)},
      volume={22},
      number={5},
       pages={1625\ndash 1673},
         url={https://doi.org/10.4171/jems/952},
}

\bib{CoronMarbachSueurZhang2019}{article}{
      author={Coron, Jean-Michel},
      author={Marbach, Fr\'{e}d\'{e}ric},
      author={Sueur, Franck},
      author={Zhang, Ping},
       title={Controllability of the {N}avier-{S}tokes equation in a rectangle
  with a little help of a distributed phantom force},
        date={2019},
        ISSN={2524-5317},
     journal={Ann. PDE},
      volume={5},
      number={2},
       pages={Art. 17},
}

\bib{CoronShengquan2021}{article}{
      author={Coron, Jean-Michel},
      author={Xiang, Shengquan},
       title={Small-time global stabilization of the viscous {B}urgers equation
  with three scalar controls},
        date={2021},
        ISSN={0021-7824,1776-3371},
     journal={J. Math. Pures Appl. (9)},
      volume={151},
       pages={212\ndash 256},
         url={https://doi.org/10.1016/j.matpur.2021.03.001},
}

\bib{CoronXiangZhang2023}{article}{
      author={Coron, Jean-Michel},
      author={Xiang, Shengquan},
      author={Zhang, Ping},
       title={On the global approximate controllability in small time of
  semiclassical 1-{D} {S}chr\"odinger equations between two states with
  positive quantum densities},
        date={2023},
        ISSN={0022-0396,1090-2732},
     journal={J. Differential Equations},
      volume={345},
       pages={1\ndash 44},
         url={https://doi.org/10.1016/j.jde.2022.11.021},
}

\bib{DucaNersesyan2025}{article}{
      author={Duca, Alessandro},
      author={Nersesyan, Vahagn},
       title={Bilinear control and growth of {S}obolev norms for the nonlinear
  {S}chr\"odinger equation},
        date={2025},
        ISSN={1435-9855,1435-9863},
     journal={J. Eur. Math. Soc. (JEMS)},
      volume={27},
      number={6},
       pages={2603\ndash 2622},
         url={https://doi.org/10.4171/jems/1420},
}

\bib{ErvedozaEtal2022}{article}{
      author={Ervedoza, Sylvain},
      author={Le~Balc'h, K\'evin},
      author={Tucsnak, Marius},
       title={Reachability results for perturbed heat equations},
        date={2022},
        ISSN={0022-1236,1096-0783},
     journal={J. Funct. Anal.},
      volume={283},
      number={10},
       pages={Paper No. 109666, 61},
         url={https://doi.org/10.1016/j.jfa.2022.109666},
}

\bib{Fernandez-CaraGuerreroImanuvilovPuel2004}{article}{
      author={Fern\'andez-Cara, E.},
      author={Guerrero, S.},
      author={Imanuvilov, O.~Yu.},
      author={Puel, J.-P.},
       title={Local exact controllability of the {N}avier-{S}tokes system},
        date={2004},
        ISSN={0021-7824},
     journal={J. Math. Pures Appl. (9)},
      volume={83},
      number={12},
       pages={1501\ndash 1542},
         url={https://doi.org/10.1016/j.matpur.2004.02.010},
}

\bib{Fernandez-CaraGuerreroImanuvilovPuel2006}{article}{
      author={Fern\'{a}ndez-Cara, Enrique},
      author={Guerrero, Sergio},
      author={Imanuvilov, Oleg~Yu.},
      author={Puel, Jean-Pierre},
       title={Some controllability results for the {$N$}-dimensional
  {N}avier-{S}tokes and {B}oussinesq systems with {$N-1$} scalar controls},
        date={2006},
        ISSN={0363-0129,1095-7138},
     journal={SIAM J. Control Optim.},
      volume={45},
      number={1},
       pages={146\ndash 173},
         url={https://doi.org/10.1137/04061965X},
}

\bib{FoiasManleyTemam1987}{article}{
      author={Foias, C.},
      author={Manley, O.},
      author={Temam, R.},
       title={Attractors for the {B{\'e}nard} problem: {Existence} and physical
  bounds on their fractal dimension},
    language={English},
        date={1987},
        ISSN={0362-546X},
     journal={Nonlinear Anal., Theory Methods Appl.},
      volume={11},
       pages={939\ndash 967},
}

\bib{FoldersHerzog2024}{article}{
      author={F\"oldes, Juraj},
      author={Herzog, David~P.},
       title={Small-time asymptotics for hypoelliptic diffusions},
        date={2024},
     journal={arXiv preprint},
      eprint={https://doi.org/10.48550/arXiv.2412.11323},
}

\bib{FursikovImanuvilov1998}{article}{
      author={Fursikov, A.~V.},
      author={Imanuvilov, O.~Yu.},
       title={Local exact boundary controllability of the {B}oussinesq
  equation},
        date={1998},
        ISSN={0363-0129,1095-7138},
     journal={SIAM J. Control Optim.},
      volume={36},
      number={2},
       pages={391\ndash 421},
         url={https://doi.org/10.1137/S0363012996296796},
}

\bib{FursikovImanuvilov1999}{article}{
      author={Fursikov, A.~V.},
      author={Imanuvilov, O.~Yu.},
       title={Exact controllability of the {N}avier-{S}tokes and {B}oussinesq
  equations},
        date={1999},
        ISSN={0042-1316,2305-2872},
     journal={Uspekhi Mat. Nauk},
      volume={54},
      number={3(327)},
       pages={93\ndash 146},
         url={https://doi.org/10.1070/rm1999v054n03ABEH000153},
}

\bib{Glass2000}{article}{
      author={Glass, Olivier},
       title={Exact boundary controllability of 3-{D} {E}uler equation},
        date={2000},
        ISSN={1292-8119},
     journal={ESAIM Control Optim. Calc. Var.},
      volume={5},
       pages={1\ndash 44},
         url={https://doi.org/10.1051/cocv:2000100},
}

\bib{Guerrero2006}{article}{
      author={Guerrero, S.},
       title={Local exact controllability to the trajectories of the
  {B}oussinesq system},
        date={2006},
        ISSN={0294-1449,1873-1430},
     journal={Ann. Inst. H. Poincar\'{e} C Anal. Non Lin\'{e}aire},
      volume={23},
      number={1},
       pages={29\ndash 61},
         url={https://doi.org/10.1016/j.anihpc.2005.01.002},
}

\bib{HairerMattingly2006}{article}{
      author={Hairer, Martin},
      author={Mattingly, Jonathan~C.},
       title={Ergodicity of the 2{D} {N}avier-{S}tokes equations with
  degenerate stochastic forcing},
        date={2006},
        ISSN={0003-486X,1939-8980},
     journal={Ann. of Math. (2)},
      volume={164},
      number={3},
       pages={993\ndash 1032},
         url={https://doi.org/10.4007/annals.2006.164.993},
}

\bib{KukavicaNovackVicol2022}{article}{
      author={Kukavica, Igor},
      author={Novack, Matthew},
      author={Vicol, Vlad},
       title={Exact boundary controllability for the ideal magneto-hydrodynamic
  equations},
        date={2022},
        ISSN={0022-0396},
     journal={J. Differential Equations},
      volume={318},
       pages={94\ndash 112},
}

\bib{KuksinNersesyanShirikyan2020}{article}{
      author={Kuksin, Sergei},
      author={Nersesyan, Vahagn},
      author={Shirikyan, Armen},
       title={Exponential mixing for a class of dissipative {PDE}s with bounded
  degenerate noise},
        date={2020},
        ISSN={1016-443X},
     journal={Geom. Funct. Anal.},
      volume={30},
      number={1},
       pages={126\ndash 187},
         url={https://doi.org/10.1007/s00039-020-00525-5},
}

\bib{KuksinShirikyan2012}{book}{
      author={Kuksin, Sergei},
      author={Shirikyan, Armen},
       title={Mathematics of two-dimensional turbulence},
      series={Cambridge Tracts in Mathematics},
   publisher={Cambridge University Press, Cambridge},
        date={2012},
      volume={194},
        ISBN={978-1-107-02282-9},
         url={https://doi.org/10.1017/CBO9781139137119},
      review={\MR{3443633}},
}

\bib{Liao2024}{article}{
      author={Liao, Jiajiang},
       title={Global exact controllability of the viscous and resistive {MHD}
  system in a rectangle thanks to the lateral sides and to distributed phantom
  forces},
        date={2024},
        ISSN={1292-8119,1262-3377},
     journal={ESAIM Control Optim. Calc. Var.},
      volume={30},
       pages={Paper No. 67, 40},
         url={https://doi.org/10.1051/cocv/2023078},
}

\bib{LiaoSueurZhang2022b}{article}{
      author={Liao, Jiajiang},
      author={Sueur, Franck},
      author={Zhang, Ping},
       title={Global controllability of the {N}avier-{S}tokes equations in the
  presence of curved boundary with no-slip conditions},
        date={2022},
        ISSN={1422-6928,1422-6952},
     journal={J. Math. Fluid Mech.},
      volume={24},
      number={3},
       pages={Paper No. 71, 32},
         url={https://doi.org/10.1007/s00021-022-00689-0},
}

\bib{LiaoSueurZhang2022}{article}{
      author={Liao, Jiajiang},
      author={Sueur, Franck},
      author={Zhang, Ping},
       title={Smooth controllability of the {N}avier-{S}tokes equation with
  {N}avier conditions: application to {L}agrangian controllability},
        date={2022},
        ISSN={0003-9527,1432-0673},
     journal={Arch. Ration. Mech. Anal.},
      volume={243},
      number={2},
       pages={869\ndash 941},
         url={https://doi.org/10.1007/s00205-021-01744-2},
}

\bib{LionsJL1991}{incollection}{
      author={Lions, Jacques-Louis},
       title={Exact controllability for distributed systems. {S}ome trends and
  some problems},
        date={1991},
   booktitle={Applied and industrial mathematics ({V}enice, 1989)},
      series={Math. Appl.},
      volume={56},
   publisher={Kluwer Acad. Publ., Dordrecht},
       pages={59\ndash 84},
}

\bib{LionsZuazua1998}{article}{
      author={Lions, Jacques-Louis},
      author={Zuazua, Enrique},
       title={Exact boundary controllability of {G}alerkin's approximations of
  {N}avier-{S}tokes equations},
        date={1998},
        ISSN={0391-173X},
     journal={Ann. Scuola Norm. Sup. Pisa Cl. Sci. (4)},
      volume={26},
      number={4},
       pages={605\ndash 621},
}

\bib{Majda2003}{book}{
      author={Majda, Andrew},
       title={Introduction to {PDE}s and waves for the atmosphere and ocean},
      series={Courant Lecture Notes in Mathematics},
   publisher={New York University, Courant Institute of Mathematical Sciences,
  New York; American Mathematical Society, Providence, RI},
        date={2003},
      volume={9},
        ISBN={0-8218-2954-8},
         url={https://doi.org/10.1090/cln/009},
}

\bib{MattinglyPardoux2006}{article}{
      author={Mattingly, Jonathan~C.},
      author={Pardoux, \'Etienne},
       title={Malliavin calculus for the stochastic 2{D} {N}avier-{S}tokes
  equation},
        date={2006},
        ISSN={0010-3640,1097-0312},
     journal={Comm. Pure Appl. Math.},
      volume={59},
      number={12},
       pages={1742\ndash 1790},
         url={https://doi.org/10.1002/cpa.20136},
}

\bib{Nersesyan2015}{article}{
      author={Nersesyan, Vahagn},
       title={Approximate controllability of {L}agrangian trajectories of the
  3{D} {N}avier-{S}tokes system by a finite-dimensional force},
        date={2015},
        ISSN={0951-7715,1361-6544},
     journal={Nonlinearity},
      volume={28},
      number={3},
       pages={825\ndash 848},
         url={https://doi.org/10.1088/0951-7715/28/3/825},
}

\bib{Nersesyan2021}{article}{
      author={Nersesyan, Vahagn},
       title={A proof of approximate controllability of the 3{D}
  {N}avier-{S}tokes system via a linear test},
        date={2021},
        ISSN={0363-0129},
     journal={SIAM J. Control Optim.},
      volume={59},
      number={4},
       pages={2411\ndash 2427},
         url={https://doi.org/10.1137/20M1368689},
}

\bib{Nersesyan2024}{article}{
      author={Nersesyan, Vahagn},
       title={The complex {G}inzburg-{L}andau equation perturbed by a force
  localised both in physical and {F}ourier spaces},
        date={2024},
        ISSN={0391-173X,2036-2145},
     journal={Ann. Sc. Norm. Super. Pisa Cl. Sci. (5)},
      volume={25},
      number={2},
       pages={1203\ndash 1223},
}

\bib{NersesyanRissel2024}{article}{
      author={Nersesyan, Vahagn},
      author={Rissel, Manuel},
       title={Global {C}ontrollability of {B}oussinesq {F}lows by {U}sing
  {O}nly a {T}emperature {C}ontrol},
        date={2025},
        ISSN={0003-9527,1432-0673},
     journal={Arch. Ration. Mech. Anal.},
      volume={249},
      number={5},
       pages={Paper No. 58},
}

\bib{NersesyanRissel2024a}{article}{
      author={Nersesyan, Vahagn},
      author={Rissel, Manuel},
       title={Localized and degenerate controls for the incompressible
  {N}avier-{S}tokes system},
        date={2025},
     journal={Comm. Pure Appl. Math.},
      volume={78},
      number={7},
       pages={1285\ndash 1319},
}

\bib{Nersisyan2011}{article}{
      author={Nersisyan, Hayk},
       title={Controllability of the 3{D} compressible {E}uler system},
        date={2011},
        ISSN={0360-5302,1532-4133},
     journal={Comm. Partial Differential Equations},
      volume={36},
      number={9},
       pages={1544\ndash 1564},
         url={https://doi.org/10.1080/03605302.2011.596605},
}

\bib{PhanRodrigues2019}{article}{
      author={Phan, Duy},
      author={Rodrigues, S\'ergio~S.},
       title={Approximate controllability for {N}avier-{S}tokes equations in
  3{D} rectangles under {L}ions boundary conditions},
        date={2019},
        ISSN={1079-2724,1573-8698},
     journal={J. Dyn. Control Syst.},
      volume={25},
      number={3},
       pages={351\ndash 376},
         url={https://doi.org/10.1007/s10883-018-9412-0},
}

\bib{Rissel2024}{article}{
      author={Rissel, Manuel},
       title={Global controllability of {B}oussinesq channel flows only through
  the temperature},
        date={2024},
     journal={arXiv preprint},
      eprint={https://doi.org/10.48550/arXiv.2411.02200},
}

\bib{RisselWang2024}{article}{
      author={Rissel, Manuel},
      author={Wang, Ya-Guang},
       title={Small-time global approximate controllability for incompressible
  {MHD} with coupled {N}avier slip boundary conditions},
        date={2024},
        ISSN={0021-7824,1776-3371},
     journal={J. Math. Pures Appl. (9)},
      volume={190},
       pages={Paper No. 103601, 66},
         url={https://doi.org/10.1016/j.matpur.2024.103601},
}

\bib{Rodrigues2006}{article}{
      author={Rodrigues, S.~S.},
       title={Navier-{S}tokes equation on the rectangle controllability by
  means of low mode forcing},
        date={2006},
        ISSN={1079-2724,1573-8698},
     journal={J. Dyn. Control Syst.},
      volume={12},
      number={4},
       pages={517\ndash 562},
         url={https://doi.org/10.1007/s10883-006-0004-z},
}

\bib{Rodrigues2007}{incollection}{
      author={Rodrigues, S\'ergio~S.},
       title={Navier-{S}tokes equation on a plane bounded domain: continuity
  properties for controllability},
        date={2007},
   booktitle={Taming heterogeneity and complexity of embedded control},
   publisher={ISTE, London},
       pages={585\ndash 615},
}

\bib{Sarychev2012}{article}{
      author={Sarychev, Andrey},
       title={Controllability of the cubic {S}chroedinger equation via a
  low-dimensional source term},
        date={2012},
        ISSN={2156-8472,2156-8499},
     journal={Math. Control Relat. Fields},
      volume={2},
      number={3},
       pages={247\ndash 270},
         url={https://doi.org/10.3934/mcrf.2012.2.247},
}

\bib{Shirikyan2006}{article}{
      author={Shirikyan, Armen},
       title={Approximate controllability of three-dimensional
  {N}avier-{S}tokes equations},
        date={2006},
        ISSN={0010-3616},
     journal={Comm. Math. Phys.},
      volume={266},
      number={1},
       pages={123\ndash 151},
         url={https://doi.org/10.1007/s00220-006-0007-3},
}

\bib{Shirikyan2007}{article}{
      author={Shirikyan, Armen},
       title={Exact controllability in projections for three-dimensional
  {N}avier--{S}tokes equations},
        date={2007},
        ISSN={0294-1449},
     journal={Ann. Inst. H. Poincar\'{e} C Anal. Non Lin\'{e}aire},
      volume={24},
      number={4},
       pages={521\ndash 537},
         url={https://doi.org/10.1016/j.anihpc.2006.04.002},
}

\bib{Shirikyan2008}{article}{
      author={Shirikyan, Armen},
       title={Euler equations are not exactly controllable by a
  finite-dimensional external force},
        date={2008},
        ISSN={0167-2789,1872-8022},
     journal={Phys. D},
      volume={237},
      number={10-12},
       pages={1317\ndash 1323},
         url={https://doi.org/10.1016/j.physd.2008.03.021},
}

\bib{Shirikyan2018}{article}{
      author={Shirikyan, Armen},
       title={Control theory for the {B}urgers equation: {A}grachev-{S}arychev
  approach},
        date={2018},
        ISSN={2189-3756,2189-3764},
     journal={Pure Appl. Funct. Anal.},
      volume={3},
      number={1},
       pages={219\ndash 240},
}

\bib{Temam1982}{article}{
      author={Temam, R.},
       title={Behaviour at time {$t=0$} of the solutions of semilinear
  evolution equations},
        date={1982},
        ISSN={0022-0396},
     journal={J. Differential Equations},
      volume={43},
      number={1},
       pages={73\ndash 92},
}

\bib{Temam1997}{book}{
      author={Temam, Roger},
       title={Infinite-dimensional dynamical systems in mechanics and
  physics.},
    language={English},
     edition={2nd ed.},
      series={Appl. Math. Sci.},
   publisher={New York, NY: Springer},
        date={1997},
      volume={68},
        ISBN={0-387-94866-X},
}

\bib{Temam2001}{book}{
      author={Temam, Roger},
       title={Navier-{S}tokes equations},
   publisher={AMS Chelsea Publishing, Providence, RI},
        date={2001},
        ISBN={0-8218-2737-5},
         url={https://doi.org/10.1090/chel/343},
        note={Theory and numerical analysis, Reprint of the 1984 edition},
}

\bib{Shengquan2023}{article}{
      author={Xiang, Shengquan},
       title={Small-time local stabilization of the two-dimensional
  incompressible {N}avier-{S}tokes equations},
        date={2023},
        ISSN={0294-1449,1873-1430},
     journal={Ann. Inst. H. Poincar\'e{} C Anal. Non Lin\'eaire},
      volume={40},
      number={6},
       pages={1487\ndash 1511},
         url={https://doi.org/10.4171/aihpc/70},
}

\end{biblist}
\end{bibdiv}

\end{document}